\documentclass[a4paper,11pt,reqno]{amsart}
\usepackage{amsmath,mathtools,amsthm,amsfonts,amssymb,latexsym,bm}
\usepackage{graphicx}
\usepackage{mathrsfs}
\usepackage{natbib}
\usepackage{soul,color}
\usepackage{pstricks}
\usepackage[breaklinks=true]{hyperref}
\usepackage{breakcites}

\usepackage[left=3.00cm,right=3.00cm]{geometry}  
\newtheorem{theorem}{Theorem}[section]
\newtheorem{lemma}[theorem]{Lemma}
\newtheorem{corollary}[theorem]{Corollary}

\theoremstyle{definition}
\newtheorem{definition}[theorem]{Definition}
\newtheorem{example}[theorem]{Example}

\newtheorem{condition}{Condition}

\theoremstyle{remark}
\newtheorem{remark}[theorem]{Remark}
{\theoremstyle{plain}}

\numberwithin{equation}{section}

\newcommand{\Balg}{\mathcal{B}}

\newcommand{\dd}{{\rm d}}
\newcommand{\E}{{\rm E}}
\newcommand{\eps}{\varepsilon}
\newcommand{\Falg}{\mathcal{F}}
\newcommand{\FF}{\mathbb{F}}
\newcommand{\Galg}{\mathcal{G}}
\newcommand{\G}{\mathcal{G}} 
\def\JS{jacod2003limit}

\newcommand{\pr}{{\rm Pr}}

\newcommand{\qv}{{\rm QV}}
\newcommand{\tsqc}{{\rm TSQC}}

\newcommand{\abs}[1]{\lvert#1\lvert} 
\newcommand{\norm}[1]{\lVert#1\rVert} 

\setstcolor{red}

\begin{document}

\title[A CLT for second difference estimators]{A CLT for second difference estimators\\
with an application to volatility and intensity}

\author[Stoltenberg]{Emil A.~Stoltenberg}
\address{The University of Chicago and BI Norwegian Business School}
\curraddr{Sofies gate 75A, 0454 Oslo, Norway}
\email{emilstoltenberg@gmail.com}
\thanks{
Emil A. Stoltenberg (emilstoltenberg@gmail.com) would like to thank the Fulbright Foundation for financial support, The University of Chicago for their hospitality, Sylvie Bendier Decety for her kindness, and the PharmaTox Strategic Research Initiative at the Faculty of Mathematics and Natural Sciences, University of Oslo. The authors would also like to thank the United States National Science Foundation under grants DMS 17-13118 and DMS-2015530 (Zhang), and DMS 17-13129 and DMS-2015544 (Mykland). An earlier version of this paper was called {``}Volatility and intensity{''} \url{https://arxiv.org/abs/1903.09873}}
\author[Mykland]{Per A.~Mykland}
\address{Dept.~of Statistics, The University of Chicago}
\curraddr{5747 South Ellis Avenue, Chicago, IL 60637, USA}
\email{mykland@pascal.uchicago.edu}

\author[Zhang]{Lan Zhang}
\address{Dept.~of Finance, University of Illinois at Chicago}
\curraddr{601 South Morgan Street, mc 168,
Chicago, IL 60607, U.S.A}
\email{lanzhang@uic.edu}



\date{\today}


\begin{abstract} In this paper we introduce a general method for estimating the quadratic covariation of one or more spot parameters processes associated with continuous time semimartingales. This estimator is applicable to a wide range of spot parameter processes, and may also be used to estimate the leverage effect of stochastic volatility models. The estimator we introduce is based on sums of squared increments of second differences of the observed process, and the intervals over which the differences are computed are rolling and overlapping. This latter feature lets us take full advantage of the data, and, by sufficiency considerations, ought to outperform estimators that are only based on one partition of the observational window. The main result of the paper is a central limit theorem for such triangular array rolling quadratic variations. We highlight the wide applicability of this theorem by showcasing how it might be applied to a novel leverage effect estimator. The principal motivation for the present study, however, is that the discrete times at which a continuous time semimartingale is observed might depend on features of the observable process other than its level, such as its (non-observable) spot-volatility process. As the main application of our estimator, we therefore show how it may be used to estimate the quadratic covariation between the spot-volatility process and the intensity process of the observation times, when both of these are taken to be semimartingales. The finite sample properties of this estimator are studied by way of a simulation experiment, and we also apply this estimator in an empirical analysis of the Apple stock. Our analysis of the Apple stock indicates a rather strong correlation between the spot volatility process of the log-prices process and the times at which this stock is traded (hence observed). \end{abstract}

\maketitle

\section{Introduction}\label{sec::intro}
With an increasing availability of high frequency data, the ambition level as to what can be estimated with reasonable precision has, naturally, also been raised. This paper concerns the estimation of the quadratic covariation of various spot parameter processes associated with continuous time semimartingales, which are observed at discrete times over a finite interval of time. The main result of the paper is a central limit theorem that applies to a class of such estimators. Estimation of the quadratic covariation associated with spot parameter processes is, for example, important for learning about the (hyper-) parameters governing the spot parameter processes, e.g.~volatility-of-volatility; or for learning about possible dependencies between concurrently observed semimartingale processes; or for estimating the possible dependency between the observation times and various spot parameter processes associated with the observable process. The motivation for the present paper is an example of the latter, namely the estimation of the quadratic covariation between the volatility of a continuous semimartingale process, and the intensity processes governing the observation times of this process. 

To fix ideas, consider a typical analysis of high frequency data: Based on $n$ discrete time observations $X_{t_1},\ldots,X_{t_n}$ of a continuous semimartingale process $X_t$ one seeks to estimate an integrated parameter $\Theta$,  
\begin{equation}
\Theta_T = \int_0^T \theta_s \,\dd s ,
\notag
\end{equation}
where $\theta_t$ is a spot parameter process such as volatility, leverage effect, an instantaneous regression coefficient, or the like. The canonical example is the case where $\theta_t = \sigma_t^2$ is the spot-volatility process associated with an It{\^o} process of the form $\dd X_t = \mu_t \,\dd t + \sigma_t \,\dd W_t$, where $W_t$ is a standard Wiener process, and the problem is to estimate the integrated volatility $\int_0^T \sigma_s^2 \,\dd s$ over one or consecutive intervals of time. This example goes back to the research on realised volatility by \citet{abdl01}, \citet{barndorffnielsenshephard02}, \citet{jacodprotter98}, \citet{zhang2005tale}, and others. The econometric interest in investigating nonparametric estimates of this type grew out of the study of volatility clustering by \citet{engle82} and \citet{bollerslev86}. For further references, see
\citet{jacod2011discretization}, \citet{myklandzhang2012}, and \citet{ait2014high}.

The general setup and results of this paper take the following form. Let $\alpha_t$ and $\beta_t$ be spot parameter processes (potentially the same) associated with one or more semimartingale processes observed at discrete times over a finite interval of time $[0,T]$. In~\citet{mykland2017assessment,mykland2017assessmentsupplement} an estimator of the quadratic covariation $[\alpha,\beta]_T$ was introduced, and it was shown that this estimator is consistent. In the present paper we further derive the convergence rates for such estimators, and prove a general central limit theorem that, under some regularity conditions, applies to a wide range of estimators based on the second differencing of estimators of integrated spot processes. As mentioned, our main example of the use of this estimator is the problem of estimating the quadratic covariation between the volatility of a semimartingale process, and the intensity of the observation times of this process. This type of endogenous time problem exists in real applications but is often overlooked. We also sketch how our estimation methods and the central limit theorem can be applied to a novel estimator of the leverage effect.  

The paper proceeds as follows. In Section~\ref{eq::general_setup} we first describe the model and state our most important assumptions, subsequently we provide a heuristic derivation of the stochastic quantities that are important for the theory that follows. Section~\ref{sec::consistency} contains the consistency results and introduces the {``}two-scale{''} estimator of $[\alpha,\beta]$. These consistency results generalise the findings in~\citet{mykland2017assessment}. In Section~\ref{sec::normality} we present the main theoretical novelty of the paper, namely a central limit theorem for triangular array rolling quadratic variations based on second differencing of estimators of integrated spot processes. The proof of this theorem is deferred to Appendix~\ref{app::CLT_proof}. Section~\ref{sec::consistency} also contains an important corollary to the effect that the {`}Observed asymptotic variance{'} developed in~\citet{mykland2017assessment} yields consistent estimates of the asymptotic variance of the two-scales estimator we introduce. In Section~\ref{sec::volatility_and_intensity} we specialise the theory developed in the preceding sections to the problem of estimating the quadratic covariation between the spot parameter process of a continuous time semimartingale, and the intensity process of its observation times. This is the volatility-intensity problem. In Section~\ref{sec::simulations} we investigate the finite samples properties of our estimator by way of a simulation study, while Section~\ref{sec::empirical} contains an empirical analysis of the Apple stock observed over $21$ trading days in January $2018$. Most technical matters as well as long proofs can be found in the appendices. Appendix~\ref{app::stable_clt_jumps} also contains a stable central limit theorem for c{\`a}dl{\`a}g martingales, as well as a corollary with some alternative conditions that might be easier to check in applications.    

\section{The general setup and problem}\label{eq::general_setup}
In this section we first present the setting for our estimation procedures, define some key quantities, provide a heuristic overview of some important results, and explain what type of estimators our central limit theorem applies to. Subsequently, in Section~\ref{sec::consistency}, we provide a more formal presentation, and state the main consistency results of the paper.  

\subsection{Setup and basic insights}\label{subsec::setup} We suppose that one or more semimartingale processes $X_t$ are observed at high frequency over a finite interval of time $[0,T]$. The semimartingales $X_t$ are typically contaminated by microstructure noise, so what we observe is $Y_{t_i} = X_{t_i} + \eps_{t_i}$, for $i =1,\ldots,n$ time points, where $\eps_{t_i}$ is microstructure noise. Based on these data we form estimators $\widehat{\Theta}^n$ and $\widehat{\Lambda}^n$, which are consistent for $\Theta_t = \int_0^t \theta_s\,\dd s$ and $\Lambda_t  = \int_0^t \lambda_s\,\dd s$, respectively, where the spot parameter processes $\theta_t$ and $\lambda_t$ are also assumed to be semimartingales. Our results continue to hold when $\theta_t$ and $\lambda_t$ are replaced by the sequences $\theta_t^{(n)}$ and $\lambda_t^{(n)}$ of semimartingale processes, but to ease the notation we drop the superscript $n$ for the time being. The spot parameter processes $\theta_t$ may be the spot volatility of the continuous part $X^c$ of the process $\dd X_t = \sigma_s \,\dd W_s  + \text{$\dd t$-terms} + \text{jumps}$, with $W_t$ a standard Wiener process, that is $\theta_t = \sigma_t^2$; it may be the instantaneous leverage effect, $\theta_t = \dd [X^c,\sigma^2]_t/\dd t$; or the instantaneous volatility of volatility, $\theta_t = \dd [\sigma^2,\sigma^2]_t/\dd t$; or the stochastic intensity process governing the frequency of the observation times, etc. 

To be clear, the notation $[X,Y]_t$ refers to the continuous time quadratic covariation of two semimartingales $X$ and $Y$ from time zero to $t$ \citep[pp.~51--52]{\JS}. Semimartingales are defined in, for example, \citet[Definition~I.4.21, p.~43]{\JS}.  

\begin{definition} We assume that all our semimartingales are c{\`a}dl{\`a}g (right continuous with left limits), and that all data generating and latent processes live on the same filtered probability space $(\Omega,\Falg,\FF,P)$ with $\FF = (\Falg_t)_{0 \leq t \leq T}$, and that this filtered space satisfies the {`}usual conditions{'} \citep[Definitions I.1.2--I.1.3, p.~2]{\JS}. When necessary, we will also invoke sequences of filtrations $\FF^n = (\Falg_t^n)_{0 \leq t \leq T}$ on $(\Omega,\Falg,P)$, that is~$\Falg_{T}^n \subseteq \Falg$ for all $n$.      
\end{definition}
For the proof of the main central limit theorem of the paper, Theorem~\ref{theorem::clt}, we will need a few additional technical conditions on the structure of the filtered probability space. 
   
We now turn to the construction of our estimator. Divide the time interval $[0,T]$ into $B_n$ blocks $(t_{i-1}^n,t_i^n]$, of equal length, with $t_0^n = 0$ and $t_{B_n}^n = T$. Set $\Delta_n = T/B_n$, and for convenience, assume that $t_i^n = i\Delta_n$ for $i = 1,\ldots,B_n$. Since we shall permit rolling and overlapping intervals, let $K_n$ be an integer no greater than $B_n/2$. From now on we drop the index $n$ from the $t_i^n$, $B_n$ and $K_n$ when it does not cause confusion. For any real functions $\Theta_t$ and $\Lambda_t$, define
\begin{equation}
\qv_{B,K}(\Theta,\Lambda)_{T} = \frac{1}{K}\sum_{i=K}^{B-K}(\Theta_{(t_i,t_{i+K}]} - \Theta_{(t_{i-K},t_{i}]})(\Lambda_{(t_i,t_{i+K}]} - \Lambda_{(t_{i-K},t_{i}]}),
\label{eq::QV.general}
\end{equation}      
where $\Theta_{(s,t]} = \Theta_{t} - \Theta_{s}$, and write $\qv_{B,K}(\Theta)_{T} = \qv_{B,K}(\Theta,\Theta)_{T}$. For $l=1,\ldots,2K_n$, the notation $i \equiv l[2K_n]$ means that  
\begin{equation}
i = 2K_nj + l, \quad {\rm for } \quad K_n \leq i \leq B_n-K_n,
\notag
\end{equation}
with $j$ an increasing sequence of integers. The basic building block for all the estimators we present is the rolling quadratic covariation
\begin{equation}
\qv_{B,K}(\widehat{\Theta}^n,\widehat{\Lambda}^n)_{T}
= \frac{1}{K}\sum_{i=K}^{B-K}(\widehat{\Theta}^n_{(t_i,t_{i+K}]} - \widehat{\Theta}_{(t_{i-K},t_{i}]})(\widehat{\Lambda}^n_{(t_i,t_{i+K}]} - \widehat{\Lambda}^n_{(t_{i-K},t_{i}]}), 
\notag
\end{equation}
where $\widehat{\Theta}^n$ and $\widehat{\Lambda}^n$ are consistent estimators of the integrated spot processes $\Theta_t = \int_0^t \theta_s\,\dd s$ and $\Lambda_t = \int_0^t \lambda_s\,\dd s$, respectively. It is important to keep in mind that $\qv_{B,K}(\Theta,\Theta)_{T}$ and $\qv_{B,K}(\widehat{\Theta}^n,\widehat{\Lambda}^n)_{T}$ are defined on the discrete grid $\{0,\Delta_n,2\Delta_n,\ldots,T\}$, as opposed to the continuous time quadratic covariation $[X,Y]_t$.    

To see how $\qv_{B,K}(\widehat{\Theta}^n,\widehat{\Lambda}^n)_{T}$ is used to estimate $[\theta,\lambda]_T$, we here present a heuristic analysis, to be made precise in the subsequent section. Under the assumption that $\widehat{\Theta}_t^n$ can be expressed as a sum of $\Theta_t = \int_0^t \theta_s\,\dd s$, an error martingale, and terms associated with the edge effects, we can write,
\begin{equation}
{\rm QV}_{B,K}(\widehat{\Theta}^n,\widehat{\Lambda}^n)_{T}
= {\rm QV}_{B,K}(\Theta,\Lambda)_{T} + \text{estimation error},
\notag
\end{equation} 
where the {`}estimation error{'} might contain terms that are not asymptotically negligible, and must be dealt with by so-called two-scale constructions \citep{zhang2005tale,mykland2019algebra}. We return to this issue shortly. From \citet[Theorem~1, p.203]{mykland2017assessment}, we have the {`}Integral-to-Spot Device{'}, that is 
\begin{equation}
\frac{\qv_{B_n,K_n}(\Theta,\Lambda)}{(K_n\Delta_n)^2}  
= \frac{2}{3}\big(1 - \frac{1}{K_n}\big)[\theta,\lambda]_{T-}
+ \frac{1}{K_n^2} \int_0^{T} \big( \big(\frac{t^{*} - t}{\Delta_n} \big)^2 + \big( \frac{t - t_{*}}{\Delta_n}\big)^2 \big)\,\dd [\theta,\lambda]_{t} + o_p(1),
\notag
\end{equation}
as $K_n\Delta_n \to 0$, and $t_{*} = \min\{i\Delta_n \colon i\Delta_n < t \}$ and $t^{*} = \max\{i\Delta_n \colon i\Delta_n \geq t \}$. The key ingredient for proving this theorem is an application of Lemma~2 in~\citet[p.~206]{mykland2017assessment}, from which we obtain that 
\begin{equation}
\begin{split}
\frac{\qv_{B_n,K_n}(\Theta,\Lambda)}{(K_n\Delta_n)^2}   &
= \frac{1}{K}\sum_{i=K}^{B - K} (\int_{t_i}^{t_{i+K}}\frac{t_{i+K} - s}{K\Delta_n}\,\dd \theta_s + \int_{t_{i-K}}^{t_{i}}\frac{s - t_{i-K}}{K\Delta_n}\,\dd \theta_s)\\
& \qquad \qquad \qquad \qquad\times (\int_{t_i}^{t_{i+K}}\frac{t_{i+K} - s}{K_n\Delta_n}\,\dd \lambda_s + \int_{t_{i-K}}^{t_{i}}\frac{s - t_{i-K}}{K_n\Delta_n}\,\dd \lambda_s)\\
& = \frac{1}{K}\sum_{l=1}^{2K}\sum_{i \equiv l[2K]} \int_{t_{i-K}}^{t_{i+K}} f_{s}^{(l,K)}\,\dd \theta_s \int_{t_{i-K}}^{t_{i+K}} f_{s}^{(l,K)}\,\dd \lambda_s, 
\end{split}
\label{eq::qv_heuristics1}
\end{equation} 
where $f_s^{(l,K)}$ for $l = 1,\ldots,2K$ are the functions
\begin{equation}
f_s^{(l,K)} = \sum_{i \equiv l[2K], K \leq i \leq B - K} (\frac{t_{i+K} - s}{K_n\Delta_n}
I\{t_{i} \leq s < t_{i+K}\} + \frac{s - t_{i-K}}{K_n\Delta_n}I\{t_{i-K} \leq s < t_{i}\}).
\label{eq::f_func1}
\end{equation}
The central limit theorem we present in Section~\ref{sec::normality} concerns quantities of the type 
\begin{equation}
\frac{1}{2K}\sum_{l=1}^{2K}\bigg\{\sum_{i \equiv l[2K]}\int_{t_{i-K}}^{t_{i+K}} f_s^{(l,n)}\,\dd\alpha_s^{(n)} \int_{t_{i-K}}^{t_{i+K}} g_s^{(l,n)}\,\dd\beta_s^{(n)} - \int_{0}^{T}f_s^{(l,n)}g_s^{(l,n)}\,\dd [\alpha^{(n)},\beta^{(n)}]_s \bigg\},
\label{eq::clt_object1}
\end{equation}
as $K_n\Delta_n\to 0$ and $K_n\to \infty$. In~\eqref{eq::clt_object1} the functions $f_s^{(l,n)}$ and $g_s^{(l,n)}$ are bounded and deterministic, while $\alpha^{(n)}$ and $\beta^{(n)}$ are sequences of semimartingale processes. We see that the right hand side of~\eqref{eq::qv_heuristics1} is a special case of \eqref{eq::clt_object1}, and so are the non-negligible terms contained in the {`}estimation error{'} referred to above. 

\subsection{Consistency}\label{sec::consistency} Suppose that $\Theta_t^{(n)} = \int_0^t \theta_s^{(n)}\,\dd s$ and $\Lambda_t^{(n)} = \int_0^t \lambda_s^{(n)}\,\dd s$ are two integrated spot-processes, and that $\theta_t^{(n)}$ and $\lambda_t^{(n)}$ are sequences of semimartingales adapted to $\FF^n$ (or $\FF$ in the case that $\theta_t^{(n)} = \theta_t$ or $\lambda_t^{(n)} = \lambda$ for all $n$), and that both sequences satisfy Condition~\ref{cond:modified-semimgs} below. If  $\theta_t^{(n)}$ and $\lambda_t^{(n)}$ depend on $n$, we assume that the pair converges in probability to limiting semimartingales $(\theta_t$ and $\lambda_t)$, and that $[\theta^{(n)},\lambda^{(n)}]$ converges in probability to $[\theta,\lambda]$. 

The two spot-processes might be associated with the same underlying semimartingale (in which case we can have $\theta^{(n)} = \lambda^{(n)}$ for all $n$), or with two different semimartingales concurrently observed. In the latter case, the sampling times can be asynchronous, and the total number of observations may differ. To not overburden the notation, however, we assume that the number of observations are the same for both processes, and equals $n$. We are given the estimators $\widehat{\Theta}^n_t$ and $\widehat{\Lambda}^n_t$ of $\Theta_t^{(n)}$ and $\Lambda_t^{(n)}$, respectively. Both $\widehat{\Theta}^n_t$ and $\widehat{\Lambda}^n_t$ are consistent and admit representations of the type $\widehat{\Theta}^n_t = \Theta_t + M_{n,t}^{\theta} + e^{\theta}_{n,t} - \tilde{e}^{\theta}_{n,0}$, in terms of a semimartingale $M_{n,t}^{\theta}$ and edge effects $e^{\theta}_{n,t}$ and $\tilde{e}^{\theta}_{n,0}$ associated with phasing in and phasing out the estimator, respectively. For $s<t$ we write $\widehat{\Theta}^n_{(s,t]} = \widehat{\Theta}^n_{t} - \widehat{\Theta}^n_{s}$. This means that for $s<t$ the estimators can be represented as    
\begin{equation}
\begin{split}
\widehat{\Theta}_{(s,t]} - \Theta_{(s,t]} &= M_{n,t}^{\theta} - M_{n,s}^{\theta} + e_{n,t}^{\theta} - e_{n,s}^{\theta},\\
\widehat{\Lambda}_{(s,t]} - \Lambda_{(s,t]} & = M_{n,t}^{\lambda} - M_{n,s}^{\lambda} + e_{n,t}^{\lambda} - e_{n,s}^{\lambda}.
\end{split}
\label{eq::GH.decomp1}
\end{equation}
The assumption, implicit in~\eqref{eq::GH.decomp1}, that the edge effect of phasing in an estimator at $s<t$ is the same as the edge effect associated with phasing out an estimator at $t$. This is exact in the (usual) case of additive estimators \citep[Section 5.1, p.~215]{mykland2017assessment}. The results that follow extend with little effort to situations where the edge effects in the two ends of the interval behave differently. 

\begin{definition} {\sc (Stable convergence)}. We say that a sequence $Z_n = (Z_{n,t})_{0 \leq t \leq T}$ of martingales converges stably in law to $Z = (Z_t)_{0 \leq t \leq T}$ with respect to $\Galg \subseteq \Falg$ if (i) $Z$ is measurable with respect to $\widetilde{\Galg}$ belonging to an extension $(\widetilde{\Omega},\widetilde{\Galg},\widetilde{P})$ of $(\Omega,\Galg,P)$; and (ii) for every $\Galg$ measurable (real-valued) random variable $Y$, the sequence $(Z_n,Y)$ converges in law to $(Z_n,Y)$. We then write $Z_n \Rightarrow Z$ stably.   
\end{definition}

\begin{condition}\label{cond::MGs} Assume that~\eqref{eq::GH.decomp1} holds, and that there are $\alpha>0$ and $\beta > 0$ such that, as $n \to \infty$, 
\begin{equation}
n^{\alpha}M_{n}^{\theta} \Rightarrow L^{\theta}\quad {\rm and}\quad n^{\beta}M_{n}^{\lambda} \Rightarrow L^{\lambda}\qquad \text{stably},
\notag
\end{equation}
with respect to a $\sigma$-algebra $\Galg \subseteq \Falg$. Both $n^{\alpha}M_{n,t}^{\theta}$ and $n^{\beta}M_{n,t}^{\lambda}$ are P-UT (see Appendix~\ref{app::appendixA}), and the quadratic variations $[L^{\theta},L^{\theta}]_{T}$ and $[L^{\lambda},L^{\lambda}]_{T}$ are measurable with respect to $\Galg$. 
\end{condition} 
\begin{remark} The requirements of Condition~\ref{cond::MGs} are likely to be satisfied in applications, but they are stronger than what we need for the present purposes. For the consistency results of this section we only need the weaker Condition~5 of \citet[p.~7]{mykland2017assessmentsupplement}, which is implied by Condition~\ref{cond::MGs}. A sequence of semimartingales $(Z_n) = (Z_{n,t})_{0 \leq t \leq T}$ fulfills this condition if it is tight and P-UT (see \citet[p.~377]{\JS} for the definition of the P-UT property).
\end{remark}

\begin{theorem}\label{theorem::Th3.multivariate} {\sc (Consistency of the covariance estimator)} Assume that $\widehat{\Theta}_t^n$ and $\widehat{\Lambda}_t^n$ satisfy \eqref{eq::GH.decomp1} and Condition~\ref{cond::MGs}. Let $K = K_n$ be positive integers, assume that $K_n\Delta_n \to 0$, and that the edge effects  $e^{\theta}_t$ and $e^{\lambda}_t$ are $o_p((K_n\Delta_n)^{1/2}n^{-\alpha})$ and $o_p((K_n\Delta_n)^{1/2}n^{-\beta})$, respectively. Then
\begin{equation}
\qv_{B,K}(\widehat{\Theta},\widehat{\Lambda})_T = 2[M_n^{\theta},M_n^{\lambda}]_{T} + \frac{2}{3}(K_n\Delta_n)^2[\theta^{(n)},\lambda^{(n)}]_{T} + o_p( (K_n\Delta_n)^{2} ) + o_p(n^{-(\alpha + \beta)}),
\notag
\end{equation}
as $n \to \infty$. 
\end{theorem}
\begin{proof} The proof follows with trivial adjustments from \citet[Theorem~3, p.~208]{mykland2017assessment}. A brief sketch of the proof along with some remarks on the edge effects are given in Appendix~\ref{app::QCV.proof1}. 
\end{proof}
In Appendix~\ref{app::QCV.proof1} we also provide the conclusion of the above theorem with slightly more stringent restrictions on the edge effects. Corresponding results for all combinations of assumptions on the edge effects can be deduced from the results in Appendix~\ref{app::QCV.proof1}.  

We now turn to estimation of the quadratic covariation $[\theta,\lambda]$. As will become clear, how one ought to estimate $[\theta,\lambda]$ depends on the convergence rates of the error martingales $M_n^{\theta}$ and $M_n^{\lambda}$, that is the $\alpha$ and $\beta$ required for $n^{\alpha}M_n^{\theta}$ and $n^{\beta}M_n^{\lambda}$ to satisfy Condition~\ref{cond::MGs}. From the conclusion of Theorem~\ref{theorem::Th3.multivariate} we see that, provided $K_n\Delta_n$ is of order $n^{-\alpha\wedge\beta}$, then 
\begin{equation}
\frac{\qv_{B,K}(\widehat{\Theta},\widehat{\Lambda})_T}{(K_n\Delta_n)^2} 
= \frac{2[M_n^{\theta},M_n^{\lambda}]_{T}}{(K_n\Delta_n)^{2}} + \frac{2}{3}[\theta^{(n)},\lambda^{(n)}]_{T} + o_p(1).
\notag
\end{equation}
By Condition~\ref{cond::MGs} the quadratic covariation of the error martingales $[M_n^{\theta},M_n^{\lambda}]$ is $O_p(n^{-(\alpha + \beta)})$, consequently,
\begin{equation}
(K_n\Delta_n)^{-2}[M_n^{\theta},M_n^{\lambda}] = O_p( (K_n\Delta_n)^{-1} n^{- \alpha\vee\beta)}  ),
\notag
\end{equation} 
which tends to zero in probability as $n \to \infty$ provided $\alpha \neq \beta$. We summarise this in a lemma.

\begin{lemma}\label{lemma::consistency1} Assume that Condition~\ref{cond::MGs} holds. Suppose that $\alpha \neq \beta$, that $K_n \to \infty$ and $\Delta_n \to 0$ such that $K_n \Delta_n$ is of order $n^{-\alpha\wedge\beta}$ as $n \to \infty$, then 
\begin{equation}
\frac{3}{2}\frac{\qv_{B,K}(\widehat{\Theta},\widehat{\Lambda})_T}{(K_n\Delta_n)^2} = [\theta,\lambda]_{T-} + o_p(1). 
\label{eq::consistency1}
\end{equation} 
\end{lemma}
\begin{proof} By Condition~\ref{cond::MGs} this is direct from the two displays above. 
\end{proof}
Notice that the conclusion of Lemma~\ref{lemma::consistency1} continues to hold when $\alpha = \beta$ provided $(K_n\Delta_n)^{-1}M_n^{\theta}$ and $(K_n\Delta_n)^{-1}M_n^{\lambda}$ are asymptotically orthogonal (see \citet[Proposition~I.4.15, p.~41]{\JS} for the notion of local martingales being orthogonal). Also note that when $\alpha = \beta$ one might choose $K_n\Delta_n$ such that $K_n\Delta_n n^{\alpha} \to \infty$, at the cost of a slower rate of convergence.

The estimation problem is harder when $(K_n\Delta_n)^{-2}[M_n^{\theta},M_n^{\lambda}]_T$ is not asymptotically negligible. This occurs, for example, when one seeks to estimate $[\theta,\theta]_T$, such that the convergence rates $\alpha$ and $\beta$ of Condition~\ref{cond::MGs} are equal. As an estimator of the quadratic covariation $[\theta,\lambda]_T$ in such situations we propose the {\it Two Scales Quadratic Covariation} ({\tsqc}) estimator. It is given by 
\begin{equation}
\tsqc_{B,K_1,K_2}(\widehat{\Theta}^n,\widehat{\Lambda}^n)_{T} 
= \frac{3}{2}\frac{\qv_{B,K_2}(\widehat{\Theta}^n,\widehat{\Lambda}^n) - \qv_{B,K_1}(\widehat{\Theta}^n,\widehat{\Lambda}^n)}{(K_{n,2}^2 - K_{n,1}^2)\Delta_n^2},
\label{eq::ts_thetalambdahat1}
\end{equation}
where $K_{n,2} > K_{n,1}$ are user specified sequences of integers (tuning parameters) tending to infinity. Since $K_{n,2}$ and $K_{n,1}$ must be of the same order, a natural choice is $K_{n,2} = \gamma K_{n,1}$ for some integer $\gamma \geq 2$, with $\gamma$ fixed and independent of $n$. We first present a consistency result, and then return to the central limit theory for this estimator at the end of Section~\ref{sec::normality}.  
\begin{corollary}\label{corr::TSQC_corollary}{\sc (Consistency of the TSQC-estimator)} Assume that the conditions of Theorem~\ref{theorem::Th3.multivariate} are in force, and that $\alpha = \beta$. Let $K_{n,2} = \gamma K_{n,1}$, for some fixed integer $\gamma \geq 2$, be positive integers tending to infinity such that $K_{n,1}\Delta_n = O(n^{-\alpha})$. Then, 
\begin{equation}
{\rm TSQC}_{B,K_1,K_2}(\widehat{\Theta}^n,\widehat{\Lambda}^n)_{T}  =  
[\theta,\lambda]_{T-} + o_p(1), 
\notag
\end{equation} 
as $n \to \infty$.
\end{corollary}
\begin{proof} From Theorem~\ref{theorem::Th3.multivariate} we have that for $j =1,2$
\begin{equation}
\frac{\qv_{B,K_j}(M_n^{\theta},M_n^{\lambda})_T}{(\gamma^2 - 1)K_{n,1}^2\Delta_n^2} = 2[n^{\alpha}M_n^{\theta},n^{\alpha}M_n^{\lambda}]_{T-}
+ O_p(K_{n,j}\Delta_n),
\notag
\end{equation}
so when $K_{n,2} = \gamma K_{n,1}$ for some $\gamma \geq 2$, then 
\begin{equation}
\frac{\qv_{B,K_2}(M_n^{\theta},M_n^{\lambda})_T - \qv_{B,K_1}(M_n^{\theta},M_n^{\lambda})_T}{(\gamma^2-1)K_{n,1}^2\Delta_n^2} = o_p(1).
\notag
\end{equation}
On the other hand, 
\begin{equation}
\frac{\qv_{B,K_2}(\Theta,\Lambda)_T}{(\gamma^2 - 1)K_{n,1}^2 \Delta_n^2} = \frac{\gamma^2}{(\gamma^2 - 1)K_{n,2}} \sum_{l=1}^{2K_2}\sum_{i \equiv l[2K_2]}\int_{t_{i-K_2}}^{t_{i+K_2}}f_s^{(l,K_2)}\,\dd \theta_s\int_{t_{i-K_2}}^{t_{i+K_2}}f_s^{(l,K_2)}\,\dd \lambda_s,
\notag
\end{equation}
so that by \citet[Theorem~7, p.~1]{mykland2017assessmentsupplement} 
\begin{equation}
\begin{split}
&\frac{\qv_{B,K_2}(\Theta,\Lambda)_T - \qv_{B,K_1}(\Theta,\Lambda)_T}{(\gamma^2 - 1)K_{n,1}^2 \Delta_n^2}\\ 
& \qquad\qquad  = \frac{\gamma^2}{(\gamma^2 - 1)K_{n,2}} \sum_{l=1}^{2K_2} \int_{0}^{T-} (f_s^{(l,K_2)})^2\,\dd [\theta,\lambda]_s\\
& \qquad \qquad \qquad \qquad - \frac{1}{(\gamma^2 - 1)K_{n,1}} \sum_{l=1}^{2K_1} \int_{0}^{T-} (f_s^{(l,K_1)})^2\,\dd [\theta,\lambda]_s
+ O_p((K_{n,1}\Delta_n)^{1/2}).
\end{split}
\notag
\end{equation}
Thus $\tsqc_{B,K_1,K_2}(\widehat{\Theta}^n,\widehat{\Lambda}^n)_{T}  = [\theta^{(n)},\lambda^{(n)}]_{T-} + ((\gamma^2 - 1)K_{n,1}^2\Delta_n^2)^{-1}o_{p}(n^{-2\alpha}) + o_{p}(1)$ as $K_{n,1}\Delta_n \to 0$ with $K_{n,1}\to \infty$, and the result follows because $K_{n,1}\Delta_n$ is of order $n^{-\alpha}$.
\end{proof}
\begin{remark} The conclusion of Corollary~\ref{corr::TSQC_corollary} is still valid when $\alpha \neq \beta$ provided $K_{n,1}\Delta_n = O(n^{-\alpha\wedge\beta})$. But if the convergence rates are known and different one would, as already mentioned, rather use the estimator in~\eqref{eq::consistency1}. There might be situations, however, where the convergence rates $\alpha$ and $\beta$ are not known exactly, but known to lie in some interval, say $\alpha,\beta\in[r_1,r_2]$. In that case, one sets $K_{n,1}\Delta_n = O(n^{-r_1})$, and the conclusion of Corollary~\ref{corr::TSQC_corollary} holds.     
\end{remark}

\section{Central limit theory}\label{sec::normality} The consistency results of the previous section are extensions of theory developed in \citet{mykland2017assessment,mykland2017assessmentsupplement}. That paper, however, did not establish limiting normality for the estimators presented, and it is to this topic we now turn. 

In a first part we present a theorem on the convergence rate of triangular array rolling quadratic covariations as approximations to quadratic covariations of spot processes. We then present the central limit theorem for such approximations. Both these results supplement the consistency result of~\citet[Theorem~7, p.~1]{mykland2017assessmentsupplement}. The proofs of both these theorems are deferred to the appendix. As an example of the use of this theorem, and to show its versatility, we show how it can be applied to a novel estimator of the leverage effect. In Section~\ref{subsec::tsqc_theory} we present theory for the {\tsqc}-estimator. In particular, we show that the observed asymptotic variance of~\citet{mykland2017assessment} can be applied to estimate the asymptotic variance of this estimator. This is important because analytical expressions for the {\tsqc} are hard to derive (see the discussion in ~\citet[pp.~198--200]{mykland2017assessment}). 
 
\subsection{Convergence rate and CLT for rolling quadratic variations}
Introduce the processes 
\begin{equation}
\alpha_{t}^{(l,n)} = \int_0^t f_{s-}^{(l,n)}\,\dd \alpha_s^{(n)},
\mbox{ and } \beta_{t}^{(l,n)} = \int_0^t g_{s-}^{(l,n)}\,\dd \beta_s^{(n)}, \quad {\rm for}\quad l = 1,\ldots,2K_n, 
\label{eq::alpha.beta.ln}
\end{equation}
where $\alpha_t^{(n)}$ and $\beta_t^{(n)}$ are sequences of semimartingales, and $f_{t}^{(l,n)}$ and $g_{t}^{(l,n)}$ are deterministic c{\`a}dl{\`a}g functions bounded by $1$ (there is nothing special about $1$ here, and it suffices that they are bounded by a constant). We denote by $\mathbb{F}$ a countable collection $f_{\cdot}^{(l,n)}\,l=1,\ldots,2K_n,\;n=1,2,\ldots$, of such functions, such as from~\eqref{eq::f_func1}, but  more generally to be defined in each case, and similarly $g_{\cdot}^{(l,n)}$ belongs to the collection $\mathbb{G}$ (see Appendix~\ref{app::appendixA} for further details). In \citet[Theorem 7, p.~1]{mykland2017assessmentsupplement} it was shown that  
\begin{equation}
\frac{1}{2K}\sum_{l=1}^{2K}\sum_{i\equiv l[2K]}(\alpha_{t_{i+K}}^{(l,n)} - \alpha_{t_{i-K}}^{(l,n)})(\beta_{t_{i+K}}^{(l,n)} - \beta_{t_{i-K}}^{(l,n)})
= \frac{1}{2K}\sum_{l=1}^{2K}[\alpha^{(l,n)},\beta^{(l,n)}]_{T-} + o_p(1).
\label{eq::key_approx}
\end{equation}
In this section we study the rate of convergence and present a central limit theorem for the approximation in~\eqref{eq::key_approx}. Such statements will help with the assessment of the accuracy and with optimal calibration of the TSQC-estimators, as well as other rolling intervals estimators that depend on approximations such as the one in~\eqref{eq::key_approx}.

\begin{theorem}\label{th::convergence.rates}{\sc (Rate of convergence).} Suppose that $\alpha_t^{(n)}$ and $\beta_t^{(n)}$ satisfy Conditions~\ref{cond:modified-semimgs}-\ref{cond:rate} in  Appendix~\ref{app::appendixA}, that $f_{\cdot}^{(l,n)} \in \mathbb{F}$ and $g_{\cdot}^{(l,n)} \in \mathbb{G}$, and that $\alpha_{t}^{(l,n)}$ and $\beta_{t}^{(l,n)}$ are as defined in~\eqref{eq::alpha.beta.ln}. Then 
\begin{equation}
\frac{1}{2K_n}\sum_{l=1}^{2K_n}\sum_{i \equiv l[2K_n]} (\alpha_{t_{i+K}}^{(l,n)} - \alpha_{t_{i-K}}^{(l,n)})
(\beta_{t_{i+K}}^{(l,n)} - \beta_{t_{i-K}}^{(l,n)})
= \frac{1}{2K_n}\sum_{l=1}^{2K_n}[\alpha^{(l,n)},\beta^{(l,n)}]_{T-} + O_p\big( (K_n\Delta_n)^{1/2} \big).
\notag
\end{equation} 
\end{theorem}
\begin{proof} See Appendix~\ref{app::conv.rates}. 
\end{proof}
Let the error term in the approximation in~\eqref{eq::key_approx} be
\begin{equation}
\begin{split}
Z_{n}(t) & = \frac{1}{2K}\sum_{l=1}^{2K}\big\{\sum_{t_{i+K}\leq t,\,i\equiv l[2K]}(\alpha_{t_{i+K}}^{(l,n)} - \alpha_{t_{i-K}}^{(l,n)})(\beta_{t_{i+K}}^{(l,n)} - \beta_{t_{i-K}}^{(l,n)}) \\
& \qquad \qquad \qquad \qquad + (\alpha_{t}^{(l,n)} - \alpha_{t_{*,l}}^{(l,n)})
(\beta_{t}^{(l,n)} - \beta_{t_{*,l}}^{(l,n)})
- [\alpha^{(l,n)},\beta^{(l,n)}]_{t}\big\}.
\end{split}
\label{eq::CLT_MG}
\end{equation}
Notice that $Z_{n}(t)$ is interpolated into a continuous time martingale. The errors are only defined at discrete times, but the interpolation error is asymptotically negligible, and consequently we only need to prove the central limit theorem for the interpolated process, which will be done by applying the general central limit theorem, Theorem~\ref{theorem::th2.28_general}, that is contained in Appendix~\ref{app::stable_clt_jumps}. 

For the notion of an $\Falg$-conditional Gaussian martingale, see \citet[Definition~II.7.4, p.~129]{jacod2003limit}, or \citet[p.~233]{jacod1997continuous}. Define 
\begin{equation}
t_{*,l} = \max\{t_{i+K} : t_{i+K}, i \equiv l[2K]\}.
\notag
\end{equation}
We write $\nu_{\alpha}^n$ for the compensator of the jump process $\mu_{\alpha}$ associated with a sequence (in $n$) of semimartingale process $\alpha^{(n)}$ (see \citet[Ch.~II.1]{\JS}). We can now state the main result of the paper.   

\begin{theorem}\label{theorem::clt}{\sc (CLT for triangular array rolling quadratic variations).} Suppose that Conditions~\ref{cond:modified-semimgs}--\ref{cond::clt2} in Appendix~\ref{app::appendixA} hold; that $\dd \langle \alpha^{(n)},\alpha^{(n)}\rangle_t/\dd t$, $\dd \langle \beta^{(n)},\beta^{(n)}\rangle_t/\dd t$, and $\dd \langle \alpha^{(n)},\beta^{(n)}\rangle_t/\dd t$ are locally continuous in mean square; and that for all $\eps > 0$, the Lindeberg conditon
\begin{equation}
\int_{\abs{x} > \eps} x^2 \nu_{\alpha}^n([0,T]\times \dd x) \overset{p}\to 0,\quad \text{and}\quad 
\int_{\abs{x} > \eps} x^2 \nu_{\beta}^n([0,T]\times \dd x) \overset{p}\to 0, 
\label{eq::spot_lindeberg}
\end{equation}
as $n \to \infty$ is satisfied for both processes. 
Set
\begin{equation}
\begin{split}
\kappa_s^{(n)} & =  \frac{1}{4 K^3 \Delta_n}\sum_{l_1=1}^{2K}\sum_{l_2=1}^{2K}
\int_{ t_{*,l_1}\vee t_{*,l_2} }^{s} \bigg\{\,f_u^{(l_1,n)}f_u^{(l_2,n)}g_s^{(l_1,n)}g_s^{(l_2,n)}\,\dd\langle\alpha^{(n)},\alpha^{(n)}\rangle_u\,\frac{\dd \langle\beta^{(n)},\beta^{(n)}\rangle_s}{\dd s}\,[2]\\
& \qquad\qquad \qquad\qquad + (f_u^{(l_1,n)}g_u^{(l_2,n)}) (g_{s}^{(l_1,n)}f_{s}^{(l_2,n)})\,\dd \langle \alpha^{(n)},\beta^{(n)} \rangle_{u} \frac{\dd\langle\beta^{(n)},\alpha^{(n)} \rangle_{s}}{\dd s}\,[2]\bigg\},\end{split}
\notag
\end{equation}
and assume that there is a $\Falg$-measurable process $\kappa_s$ for which
\begin{equation}
\int_0^t \kappa_s^{(n)}\,\dd s \overset{p}\to \int_0^t \kappa_s\,\dd s, \quad \text{for each $t \in [0,T]$}.
\notag
\end{equation}
Then $(K_n\Delta_n)^{-1/2}Z_{n}$ converges stably in law to an $\Falg$-conditional Gaussian martingale $\mathscr{Z}$ with quadratic variation 
\begin{equation}
\langle \mathscr{Z},\mathscr{Z}\rangle_t = \int_0^t \kappa_s\,\dd s.
\notag
\end{equation} 
\end{theorem} 

\begin{proof} See Appendix~\ref{app::CLT_proof}.
\end{proof}
The notation {``}$[2]${''} means that we sum over two terms, the one given and the corresponding one where, in $\kappa^{(n)}$, $f$ and $g$ have changed place. As an example, $a_1b_2 + a_2b_1 = a_1b_2[2]$. The meaning of the notation will be clear from the context. The notion of being locally continuous in mean square is defined in Definition~\ref{assumption3} in Appendix~\ref{cond:modified-semimgs}. 
Before we proceed to Section~\ref{sec::volatility_and_intensity}, and the application that motivated the present study, we showcase the applicability of Theorem~\ref{theorem::clt} by considering the problem of leverage effect estimation.
\begin{example}\label{example::Leverage_effect_estimation}{\sc (Leverage effect estimation)}. 
Estimators of the leverage effect have been studied previously by \citet{wang2014estimation}, \citet{kalnina2017nonparametric}, \citet{ait2017estimation}, to mention some. In this example we introduce a rolling intervals estimator of the leverage effect, and show how Theorem~\ref{theorem::clt} can be used to derive the limit distribution of this estimator. We limit ourselves to the following simple model. Assume that the process $X_t = X_0 + \int_0^t \sigma_s\,\dd W_s$ is observed at the discrete and equidistant times $0 = t_{0,n} < t_{1,n} < \cdots < t_{n-1,n} < t_{n,n} = T$, and that there is no microstructure noise; $W_t$ a one dimensional Wiener process, and $\sigma^2$ is a locally bounded It{\^o} process which may or may not be correlated with $W_t$. The leverage effect is the spot process $\dd [\sigma^2,X]_t/\dd t$. A natural estimator of $[X,X]_T$ is the realised volatility (see the references in the Introduction),
\begin{equation}
\widehat{\Theta}_T^n = \sum_{t_{i+1,n}\leq T}  (X_{t_{i+1,n}} - X_{t_{i,n}})^2,
\notag
\end{equation}
and define $M_{n,t}$ via $\widehat{\Theta} = [X,X]_t + M_{n,t}$. It can then be shown that $n^{1/2}M_{n}$ converges stably in law to a normal distribution with (random) variance $2T \int_0^T \sigma_s^4\,\dd s$ \citep[Corollary~2.30, p.~154]{myklandzhang2012}, hence $n^{1/2}M_n$ satisfies Condition~\ref{cond::MGs}. In analogy with~\eqref{eq::QV.general}, consider 
\begin{equation}
\qv_{B,K}(\widehat{\Theta}^n,X)_T = \frac{1}{K_n}\sum_{i=K_n}^{B_n-K_n}
(\widehat{\Theta}^n_{(t_{i,n},t_{i+K,n}]} - \widehat{\Theta}^n_{(t_{i-K,n},t_{i,n}]}) 
(X_{t_{i+K,n}} - X_{t_{i-K,n}}),
\label{eq::leverage_effect_estimator}
\end{equation} 
where, due to the equidistant sampling times, we take $B_n = n$. It then follows from Lemma~\ref{lemma::consistency1} that 
\begin{equation}
(K_n\Delta_n)^{-1}\qv_{B,K}(\widehat{\Theta}^n,X)_T = [\sigma^2,X]_T + o_p(1). 
\notag
\end{equation}
To sketch the application of the central limit theorem of this paper, note that we may write $\qv_{B,K}(\widehat{\Theta}^n,X)_T = \qv_{B,K}(\Theta,X)_T + \qv_{B,K}(M_n,X)_T$. Let $f_s^{(l,n)}$ be as defined in \eqref{eq::f_func1}, and introduce 
\begin{equation}
g_{s}^{(l,n)}  = \sum_{K \leq i \leq B - K,\,i \equiv l[2K]} ( I\{t_i \leq s < t_{i+K}\} - I\{t_{i-K} \leq s < t_{i}\}   ),
\label{eq::g_func1}
\end{equation}
for $l = 1,\ldots,2K$. Let $M_{t}^{(l,n)} = \int_{0}^{t} \bar{g}_s^{(l,n)}\,\dd M_{n,s}$, and define the two continuous time martingales 
\begin{equation}
\begin{split}
U_{n,l}(t) & = \sum_{t_{i+K}\leq t \,:\, i\equiv l[2K]} (\theta_{t_{i+K}}^{(l,n)} - \theta_{t_{i-K}}^{(l,n)})(X_{t_{i+K}} - X_{t_{i-K}})\\
& \qquad \qquad \qquad \qquad + (\theta_{t}^{(l,n)} - \theta_{t_{*,l}}^{(l,n)})(X_{t} - X_{t_{*,l}})
- [\theta^{(l,n)},X]_t,
\end{split}
\notag
\end{equation}
and 
\begin{equation}
\begin{split}
V_{n,l}(t) & = \sum_{t_{i+K}\leq t \,:\, i\equiv l[2K]}(M_{t_{i+K}}^{(l,n)} -M_{t_{i-K}}^{(l,n)})(X_{t_{i+K}} - X_{t_{i-K}})\\ 
& \qquad \qquad \qquad \qquad+ 
(M_{t}^{(l,n)} -M_{t_{*,l}}^{(l,n)})(X_{t} - X_{t_{*,l}})  - [M^{(l,n)},X]_{t}.
\end{split}
\notag
\end{equation}
Then $(K\Delta_n)^{-1}\qv_{n,K}(\widehat{\Theta}^n,X) - K^{-1}\sum_{l=1}^{2K} [\theta^{(l,n)},X]_T$ is asymptotically equivalent to $Z_n = K^{-1}\sum_{l=1}^{2K}\{U_{n,l}(T) + V_{n,l}(T) \}$. The predictable quadratic variation of $Z_{n}$ is 
\begin{equation}
\begin{split}
\langle Z_{n},Z_{n}\rangle_{T} & = 
\frac{1}{K^2} \sum_{l_1=1}^{2K}\sum_{l_2=1}^{2K}\langle U_{n,l_1} + V_{n,l_1},U_{n,l_2} + V_{n,l_2} \rangle_{T}\\
& = \frac{1}{K^2} \sum_{l_1=1}^{2K}\sum_{l_2=1}^{2K}\{\langle U_{n,l_1},U_{n,l_2}\rangle_{T}
 + \langle V_{n,l_1},V_{n,l_2}\rangle_{T}
 + 2\langle U_{n,l_1},V_{n,l_2}\rangle_{T}\}.
 \end{split}
\notag
\end{equation}
Provided that the processes involved satisfy the assumptions of Theorem~\ref{theorem::clt}, we see how the development so far leads to a central limit theorem for the leverage effect estimator of~\eqref{eq::leverage_effect_estimator}. In particular, $(K_n\Delta_n)^{-1/2}Z_n$ converges stably in law to a Gaussian martingale with (random) asymptotic variance of the form $\int_{0}^T (a_s + b_s + 2c_s)\,\dd s$. How this leverage effect estimator generalises to more complicated data structures, i.e.~non-equidistant sampling times, microstructure noise, and edge effects, is a topic we plan to explore in a subsequent paper. 
\end{example}

\subsection{Uncertainty of the {\tsqc}}\label{subsec::tsqc_theory} To compute the uncertainty associated with the {\tsqc} estimator we use the {\it observed asymptotic variance} of \citet{mykland2017assessment}, which allows us to circumvent the derivation of an explicit expression for the asymptotic variance of the {\tsqc} estimator. The applicability of the observed asymptotic variance is contingent on the sequences of semimartingales in question satisfying Condition~\ref{cond::MGs} in Section~\ref{sec::consistency}, or the weaker Condition~5 in \citet[p.~7]{mykland2017assessmentsupplement}. According to this latter condition the sequence of error martingales associated with the estimator whose uncertainty one wants to compute needs to be tight and P-UT. Consider 
\begin{equation}
(K_{n,1}\Delta_n)^{-1/2}(\tsqc_{B,K_2,K_1}(\widehat{\Theta}^n,\widehat{\Lambda}^n)_t - [\theta,\lambda]_t).
\label{eq::tsqc_sequence}
\end{equation}
This is the sequence for which we are going to use the observed asymptotic variance to compute its uncertainty. For $l = 1,\ldots,2K$ and $K \leq B/2$, define the interpolated processes
\begin{equation}
\begin{split}
\mathbb{Z}_K^{(l)}(f,\alpha,g,\beta)_t 
& = \sum_{i \equiv l[2K]} \int_{t_{i-K}}^{t_{i+K}} f_s^{(l,K)}\,\dd\alpha_s\int_{t_{i-K}}^{t_{i+K}} g_s^{(l,K)}\,\dd\beta_s \\
&\qquad \qquad \quad   + \int_{t_{*,l,K}}^{t} f_s^{(l,K)}\,\dd\alpha_s\int_{t_{*,l,K}}^{t} g_s^{(l,K)}\,\dd\beta_s
- \int_0^t f_s^{(l,K)}g_s^{(l,K)}\,\dd [\alpha,\beta]_s,
\end{split}
\notag
\end{equation}
for general semimartingales $\alpha$ and $\beta$, and families of functions $f = (f_{\cdot}^{(l,K)})_{1 \leq l \leq 2K}$ and $g = (g_{\cdot}^{(l,K)})_{1 \leq l \leq 2K}$ belonging to the classes $\FF$ and $\mathbb{G}$, respectively (see Definition~\ref{def:modulus}). Define also $\mathbb{Z}_{K}(f,\alpha,g,\beta) = (2K)^{-1}\sum_{l=1}^{2K}\mathbb{Z}_K^{(l)}(f,\alpha,g,\beta)_t$. Now, let the functions $f_s^{(l,K)}$ and $g_s^{(l,K)}$ be as defined in~\eqref{eq::f_func1} and \eqref{eq::g_func1}, respectively. Writing $(\mathbb{Z}_{K_2} - \mathbb{Z}_{K_1})(f,\alpha,g,\beta) = \mathbb{Z}_{K_2}(f,\alpha,g,\beta) - \mathbb{Z}_{K_1}(f,\alpha,g,\beta)$, and imposing the assumptions of Theorem~\ref{theorem::Th3.multivariate}, we can write 
\begin{equation} 
\begin{split}
\tsqc_{B,K_2,K_1}& (\widehat{\Theta}^n,\widehat{\Lambda}^n)_T 
 = \frac{(\gamma^2\mathbb{Z}_{K_2} - \mathbb{Z}_{K_1})(f,\theta,f,\lambda)_T}{\gamma^2 - 1} + [\theta,\lambda]_T\\
&  + \frac{n^{-\beta}( \gamma\mathbb{Z}_{K_2} - \mathbb{Z}_{K_1})(f,\theta,g,n^{\beta}M_n^{\lambda})_T}{(\gamma^2 -1 )K_{n,1}\Delta_n}\\
&  + \frac{n^{-\alpha}( \gamma\mathbb{Z}_{K_2} - \mathbb{Z}_{K_1})(g,n^{\alpha}M_n^{\theta},f,\lambda)_T}{(\gamma^2 -1 )K_{n,1}\Delta_n}\\  
& +\frac{n^{-(\alpha+\beta)}( \mathbb{Z}_{K_2} - \mathbb{Z}_{K_1})(g,n^{\alpha}M_n^{\theta},g,n^{\beta}M_n^{\lambda})_T}{(\gamma^2 - 1)K_{n,1}^2\Delta_n^2}  
+ o_p((K_{n,1}\Delta_n)^{1/2}).
\end{split}
\label{eq::tsqc_clt_decomp}
\end{equation}
From this expression we see that when $\alpha = \beta$ and $K_{n,1}\Delta_n$ is of order $n^{-\alpha}$, then all four terms in this sum will contribute the asymptotic variance of the sequence in~\eqref{eq::tsqc_sequence}.   

\begin{corollary}\label{corollary::tsqc_uncertainty} {\sc (Uncertainty of the {\tsqc})} Assume that $\theta^{(n)}$ and $\lambda^{(n)}$, as well as the error martingales $n^{\alpha}M_{n}^{\theta}$ and $n^{\beta}M_{n}^{\lambda}$ satisfy the conditions imposed on $\alpha^{(n)}$ and $\beta^{(n)}$ in Theorem~\ref{theorem::clt}. The conditions of Lemma~\ref{corr::TSQC_corollary} are also in force. Then the sequence in \eqref{eq::tsqc_sequence} is tight and P-UT. 
\end{corollary}
\begin{proof} Since $K_{n,1}\Delta_n$ and $K_{n,2}\Delta_n$ are of the same order, it follows from Theorem~\ref{theorem::clt} that all the error martingales $(K_j\Delta_n)^{-1/2}\mathbb{Z}_{K_j}$ for $j=1,2$ in \eqref{eq::tsqc_clt_decomp} are tight and P-UT. Since $K_{n,1}\Delta_n$ is of the same order as $n^{-\alpha\wedge \beta}$, the factors outside the last three terms in \eqref{eq::tsqc_clt_decomp} are either $o(1)$ or will tend to one. By assumption, the quadratic variations of the eight $(K_j\Delta_n)^{-1}[\mathbb{Z}_{K_j},\mathbb{Z}_{K_j}]_t$ have continuous limits (or tend to zero). Combined with the Lindeberg-condition of Theorem~\ref{theorem::clt} this entails that the $(K_j\Delta_n)^{-1/2}\mathbb{Z}_{K_j}$ are $C$-tight (see the proof of Theorem~\ref{app::stable_clt_jumps}). Sums of $C$-tight sequences are $C$-tight \citep[Corollary~VI.3.33, p.~353]{\JS}, and sums of sequences that are P-UT are P-UT \citep[VI.6.4, p.~377]{\JS}.        
\end{proof}

\begin{remark} Inspection of the proof of Theorem~\ref{theorem::clt} reveals that it is fully possible to derive a central limit theorem for the {\tsqc}-estimator. The key to such a proof is to show that quadratic covariations of the form (assuming that we are dealing with martingales)
\begin{equation}
\langle\mathbb{Z}_{K_{n,2}},\mathbb{Z}_{K_{n,1}}\rangle_t,
\notag
\end{equation}
converge in probability to a continuous limit when $K_{n,1}\Delta_n \to 0$, and $K_{n,2} = \gamma K_{n,1}$. That such convergence in probability occurs under the assumptions of Theorem~\ref{theorem::clt} can be shown by the same techniques used to prove said theorem, albeit at the cost of a somewhat heavier notational burden. For the present purposes, all we want is to show that the observed asymptotic variance can be applied to the {\tsqc}-estimator, and for that tightness and P-UT is sufficient.   
\end{remark}



\section{Volatility and intensity}\label{sec::volatility_and_intensity}
In this section we turn to the application that motivated the current paper, namely the estimation of the quadratic covariation between the volatility process of a continuous time semimartingale, and the intensity process of the observation times. 
When estimating parameters associated with a continuous time process that is only observed at discrete times, simplifying assumptions are often imposed on the relation between the observation times and the underlying process. The observation times are typically either taken as fixed and equidistant, or they are governed by a stochastic process postulated to be independent of the observable process (see e.g.,~\citet[Ch.~9]{ait2014high} for a discussion). We refer to both cases as {`}exogenous times{'}. In many settings the assumption of exogenous times is violated, the case of high-frequency financial data being, at least in some cases, a pertinent example. Decisions to buy or sell a given security may, in part, be determined by features of that security, and since it is only at the times at which transactions are conducted that we get a glimpse of the continuous processes ticking in the background (modulo microstructure noise), one would expect that the observation times may be correlated with transaction-igniting features of the underlying process.

In recent years, much progress has been made when the assumption of exogenous times is relaxed. In \citet{li2013volatility,li2014realized} the realised volatility estimator is studied in the presence of endogenous observation times, and it is shown that a {`}bias{'} term appears in the limiting distribution of this estimator. This {`}bias{'} term is of the same order of magnitude as the process tending (stably) to a normal limit, and is thus not a bias term in the traditional sense. The reasons for caring about it have to do with efficiency considerations, and not with the estimation being off-the-target in an expected value sense. 
\citet{jacod2019estimating} construct an estimator of the integrated volatility in the presence of microstructure noise, jumps, and endogenous times. Other papers have dealt with consistency and central limit theorems under irregular and random times \citep{renault2011causality,hayashi2011irregular,fukasawa2012central,
potiron2017estimation}. 
Common for all the above papers is that the endogeneity of the observation times comes about because the times depend on the efficient price process itself, as opposed to latent spot parameter processes governing the evolution of the efficient price process. The tools developed in Section~\ref{eq::general_setup} allow us to statistically study situations where the observation times might depend on underlying non-observable features of the efficient price process, such as its spot-volatility process, the associated volatility-of-volatility, the leverage effect, and so on. To assess the direction and magnitude of such correlations, we can use the {\tsqc}-estimator of~\eqref{eq::ts_thetalambdahat1}, and also a correlation estimator based on the {\tsqc}. In this section we first present some theory specific to the volatility-intensity covariance estimation, then, in Section~\ref{sec::simulations} we perform a simulation study to assess the finite sample behaviour of our estimators, while Section~\ref{sec::empirical} contains an empirical study of the Apple stock over $21$ trading days in January $2018$.

\subsection{A model for volatility-intensity covariance estimation}
For a given frequency of observations, indexed by $n\geq 1$, the succesive observations occur at times $0 = T_{n,0} < T_{n,1} < \cdots$, where $(T_{n,i})_{n\geq 1}$ is a sequence of finite stopping times. Define the sequence of counting processes $N_{n,t} = \sum_{i\geq 1}I\{T_{n,i}\leq t\}$. We are going to assume (in Condition~\ref{cond::lambda.lim}) that, for observation frequency $n$, the inter-observational lags $T_{n,i} - T_{n,i-1}$ are of the same order of magnitude as $1/n$, and moreover, that $n^{-1}N_{n,t}$ has a possibly random probability limit when $n$ goes to infinity (see \citet{li2014realized} and \citet{jacod2017statistical,jacod2019estimating} for similar constructions). Based on the $N_{n,T}$ observations of $X_t$, we form an estimator $\widehat{\Theta}_{t}^n$ of $\Theta_{t} = \int_0^t \theta_s \,\dd s$, where the spot parameter process $\theta_s$ is itself assumed to be a semimartingale, and assume that $\widehat{\Theta}_{t}^n$ is consistent for $\Theta_t$. In the following we think of $\theta_t$ as the spot-volatility process $\sigma^2$, and $\Theta_t$ as the integrated volatility $\int_0^t \sigma_s^2\,\dd s$. The counting process $N_{n,t}$ can be decomposed as $N_{n,t} = M_{n,t} + \Lambda_{n,t}$, in terms of a martingale $M_{n,t}$ and an increasing and predictable process $\Lambda_{n,t}$. We assume that the latter process is absolutely continuous, so that $\Lambda_{n,t} = \int_0^t\lambda_{n,s}\,\dd s$, and that $\lambda_{n,t}$, called the intensity process, is itself a semimartingale. The process we seek to estimate is then $[\theta,\lambda]_t$ over one or consecutive observation windows.

Since $X$ is followed over the finite interval $[0,T]$, where $T$ is fixed, our arguments are based on asymptotics as the observation frequency gets higher, that is $\max_{i \geq 1}\abs{T_{n,i} - T_{n,i-1}} \to 0$, so-called infill asymptotics. To let the number of observations $N_{n,T}$ tend to infinity, and at the same time get a finite expression for the limiting intensity of the observation times, we impose the following condition.
\begin{condition}\label{cond::lambda.lim} There is a non-negative semimartingale $\lambda_t$ such that $n^{-1}\Lambda_{n,t} \overset{p}\to \Lambda_{t} \coloneqq \int_0^t \lambda_s\,\dd s$, for all $t\in[0,T]$.
\end{condition}  
One may think of $1/n$ as proportional to the expected distance between two observation times, or $n$ as being proportional to the expected number of observations per period. The point is that Condition~\ref{cond::lambda.lim} allows us to develop asymptotic theory in terms of $N_{n,T}$ for the estimators we construct. This construction is similar to that previously employed by \citet{li2013volatility}; and by \citet[Assumption~(O-$\rho$, $\rho^{\prime}$), p.~82]{jacod2019estimating}. %


Suppose that the estimator $\widehat{\Theta}_t^n$ satisfies the decomposition in~\eqref{eq::GH.decomp1}, and that its error process martingale $M_{n,t}^{\theta}$ obeys Condition~\ref{cond::MGs}. We return to the assumptions on the edge effects in due time. Define  $\widetilde{\Lambda}_t^n = n^{-1}N_{n,t}$. The counting process $N_{n,t}$ simply counts the transactions and is hence observable, whereas $n$ is a non-observable abstraction introduced so that the asymptotic theory developed in the two preceding sections generalises to volatility-intensity estimation. This means that $\widetilde{\Lambda}_t^n$ is 
a rescaling of an estimator. For the (finite sample) empirical applications of our estimator, the index $n$ will turn out to be immaterial.

\begin{remark}\label{rmk:n} We emphasize that $n$ does not need to be observed for the developments in this section to be valid.
We need $n$ to exist in the sense of Condition \ref{cond::lambda.lim}, but otherwise $n$ is a notational convenience that permits us to state results more simply, and $n$ is in this sense always only a scaling. For example, $\widetilde{\Lambda}_t^n \overset{p}\to \Lambda_t$ can be restated as $\widehat{\Lambda}_{n,t} = \Lambda_{n,t}(1 + o_p(1))$, where
$\widehat{\Lambda}_{n,t} = N_{n,t}$. 
\end{remark}

Notice that there are no edge effects associated with $\widetilde{\Lambda}_t^n$, so \eqref{eq::GH.decomp1} becomes $\widetilde{\Lambda}_t^n = n^{-1}\Lambda_{t,n} + M_{n,t}^{\lambda}$, where $M_{n,t}^{\lambda} = n^{-1}(N_{n,t} - \Lambda_{t,n})$ is a martingale sequence. Moreover, as $n\to \infty$,  
\begin{equation}
n[M_{n}^{\lambda},M_{n}^{\lambda}]_t = n^{-1}N_{n,t} \overset{p}\to \Lambda_t,
\label{eq::qvMM_conv} 
\end{equation}
by Condition~\ref{cond::lambda.lim}. The convergence in~\eqref{eq::qvMM_conv} combined with the fact that $\Lambda_t$ is increasing and continuous, yield 
\begin{equation}
n^{1/2}M_{n,t}^{\lambda} \Rightarrow \int_0^t \lambda_s^{1/2}\,\dd W_s^{\prime}\quad \text{stably},
\notag
\end{equation} 
where $W_s^{\prime}$ is a Wiener process defined on an extension of the original probability space (see Theorem~\ref{theorem::th2.28_general} in Appendix~\ref{app::stable_clt_jumps}). Set $L_t^{\lambda} = \int_0^t \lambda_s^{1/2}\,\dd W_s^{\prime}$, and we have the first part of Condition~\ref{cond::MGs}. For Theorem~\ref{theorem::Th3.multivariate} to be applicable, the sequence of martingales $n^{1/2}M_{n,t}^{\lambda}$ must also be P-UT.
\begin{lemma}\label{lemma::M_is_PUT} Assume Condition~\ref{cond::lambda.lim}. Then $n^{1/2}M_{n,t}^{\lambda}$ is P-UT.
\end{lemma}
\begin{proof} That $n\langle M_{n}^{\lambda},M_{n}^{\lambda}\rangle_t = n^{-1}\int_0^t \lambda_{n,s}\,\dd s$ ensures that $n\langle M_{n}^{\lambda},M_{n}^{\lambda}\rangle_t$ is tight \citep[Proposition VI.3.26, p.~351]{\JS}. Being a counting process martingale, the jumps $n^{1/2}\abs{\Delta M_{n,t}^{\lambda}} \leq 1$, and \citet[Proposition VI.6.13, p.~379]{\JS} gives the result.   
\end{proof}
In the absence of edge effects on the part of $\widetilde{\Lambda}^{n}_t$, ${\rm QV}(\widehat{\Theta}^n,\widetilde{\Lambda}^{n})$ can be decomposed as (cf.~the decomposition in \eqref{eq::qv_before_rates} of Appendix~\ref{app::QCV.proof1}),
\begin{equation}
\qv(\widehat{\Theta}^n,\widetilde{\Lambda}^{n})
= \overline{\rm QV}(\widehat{\Theta}^n,\widetilde{\Lambda}^{n}) + O_p\big( {\rm QV}(\widetilde{\Lambda}^{n})^{1/2}R_{n,k}(\Theta)^{1/2} \big),
\label{eq::qv_Theta_Lambda}
\end{equation}
by the Cauchy--Schwarz inequality, where 
\begin{equation}
\overline{\qv}(\widehat{\Theta}^n,\widetilde{\Lambda}^{n}) 
= \qv(\Theta,\Lambda_n/n) + {\rm QV}(M^{\theta},\Lambda_n/n) + {\rm QV}(\Theta,M_n^{\lambda}) + \qv(M^{\theta},M_n^{\lambda}),
\notag
\end{equation}
and $R_{n,K}(\Theta) = K^{-1}\sum_{i=K}^{B-K} 
(e^{\theta}_{t_{i+K}} - e^{\theta}_{t_{i}} - (e^{\theta}_{t_{i}}- e^{\theta}_{t_{i-K}}))^2$. 

\begin{corollary}\label{cor::volatility_intensity} Suppose that $\widehat{\Theta}_{t}^n$ satisfies Condition~\ref{cond::MGs} in Section \ref{sec::consistency}, that $(\lambda_{n,t}/n)_{n\geq 1}$ is tight and P-UT, and that $e_{t}^{\theta}$ are $o_p((K_n\Delta_n)^{1/2}n^{-\alpha})$. Then, as $K_n\Delta_n \to 0$
\begin{equation}
\overline{\qv}(\widehat{\Theta}^n,\widetilde{\Lambda}^{n})_T 
 = 2[M_n^{\theta},M_n^{\lambda}]_{T-} + \frac{2}{3}(K_n\Delta_n)^2[\theta,\lambda_n/n]_{T-}
+ o_p\big( (K_n\Delta_n)^{2} \big) 
 + o_p\big(n^{-\alpha}n^{-1/2}\big).
\notag
\end{equation}
\end{corollary}
\begin{proof} By Lemma~\ref{lemma::M_is_PUT}, the sequence $\widetilde{\Lambda}_{t}^n = n^{-1}\Lambda_{n,t} + M_{n,t}^{\lambda}$ satisfies Condition~\ref{cond::MGs}, and the conditions on $(\lambda_{n,t}/n)_{n\geq 1}$ ensure that Theorem~7 in \citet{mykland2017assessmentsupplement} is applicable. The second part of Theorem~\ref{theorem::Th3.multivariate} then gives the result. 
\end{proof}
We have that $\qv(\widetilde{\Lambda}^{n}) = O_p(K_n\Delta_n + n^{-1/2})$, which via~\eqref{eq::qv_Theta_Lambda} shows how differing restrictions on the edge effects associated with the integrated volatility estimator give differing conclusions about ${\rm QV}(\widehat{\Theta}^n,\widetilde{\Lambda}^{n})$ (see the discussion in Appendix~\ref{app::QCV.proof1}). If we assume that the edge effects associated with $\widehat{\Theta}_t^n$ are $o_p((K_n\Delta_n)^{3/4}n^{-\alpha})$, which is not unrealistic when working with two-scales estimators and pre-averaged observations (see~\citet{zhang2005tale} and \citet{mykland2019algebra}), then the conclusion of Corollary~\ref{cor::volatility_intensity} is 
\begin{equation}
\begin{split}
\qv(\widehat{\Theta}^n,\widetilde{\Lambda}^{n})_T &
 = 2[M_n^{\theta},M_n^{\lambda}]_{T-} + \frac{2}{3}(K_n\Delta_n)^2[\theta,\lambda_n/n]_{T-}\\
& \qquad \qquad + O_p\big( (K_n\Delta_n)^{5/2} \big) 
 + O_p\big((K_n\Delta_n)^{1/2}n^{-\alpha}n^{-1/2}\big).
 \end{split}
\notag
\end{equation}   
Since $[\theta,\lambda_n/n]_{T-} \to_p [\theta,\lambda]_{T-}$, Corollary~\ref{corr::TSQC_corollary} entails that ${\rm TSQC}_{B,K_1,K_2}(\widehat{\Theta}^n,\widetilde{\Lambda}^n)$ is consistent. With the definitions in Remark~\ref{rmk:n}, ${\rm TSQC}_{B,K_1,K_2}(\widehat{\Theta}^n,\widehat{\Lambda}^n)$ is also consistent. 
Also, consider the process $\rho_t(\cdot,\cdot)$, given by 
\begin{equation}
\rho(\theta,\lambda)_t = \frac{[\theta,\lambda]_t}{([\theta,\theta]_t[\lambda,\lambda]_t)^{1/2}}.
\notag
\end{equation}   
Notice that $0 \leq \rho(\theta,\lambda)_t \leq 1$ for all $t$ due to the Kunita--Watanabe inequality~\citep[Theorem~II.25, p.~69]{protter2004}. For each $t$ we see that $\rho(\theta,\lambda_n)_t = \rho(\theta,\lambda_n/n)_t \overset{p}\to \rho(\theta,\lambda)_t$ by the continuous mapping theorem, which means that the coefficient $\rho(\theta,\lambda)_t$ can be consistently estimated using the estimators $\widehat{\Theta}_t^n$ and $\widehat{\Lambda}_t^n$, the latter simply defined as $\widehat{\Lambda}_t^n = N_{n,t}$. In particular, define 
\begin{equation}
\rho_{\tsqc}(\widehat{\Theta}^n,\widehat{\Lambda}^n)_{T} = \frac{{\rm TSQC}(\widehat{\Theta}^n,\widehat{\Lambda}^n)_{T}}{({\rm TSQC}(\widehat{\Theta}^n)_{T}{\rm TSQC}(\widehat{\Lambda}^n)_{T})^{1/2}} ,
\notag
\end{equation}
and note that $\rho_{\tsqc}(\widehat{\Theta}^n,\widehat{\Lambda}^n)_{T} = \rho_{\tsqc}(\widehat{\Theta}^n,\widetilde{\Lambda}^n)_{T}$, from which consistency of this estimator follows. When $\widehat{\Theta}^n$ and $\widetilde{\Lambda}^n$ have different convergence rates, as in Lemma~\ref{lemma::consistency1}, another consistent estimator for $\rho(\theta,\lambda)_t$ is $\qv_{B,K_2}(\widehat{\Theta}^n,\widetilde{\Lambda}^n)/({\rm TSQC}(\widehat{\Theta}^n)_{T}{\rm TSQC}(\widehat{\Lambda}^n)_{T})^{1/2}$. Since the speed at which $\qv_{B,K}$ converges is governed by the inferior convergence rate, there is, however, not that much to be gained in using this latter estimator, potentially apart from some less fine tuning of the $K_1$ and $K_2$ parameters. 
These two estimators of $\rho(\theta,\lambda)_t$ have a similar flavour to them, but are different from, the first-order correlation estimator introduced in \citet[Sections 3.1-3.2, pp.~899--903]{barndorff2004econometric}.

In Section~\ref{sec::simulations} we study the performance of $\rho_{\tsqc}(\widehat{\Theta}^n,\widehat{\Lambda}^n)_{t}$ on simulated data, and investigate its sensitivity to the choice of tuning parameters $K_1$ and $K_2$. Before proceeding to the simulations and the empirical application, we provide an example of a simple model satisfying the above assumptions. 

\begin{example}\label{example.model} {\sc (A volatility-intensity model)}. Suppose that we observe samples from the process $X_t = X_0 + \int_0^t\sigma_s\,\dd W_s$, where the spot volatility and the intensity follow CIR-processes \citep{cox2005theory} given by, 
\begin{equation}
\begin{split}
&\dd \sigma_t^2  = \kappa(\alpha - \sigma_t^2)\,\dd t + \gamma \sigma_t\,\dd Z_{t},\quad \sigma_{0}^2 = \alpha,\\
&\dd \lambda_{n,t}  = \beta_n(\xi_n - \lambda_{n,t})\,\dd t + \nu_n \lambda_{n,t}^{1/2}\,\dd B_{t}, \quad \lambda_{n,0} = \xi^n,
\label{eq::example.model}
\end{split}
\end{equation}
where $Z_t$ and $B_t$ are Wiener processes such that ${\rm corr}(Z_t,B_t)= \rho$, and $W_t$ is a Wiener process that may or may not be correlated with $Z_t$, thus allowing for a leverage effect, or $B_t$. 
The parameters $\kappa,\alpha$ and $\gamma$ as well as $\beta_n,\xi_n$ and $\nu_n$ are positive and we assume that the Feller condition \citep{feller1951two} holds for both the volatility and the intensity, that is $2\kappa\alpha \geq \gamma^2$, and $2\beta_n \xi_n \geq \nu_n^2$ for all $n \geq 1$. In this model, the dependency between $\sigma_t^2$ and $\lambda_{n,t}$ is introduced by the correlation between $Z_t$ and $B_t$. Suppose that $\xi^n = n\xi$, $\nu_n = \sqrt{n}\nu$ and that $0 < \beta \leq \beta_n \to \infty$ as $n\to \infty$. Then, for each $t\in [0,T]$, we have that $n^{-1}\Lambda_{n,t} \overset{p}\to \xi t$, and that 
\begin{equation}
n^{-1}[\sigma^2,\lambda_n]_t \overset{p}\to  [\sigma^2,\lambda]_t =  
\rho\gamma\nu \xi^{1/2} \int_0^t \sigma_s\,\dd s,
\label{eq::qcv.convergence}
\end{equation}
as $n \to \infty$. See Appendix~\ref{app::example.model} for details. 
\end{example}
In the next section the model of Example~\ref{example.model} is used as the basis for a simulation study.

\subsection{Simulations}\label{sec::simulations}
The data were simulated from the model presented in Example~\ref{example.model}. The initial  observations for the volatility and intensity processes were sampled from a Gamma distribution with parameters $(2\kappa\alpha/\gamma^2,2\kappa/\gamma^2)$ and a Gamma distribution with parameters $(2\beta_n\xi_n/\nu_n^2,2\beta_n/\nu_n^2)$ distribution, respectively. The parameter values were $\alpha=2.172,\kappa=2.345,\gamma = 1.000$ (volatility model), $\xi_n = n8.912,\beta_n = n^{1/4}0.169,\nu = \sqrt{n}1.000$, with $n = 40~000$. The microstructure noise was taken as additive on the efficient price and independent of the three underlying Brownian motions, that is, we observe 
\begin{equation}
Y_{t_i} = X_{t_i} + \eps_{t_i}, 
\notag
\end{equation}
where the $\eps_{t_i}$ were independent mean zero normals with standard deviation $0.0005$, independent of $W,Z$ and $B$. These three process were all Brownian motions, $W$ was independent of $Z$ and $B$, while $Z$ and $B$ were jointly Brownian with correlation $\rho = 0.912$. The data were simulated to mimic features of the actual Apple stock data that we analyse in Section~\ref{sec::empirical}. With $[0,T]$ one trading day ($6.5$ hours) the intensity function $\lambda_{n,t}$ is such that we have about $275~000$ observations of $Y_t$ per day. This is a common number of daily trades of a liquid stock such as that of Apple. As our estimator of the integrated volatility we used the Two-Scales Realised Volatility (TSRV) of \citet{zhang2005tale}, while $\widetilde{\Lambda}^n$ was used to estimate the cumulative intensity of the observation times. The TSRV we used is given by 
\begin{equation}
\widehat{\Theta} 
= \big\{\big(1 - \frac{K - J + 1/3}{N}\big)(K-J)\big\}^{-1}\big\{K[\bar{Y},\bar{Y}]^{(K)} - J[\bar{Y},\bar{Y}]^{(J)}\big\},
\label{eq::tsrv}
\end{equation}
where $\bar{Y}$ are pre-averaged observations, and $[\bar{Y},\bar{Y}]^{(K)} = K^{-1}\sum_{i=1}^{N-K}(\bar{Y}_{i+K} - \bar{Y}_{i})^2$, where $N$ are the number of {`}observations{'} of $\bar{Y}$, and $K$ is a tuning parameter chosen by the user (\citet[Eq.~(17),	p.~106]{mykland2019algebra} for this construction). Recall that the rescaling by $n$ 
is an abstraction that does not affect consistency, cf. Remark \ref{rmk:n}.
For each simulation we estimated the quadratic covariation $[\sigma^2,\lambda]_{T-}$, the coefficient $\rho(\sigma^2,\lambda)_{T-}$ and $\beta_{T-}$, the latter defined as $\beta_t = [\sigma^2,\lambda]_t/[\lambda,\lambda]_t$. All the quadratic (co-)variations were estimated using the {\tsqc}-estimator. Note, however, that the quadratic covariation $[\sigma^2,\lambda]$ could have been estimated directly using $\qv_{B,K}(\widehat{\Theta}^n,\widetilde{\Lambda}^n)$, this is because the TSRV of~\eqref{eq::tsrv} has convergence rate $n^{1/6}$ {to $n^{1/4}$ (depending on the degree of preaveraging)}, while $\widetilde{\Lambda}^n$ converges at the $n^{1/2}$ rate (see \citet[Theorem~4, p.~1402]{zhang2005tale} and Lemma~\ref{lemma::consistency1}). In Figure~\ref{fig::deviance_1} we have plotted the deviance of the estimates from the (random) estimands for various values of $K_1$, with $K_2 = 2K_1$ throughout.  

 \begin{figure}
\centering 
\includegraphics[scale=0.50,angle=270]{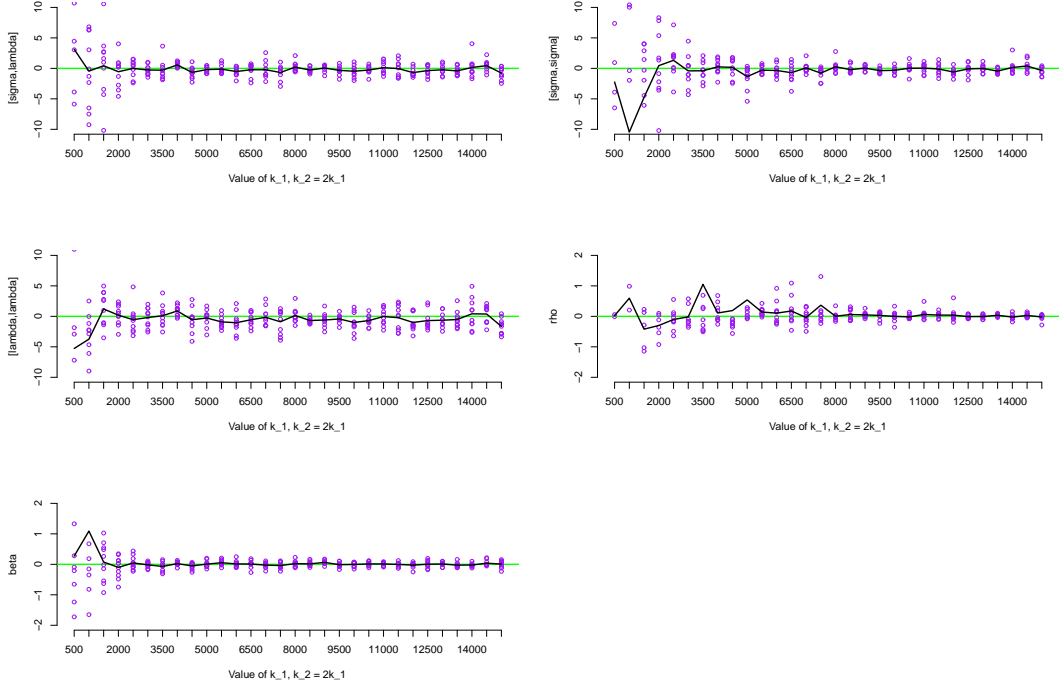} %
\caption{Values of $K_1$ on the $x$-axis ($K_2 = 2K_1$). Deviance of the estimate from the random truth, i.e.,~$\hat{\theta} - \theta$, on the $y$-axis. The wiggly lines are the means of the $10$ simulations performed for each value of $K_2$; the dots are the actual deviances; the straight lines indicate zero deviance. The TSRV-estimator (with $K=2$ and $J=1$, see~\citet[Eq.~(17),	p.~106]{mykland2019algebra}) was used to estimate the integrated volatility.}
\label{fig::deviance_1} 
\end{figure}

\subsection{An empirical application}\label{sec::empirical}
In the empirical study we analyse features of the Apple stock as traded	over a period of $21$ trading days in January 2018. All transactions registered in the U.S. National Market System conducted between 9:45 am - 3:45 pm Eastern Standard Time are included. The reason for choosing this window is to avoid abnormal trading activity during the opening and closing of the New York Stock Exchange, and to avoid those pre- and post-market hours during which the trading frequency is low \citep[p.~205]{wang2014estimation}. The Apple stock data is recorded down to the nanosecond ($10^{-9}$ seconds), and for the period under study the mean number of transactions over a trading day during the time window we use was $203\,924$, which is about nine transactions per second. After some data cleaning, the data was pre-averaged and the TSRV estimator of \citet{zhang2005tale} was used to estimate the integrated volatility. The cumulative intensity of the observation times was estimated by $10^{-6}N_{t}$, where $N_t$ counts the number of transactions conducted from 9:45 am to 9:45 am plus $t$. Besides making the plots more aesthetically pleasing, the number $10^{-6}$ plays no role.   

We used the {\tsqc}-estimator for daily estimation of the volatility-intensity covariance matrix and the two transformations thereof, $\rho(\sigma^2,\lambda)_t$ and $\beta_t$. The estimates of $\rho(\sigma^2,\lambda)_t$ are time-varying and lie between $0.5$ and $0.8$ for most of the days under study, indicating that the two processes are indeed correlated. To estimate the (pointwise) confidence bands of our TSQC-estimators we employed the Observed asymptotic variance of \citet{mykland2017assessment}. This estimator of the asymptotic variance is akin to the observed information in likelihood theory, and by using it, we avoid the difficulty of finding an explicit expression for the asymptotic variance. The applicability of the observed asymptotic variance is ensured by Corollary~\ref{corollary::tsqc_uncertainty}.       
     
\begin{figure}
\centering 
\includegraphics[scale=0.50,angle=270]{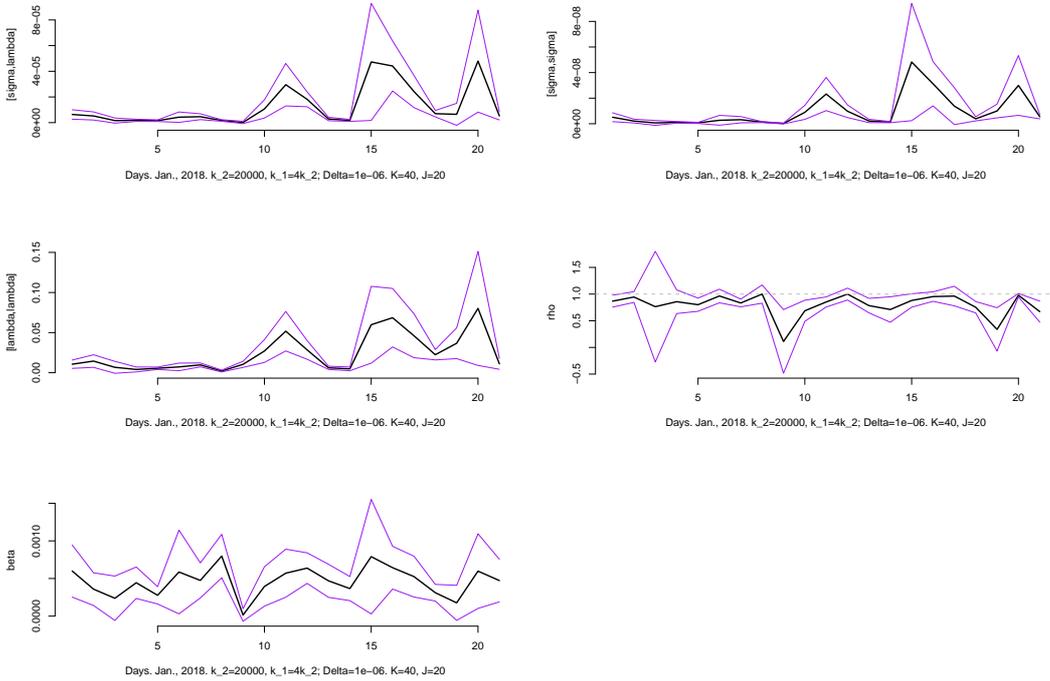} %
\caption{The Apple stock January 2.-31., 2018. Daily estimates of $[\sigma^2,\sigma^2]_{T}$, $[\sigma^2,\lambda]_{T}$ and $[\lambda,\lambda]_{T}$, as well as the parameters $\beta_T$ and $\rho_T$. The TSRV was used as the estimator of the integrated volatility. The purple lines are pointwise $95$ percent confidence bands computed using the observed asymptotic variance of~\citet{mykland2017assessment}, along with the delta method. In the plot with the daily estimates of $\rho_{T}$, the value $1$ is indicated by the dashed grey line.}  
\label{fig::empirical_5_RV} 
\end{figure}


\subsection{Using the volatility-intensity relationship to gain efficiency}
We have seen above that
\begin{equation}
\dd \theta_t = \beta_t \dd \lambda_t + \dd Z_t,\quad \text{and} \quad \dd\theta_{n,t} = \beta_{n,t} \dd \lambda_{n,t} + \dd Z_{n,t} ,
\notag
\end{equation}
where, in the latter equation, there is no normalization by $n$, hence the two equations are equivalent, and, once again, one can calculate as if $n$ were known. This is an ANOVA decomposition along the lines of~\cite{mykland2006anova}, but in this case, $\theta$ and $\lambda$ are unobserved. The process $\beta$ is estimated as above in this paper. The quantities $\theta$ and $\lambda$ can be estimated as spot
(instantaneous) quantities, as in~\cite{mz2new}. 

When microstructure is present in prices, but not in the observation times (as is the usual understanding), then $\widehat{\Lambda}_{n,t}$ has a faster rate of convergence than $\widehat{\Theta}_{n,t}$, and hence this is also true for $\widehat{\lambda}_{n,t}$ and $\widehat{\theta}_{n,t}$. The construction in~\cite{mz2new} uses $\widehat{\theta}_{n,t} = (\widehat{\Theta}_{n,t} - \widehat{\Theta}_{n,t-h_{n,\theta}})/h_{n,\theta}$, and similarly for $\widehat{\lambda}_{n,t}$, where $h_{n,\theta}$ and $h_{n,\lambda}$ are chosen to be (at least rate-) optimal, by the use of a variance-variance tradeoff. This leads to the rates for $\widehat{\theta}_{n,t}$ and $\widehat{\lambda}_{n,t}$ to be $n^{-1/8}$ and $n^{-1/4}$, respectively (when a rate optimal estimator of volatility is used, such as the S-TSRV which is used in this paper, or the multi-scale estimator of~\cite{lanmsrv06}, see also \cite{bibingermykland16} for the multivariate case and the connection to realised kernels, as well as the references therein. Finally, Lemma~\ref{lemma::consistency1} and Theorem~\ref{th::convergence.rates} provide for $\widehat{\beta}_{n,t}$ to have a rate of convergence of
$n^{-1/4}$, thus
\begin{equation}
\int_{0}^t \widehat{\beta}_{n,s}\,\dd \widehat{\lambda}_{n,s} - \int_{0}^t {\beta}_{n,s} \,\dd {\lambda}_{n,s} = O_p ( n^{-1/4} ) .
\label{eq:beta-lambda-est}
\end{equation}
If, as in our data, the residual $Z_t$ (or $\widehat{Z}_{n,t}$) is small, the question naturally occurs whether to prefer
$\widehat{\theta}_{n,t}$, with a low rate of convergence, or $\int_{0}^t \widehat{\beta}_{n,s} d \widehat{\lambda}_{n,s}$
with a much better rate of convergence, but with a bias of $Z_t$ (or $Z_{n,t}$). The conventional asymptotics-based answer to this question is that a slow convergence rate of $O_p( n^{-1/8} )$ is preferable to a much better convergence rate $O_p ( n^{-1/4} )$ to a limit with an $O_p(1)$ bias. In other words, pick $\widehat{\theta}_{n,t}$, even if $Z_t$ is small. 

This answer is uncomfortable, and has already caused some degree of argument in connection with volatility estimation, where there is an argument over whether intra-day estimators are always preferable, or whether to draw on longer time periods. Assumptions of stationarity will not help, and longer time periods are usually introduced by drawing on more highly specified models, such as ARCH and GARCH type models, going back to the
seminal papers of \cite{engle82} and  \cite{bollerslev86}.
There is a huge literature in this area, see, for example the survey by 
\cite{engleed95}.

Another path is to express {``}$Z_{n,t}$ is small{''} by a triangular array asymptotic regime whereby  $Z_{n,t} = o_p(1)$ as $n \to \infty$. 
Triangular array asymptotic regimes are often used close to a singularity, see, e.g.,~\citet{chanwei87} and~\citet{phillips1987towards} in the context of time series close to the unit root. In this context, it is often referred to as {`}local to unity asymptotics{'}. Under this regime, one can augment the estimate of $d\theta$ by adding an estimate of $\beta_t d \lambda_t$, giving rise to an estimate of the form 
\begin{equation}
\breve{\theta}_{n,t} = c_n \widehat{\theta}_{n,t} +  (1-c_n) \int_{0}^t \widehat{\beta}_{n,t}\,\dd \widehat{\lambda}_{n,t}.
\label{eq:improved-spot-vol}
\end{equation}
The tuning parameter $c_n$ should then be chosen to minimise the (random) mean squared error in $\breve{\theta}$, and in any case, $c_n \to 0$, thus improving the rate of convergence. A proper analysis of (\ref{eq:improved-spot-vol}) would require an assessment of the mean squared error of $\breve{\theta}_{n,t}$, which would presumably involve the estimation of $[Z_{n},Z_{n}]_t$, which brings us back to the ANOVA problem of~\citet{mykland2006anova}, but now with latent variables everywhere. This is beyond the scope of the present paper. 

If a reasonable solution can be found, similar methods may apply to a number of estimators that involve the estimation of spot volatility, such as leverage effect (Example~\ref{example::Leverage_effect_estimation}), volatility-of-volatility (in this paper, and also~\citet{vetter-volofvol} and~\citet{mykland2017assessment}), as well as regression, and ANOVA (\citet[Section~4.2, pp.~1424--1426]{mykland2009inference}, \citet[Section~4, pp.~268--273]{lanmodsel12}, \citet{reiss2015nonparametric}, and the references therein). 

\section{Conclusion}
This paper introduces a consistent estimator of the quadratic covariation between two non-observable spot-process semimartingales, derives the convergence rates of this estimator, and presents a central limit theorem for such estimators. The main theoretical contribution of the paper is this central limit theorem, a theorem that is applicable to a wide range of estimators based on triangular arrays of rolling quadratic covariations and second differencing of estimators of integrated spot processes. 

As recognised in much recent literature on estimation in high-frequency data, the assumption of exogenous observation times is often untenable, and one typically allows for dependency between the observation times and the price process. In this paper we have considered possible dependencies between the observation times and non-observable spot-processes associated with the price process, of which the spot volatility is a prime example. A simulation study shows that the estimators perform well with decent amounts of data. The empirical study of the Apple stock indicates that the observation times and the volatility process of this stock are positively correlated. 


\appendix 

\section{Notation and conditions}\label{app::appendixA}

We start by recalling some definitions from \cite{mykland2017assessment}.
\begin{definition}\label{def:order-in-prob-b}{\sc (Orders in Probability)}
For a sequence $\alpha^{(n)}_t$ of semimartingales, we say that $(\alpha^{(n)}_t) = O_p(1)$ if the sequence is tight, with respect to convergence in law relative to the Skorokhod topology on $\mathbb{D}$ \citep[Theorem VI.3.21, p.~350]{\JS}. For scalar random quantities, $O_p(\cdot)$ and $o_p( \cdot )$ are defined as usual, see, e.g.,~\cite[Appendix A]{pollard1984convergence}.
\end{definition}
\noindent

\begin{condition}\label{cond:modified-semimgs}
Let $\alpha^{(n)}_t$ and $\beta^{(n)}_t$ be sequences (in $n$) of semimartingales. Each of these sequences are (separately) assumed to be $O_p(1)$.
\end{condition}

\begin{definition} {\sc (Notation).}
\label{def:modulus}
The symbol $\mathbb{F}$ will refer to a collection of
nonrandom functions $f_{\cdot}^{(l,n)}$ c{\`a}dl{\`a}g on $[0,T]$, with $n \in \mathbb{N}$, and $l=1,{.}{.}{.}, 2K_n$, satisfying
\begin{equation}
| f_t^{(l,n)} |  \le 1 \mbox{ for all } t , l, \mbox{ and } n .
\notag
\end{equation}
Similarly, $\mathbb{G}$ will refer to a collection $g_t^{(l,n)}$ with the same size and properties.

Given $\mathbb{F}$ and $\mathbb{G}$,
set
\begin{equation}
\alpha^{(l,n)}_t = \int_0^t f_{s-}^{(l,n)} \,\dd \alpha^{(n)}_s \mbox{ and } \beta^{(l,n)}_t = \int_0^t g_{s-}^{(l,n)} \,\dd \beta^{(n)}_s \mbox{ for } l=1,{.}{.}{.}, 2K_n .
\notag
\end{equation}

For a random variable $X \in L^p(\Omega,\Falg,P)$ the norm is $\norm{X}_p = (\E\, \abs{X}^p)^{1/p}$. If $f(s)$ and $g(s)$ are defined on $[0,T]$ and $f(s) \leq c g(s)$ for all $0 \leq s \leq T$, for a fixed constant $c$, we write $f(s) \lesssim g(s)$. 

\end{definition}

\begin{condition} {\sc (Conditions for Rate-of-Convergence Statements and CLT)}
\label{cond:rate}  
The sequence of semimartingales $\alpha^{(n)}_t$, possibly defined on a sequence of filtrations $(\Falg_t^n)_{0 \le t \le T}$, is said to satisfy this condition if it can be written as $\alpha^{(n)}_t = \alpha^{(n)}_0 + \alpha^{(n,{\rm MG})}_t + \int_0^t b_s^{(n)} ds$, where for each $n$, $\alpha^{(n,{\rm MG})}_t$ is a square integrable martingale with predictable quadratic variation $\langle\alpha^{(n,{\rm MG})},\alpha^{(n,{\rm MG})}\rangle_t$ that is absolutely continuous, and $b^{(n)}_t$ and $\dd\langle \alpha^{(n,{\rm MG})},\alpha^{(n,{\rm MG})} \rangle_t/\dd t$ are locally bounded uniformly in $n$. 
\end{condition}
A single semimartingale $\alpha_t$ is said to satisfy Condition \ref{cond:rate} if the above is satisfied for the constant sequence $\alpha_t = \alpha^{(n)}_t$. Note also that Condition~\ref{cond:rate} implies that each $\alpha^{(n)}_t$ is an It{\^o}-semimartingale (see \citet[Eq.~(4.4.1), p.~114]{jacod2012discretization}).

\begin{definition}\label{assumption3} A processes $\xi_t$ is locally continuous in mean square if 
\begin{equation}
\sup_{0 \leq |t-s|\leq \delta }\E\,(\xi_{t} - \xi_{s})^2 \to 0,\quad\text{as}\quad \delta \to 0, 
\notag
\end{equation}
provided $t\vee s \leq \tau_n$, where $\tau_n$ is a stopping time such that $P(\tau_n = T )\to 1 $ as $n \to \infty$.
\end{definition}  

For the proof of Theorem~\ref{theorem::clt} contained in Appendix~\ref{app::CLT_proof} we need to be more specific about the construction of the probability space on which the sequence of processes $\alpha^{(n)}$, $\beta^{(n)}$, as well as potentially stochastic spot-processes related to these two, are defined. Since the result of Theorem~\ref{theorem::clt} is a {\it stable} convergence result, we need everything (except, possibly, microstructure noise) to be defined on the same probability space. Let $(\Omega,\Falg,\FF,P)$, with $\FF = (\Falg_t)_{0 \leq t \leq T}$ be a filtered probability space on which the processes are defined, and for each $n$ let $\FF^n =  (\Falg_t^n)_{0 \leq t \leq T}$ be a filtration on $(\Omega,\Falg)$. 

\begin{condition}\label{cond::clt1} A filtration $(\Falg_t)_{0 \leq t \leq T}$ on $(\Omega,\Falg)$ is said to satisfy the current condition if it is generated by $(\mu,W^{(1)},W^{(2)},\ldots)$ where $\mu$ is a Poisson random measure with deterministic compensator $\nu$ that is absolutely continuous as a function of time, and $W^{(1)},W^{(2)},\ldots$ are independent one-dimensional Wiener processes. 
\end{condition}
\begin{condition}\label{cond::clt2} For any finite family of $\Falg_t$-adapted bounded martingales $(X_1,\ldots,X_p)$ there is a sequence of $\Falg_t^n$-adapted martingales $(X_1^n,\ldots,X_p^n)$ such that $(X_1^n,\ldots,X_p^n) \to_p (X_1,\ldots,X_p)$.
\end{condition}
By \citet[Theorem~14.5.7, p.~360]{cohen2015stochastic} Condition~\ref{cond::clt1} is sufficient to represent the local martingales encountered in Theorem~\ref{theorem::clt}. Importantly, any martingale $X$ (resp.~$X^n$) adapted to $\FF$ (resp.~$\FF^n$) has a predictable quadratic variation process $\langle X,X\rangle$ (resp.~$\langle X^n,X^n\rangle$) that is absolutely continuous with respect to Lebesgue measure. 

\section{A stable central limit theorem for c{\`a}dl{\`a}g martingales}\label{app::stable_clt_jumps}
We find the following theorem and its corollaries to be convenient in applications. It is a generalisation of Theorem~2.28 in~\citet[p.~152]{myklandzhang2012} (originally stated in~\citet{zhang2001phd}), and is a special case of a theorem found in~\citet[ch.~IV.7]{jacod2003limit}, but with a different and perhaps more accessible statement and proof. The proof of the present theorem employs techniques from the proofs of both these earlier theorems. The formulation of our theorem also gives rise to Corollary~\ref{lemma::condition_ii}. This corollary provides alternative Lindeberg type conditions that are easier to check. 

We have a filtered probability space $(\Omega,\Falg,\FF,P)$, where $\FF = \{\Falg_t\}_{0\leq t\leq T}$. For each $n$, we have a filtration $\FF^n = \{\Falg_t^n\}_{0\leq t \leq T}$ and a $\Falg_t^n$-adapted square integrable martingale $M^n = \{M_t^n \colon 0 \leq t \leq T\}$. This is the martingale that we wish to show that converges stably in distribution. 

We assume that $\Falg$ is countably generated, that is, $\Falg = \sigma(A_1,A_2,\ldots)$ for a countable sequence $A_1,A_2,\ldots$ in $\Omega$. There is then a sequence $\{Y_m\}_{m\geq 1}$ of random variables that is dense in $L^1(\Omega,\Falg,P)$ \citep[Theorem~3, p.~382]{kolmogorov1970}. Set $N_{t}^m = \E\,(Y_m\mid \Falg_t)$, which is then a bounded martingale on $(\Omega,\Falg,\FF,P)$. Here are two results that can be found in \citet[pp.~114--115]{jacod1979calcul}, and also stated in \citet{jacod1997continuous}.
\begin{itemize}
\item[(a)] Every bounded martingale is the limit in $L^2$, uniformly in time, of a sequence of stochastic integrals with respect to a finite number of $N_{t}^m$.  
\item[(b)] If $\Galg_t$ is the smallest filtration with respect to which $(N_{t}^m)_{m\geq 1}$ is adapted, then $\Galg_t = \Falg_t$ up to $P$-null sets. 
\end{itemize}
The countably many $\Falg_t$-adapted bounded martingales $N_{t}^m$ play a role similar to the Wiener processes appearing in Condition~2.26 in~\citet[p.~151]{myklandzhang2012}. 


\begin{theorem}\label{theorem::th2.28_general} 
Assume Condition~\ref{cond::clt2}. Let $M^n = \{M_t^n \colon 0 \leq t \leq T\}$ be a sequence of locally square integrable martingales on $(\Omega,\Falg,P)$, adapted to $\Falg_t^n$ for each $n$. Suppose that there is an $\Falg_t$-adapted process $f_t$ such that
\begin{itemize}
\item[{\rm(i)}] $\langle M^n,M^n\rangle_t \to_p \int_0^t f_s^2\,\dd s$ for all $t$;
\item[{\rm(ii)}] $\int_{|x| > \eps}x^2\nu^n([0,T]\times \dd x)\to_p 0$ for all $\eps > 0$;
\item[{\rm(iii)}] $\langle M^n,X^n\rangle_t \to_p 0$ for all $t$ and all bounded martingales $X$ on $(\Omega,\Falg,\FF,P)$.
\end{itemize}
Then $M^n$ converges stably in distribution to $M_t = \int_0^t f_s \,\dd W_s$, where $W_s$ is a Wiener process defined on an extension of the original probability space.
\end{theorem}

\begin{proof} 
Convergence in probability implies convergence in distribution, so (i) implies that $\langle M^n,M^n\rangle_t \to_d \int_0^t f_s^2\,\dd s$ in the sense of finite dimensional distributions. Combining this with the facts that $\langle M^n,M^n\rangle_t$ is a non-decreasing process and has a non-decreasing and continuous limit, Theorem VI.3.37 in~\citet[p.~354]{jacod2003limit} yields process convergence of $\langle M^n,M^n\rangle$ to $\int_0^{\cdot} f_s^2\,\dd s$. The sample paths $t\mapsto \int_0^t f_s^2\,\dd s$ are continuous, so $\langle M^n,M^n\rangle$ is $C$-tight~\citep[Def.~3.25, p.~351]{jacod2003limit}, implying that $M^n$ is tight \citep[Theorem VI.4.12, p.~358]{jacod2003limit}. Condition~(ii) implies that $\sup_{s\leq T}|\Delta M_t^n| \to_p 0$, combined with the tightness of $M^n$ this implies that $M^n$ is $C$-tight~\citep[Lemma~VI.4.22, p.~360, and Theorem VI.3.26(iii), p.~351]{jacod2003limit}. Recall that $N_{t}^m =\E\, (Y^m\mid \Falg_t)$, and denote $\mathscr{N} = (N^m)_{m \geq 1}$. By Condition~\ref{cond::clt2} there is a sequence $\mathscr{N}^n = (N_1^n,N_2^n,\ldots)$, such that $\mathscr{N}^n \to_p \mathscr{N}$. Since $M^n$ is $C$-tight and $\mathscr{N}^n = (N_1^n,N_2^n,\ldots)$ is tight by Condition~\ref{cond::clt2}, Corollary~3.33 in~\citet[p.~353]{jacod2003limit} gives that $(M^n,\mathscr{N}^n)$ is tight. By Prokhorov{'}s theorem (see e.g.~\citet[Theorem 2.5(ii), p.~8]{vandervaart1998asymptotic}), this tightness entails that we can for any subsequence $n_k$ find a further subsequence $n_{k_j}$ such that 
\begin{equation}
(M^{n_{k_j}},\mathscr{N}^{n_{k_j}})\Rightarrow(M,\mathscr{N}).
\label{eq::conv1}
\end{equation}
For each $n$, write 
\begin{equation}
M_t^n 
= M_t^{n,b}  + xI\{|x| > 1\}\star(\mu^n - \nu^n)_t,
\label{eq::key_decomp1} 
\end{equation} 
in terms of the measure $\mu^n$ associated with the jumps of $M^n$, and its compensator $\nu^n$, and where $M_t^{n,b}$ is a local martingale with bounded jumps. For the decomposition in~\eqref{eq::key_decomp1}, see e.g., \citet[Eq.~(2.1.10), p.~29]{jacod2012discretization} and use that $M^n$, their $X$, is a martingale; or see Proposition~II.2.29 in~\citet[p.~82]{jacod2003limit}, and the fact that their $A \equiv 0$ in the martingale case. Since $xI\{|x| > 1\}\star\nu^n$ is the predictable compensator of $xI\{|x| > 1\}\star\mu^n$, it follows from Lenglart{'}s inequality \citep[Lemma~3.30(a), p.~35]{jacod2003limit} and Condition~(ii) that $xI\{|x| > 1\}\star\mu_t^n\to_p 0$ for all $t \in [0,T]$, thus 
\begin{equation}
\sup_{t \leq T} |M_t^n - M_t^{n,b}| \overset{p}\to 0. 
\label{eq::conv_in_prob}
\end{equation}
But~\eqref{eq::conv_in_prob} must also hold for any subsequence, so~\eqref{eq::conv1} and the Cram{\'e}r--Slutsky rules entail that $(M^{n_{k_j},b},\mathscr{N}^{n_{k_j}})$ converges in law to $(M,\mathscr{N})$. Since $M^{n,b}$ has bounded jumps $|\Delta M^{n,b}|\leq 1$, Theorem IX.1.17 in~\citet[p.~526]{jacod2003limit} gives that $M$ is a local martingale with respect to the filtration generated by $\mathscr{N}$ (hence the importance of fact (b), and where we use that Theorem IX.1.17 extends from the finite to the countable case, see~\citet[p.~586]{jacod2003limit}). 

We now want to show that $M^n$ is P-UT, because that will ensure joint convergence of $(M^{n_{k_l}},[M^{n_{k_l}},M^{n_{k_l}}])$. Let $H^n \in \mathscr{H}^n$, where $\mathscr{H}^n$ as well as the elementary stochastic integral $H^n\cdot M_t^n$ are as defined in~\citet[p.~377]{\JS}. Then $\E\, \abs{H^n\cdot M_t^n}^2 = \E\, (H^n)^2\cdot [M^n,M^n]_t \leq \E\, [M^n,M^n]_t = \E\, \langle M^n,M^n\rangle_t$. So by Lenglart{'}s inequality \citet[Lemma I.3.30(a), p.~35]{jacod2003limit}, for every $t$, and for any $H^n \in \mathscr{H}$, and for any $a,\eta > 0$,
\begin{equation}
P(\abs{H^n\cdot M_t^n}_t \geq a) \leq P(\sup_{0 \leq t \leq T}\abs{H^n\cdot M_t^n}_t \geq \eps)
\leq \frac{\eta}{a^2} 
+ P(\langle M^n,M^n\rangle_t \geq \eta).
\notag
\end{equation}
But since $\langle M^n,M^n\rangle_t$ is tight, this shows that $M^n$ is P-UT. Since $M^n$ is P-UT, Theorem VI.6.26 in~\citet[p.~384]{jacod2003limit} gives that $(M^{n_{k_l}},[M^{n_{k_l}},M^{n_{k_l}}])$ converges in law to $(M,[M,M])$; from continuity of $M$ we get that $[M,M] = \langle M,M\rangle$ \citep[Theorem~I.4.52, p.~55]{jacod2003limit}; and by Condition (i), $\langle M,M\rangle_t = \int_0^t f_s\,\dd s$.  

Assume without loss of generality that $f_s > 0$ (see~\citet[p.~152]{myklandzhang2012}), and set $W_t = \int_0^t f_s^{-1/2}\,\dd M_s$. Then $\langle W,W\rangle_t = t$ and  by Condition~(iii) $\langle W,X\rangle_t = \int_0^t f_s^{-1/2}\langle M,X\rangle_t = 0$ for any bounded martingale $X$. L{\'e}vy{'}s theorem \citep[p.~102]{jacod2003limit} then gives that $W$ is a Wiener process. Since $W$ is independent of $\Falg$ by Condition~(iii), we can realise $W$ on the extension $\widetilde{\Omega} = \Omega\times C[0,T]$, $\widetilde{\Falg} = \Falg\otimes \Balg$, $\widetilde{\Falg}_t = \cap_{s>t}\Falg_{s} \otimes \Balg_{s}$, $\widetilde{\Pr}(\omega,\dd x) = \pr(\dd\omega)Q(\omega,x)$, where $C[0,T]$ is the space of all continuous functions on $[0,T]$, and for $\omega$ fixed, $Q(\omega,\dd x)$ is the Wiener measure. Then $W(\omega,x)$ is a Wiener process for each $\omega$, and $M_t(\omega,x) = \int_0^t f_s(\omega)\,W(\omega,\dd x)$ is a continuous process on the extension, orthogonal to all bounded martingales on $(\Omega,\Falg,\FF,P)$, and $\langle M,M\rangle_t = \int_0^t f_s^2\,\dd s$ is $\Falg$-measurable by Condition~(i). Thus, $M$ is an $\Falg$-conditional Gaussian martingale on the extension. This proves the theorem for a subsequence $n_{k_j}$, but since the subsequence was arbitrary, the claim of the theorem follows (see corollary on p.~337 in~\citet{billingsley1995}, or \citet[Theorem~2.6, p.~20]{billingsley1999}).
\end{proof}

\begin{corollary}\label{lemma::condition_ii} Assume {\rm(i)} and {\rm(iii)} of Theorem~\ref{theorem::th2.28_general}. If Condition~{\rm (ii)} in that theorem is replaced by one of the following conditions,
\begin{itemize}
\item[{\rm(ii)}$^{\prime}$] $\E\, \sum_{s\leq t}|\Delta M_s^n|^2I\{|\Delta M_s^n| \geq \eps\} \to 0$ for all $\eps > 0$ and for all $t$;
\item[{\rm(ii)}$^{\prime\prime}$] $\sup_{t\leq T}|\Delta M_t^n| \to_p 0$, and $\E\,\sup_{t\leq T}|\Delta M_t^n|^2 < \infty$ for all $n$;
\end{itemize} 
the conclusion of Theorem~\ref{theorem::th2.28_general} still holds. 
\end{corollary} 
\begin{proof} For (ii)$^{\prime}$: By Proposition~II.1.28 (p.~72) and Theorem~I.3.17 (p.~32) in~\citet{jacod2003limit}, we have that
\begin{equation}
\E\, \sum_{s\leq t}|\Delta M_s^n|^2I\{|\Delta M_s^n| \geq \eps\} = \E\, \int_{|x|\geq \eps}|x|^2\mu^n([0,T]\times \dd x) = \E\, \int_{|x|\geq \eps}|x|^2\nu^n([0,T]\times \dd x),
\notag
\end{equation} 
which proves that (ii)$^{\prime}$$\Rightarrow$(ii). For (ii)$^{\prime\prime}$: We must show that (ii)$^{\prime\prime}$ implies~\eqref{eq::conv_in_prob}. Using the the triangle inequality and the fact that $\nu^n$ is a (non-negative) measure 
\begin{equation}
\begin{split}
|M_t^n - M_t^{n,b}| 
& \leq \sum_{s\leq t}|\Delta M_s^n|I\{|\Delta M_s^n|>1\} + |x|I\{|x| > 1\}\star \nu_t^n\\
& \leq \sup_{s\leq t}|\Delta M_s^n| \sum_{s\leq t}I\{|\Delta M_s^n|>1\} + |x|I\{|x| > 1\}\star \nu_t^n. 
\end{split}
\label{eq::can_we_show_this}
\end{equation}
Since $\langle M^n,M^n\rangle_t$ is tight, $M^n$ is P-UT, and we have that $\sum_{s\leq t}I\{|\Delta M_s^n|>1\} = O_p(1)$ for all $t > 0$~\citep[Theorem~VI.6.16, p.~380]{jacod2003limit}, so the first term on the right in~\eqref{eq::can_we_show_this} tends to zero in probability by the Cram{\'e}r--Slutsky rules. For the second term on the right, since $|x|I\{|x| > 1\}\star \nu_t^n$ is the predicable compensator of  the adapted process $\sum_{s\leq t}|\Delta M_s^n|I\{|\Delta M_s^n|>1\} = |x|I\{|x| > 1\}\star \mu_t^n$, Lenglart{'}s inequality~\citep[Lemma~I.3.30(b), p.~35]{jacod2003limit} gives that for all $\eps,\eta > 0$, 
\begin{equation}
\begin{split}
\pr( |x|I\{|x| > 1\}\star \nu_t^n \geq \eps) 
 \leq \frac{1}{\eps}\big(\eta + \E\, \sup_{s\leq t} |\Delta M_s^n|  \big) + P( |x|I\{|x| > 1\}\star \mu_t^n \geq \eta).
\end{split}
\notag
\end{equation}
As we saw above $P( |x|I\{|x| > 1\}\star \mu_t^n \geq \eta) \to 0$. For all $\sigma > 0$, by H{\"o}lder{'}s inequality 
\begin{equation}
\begin{split}
\E\, \sup_{s\leq t} |\Delta M_s^n|  & \leq \E\, \sup_{s\leq t} |\Delta M_s^n|I\{|\Delta M_s^n| \geq \sigma\} + \sigma\\ 
& \leq (\E\,\sup_{s\leq t} |\Delta M_s^n|^2)^{1/2}\,P(\sup_{s\leq t} |\Delta M_s^n| \geq \sigma)^{1/2} + \sigma \to \sigma,
\end{split}
\notag
\end{equation}
as $n \to \infty$. Since $\eps,\eta,\sigma$ were arbitrary, $|x|I\{|x| > 1\}\star \nu_t^n$ converges in probability to zero for all $t > 0$.     
\end{proof}

\begin{remark}\label{remark::cont_to_PUT} In the proof of Theorem~\ref{theorem::th2.28_general} we use that if $M^n$ is a sequence of local square integrable martingales, that $\langle M^n,M^n\rangle_t \to \langle M,M\rangle_t$ for all $t$ as $n \to \infty$, and that $\langle M,M\rangle_t$ is continuous, then $M^n$ is P-UT. This is a useful implication that, perhaps because it is deemed obvious, is not spelled out explicitly in \citet[ch.~VI.6]{jacod2003limit}. Using this implication, an immediate corollary to Proposition~6 in \citet[p.~12]{mykland2017assessmentsupplement} is: If $M^n$ converges in law to $M$, and $\langle M^n,M^n \rangle_t \to_p V$, with $V$ being continuous, then $M^n$ converges $\Galg$-stably in law, where $\Galg = \sigma(V)$. For several other results associated with the P-UT property, see~\citet[Appendix D]{mykland2017assessmentsupplement}.  
\end{remark}

\section{Proof of the claims in Example~\ref{example.model}}\label{app::example.model}
Assume that $\xi^n = n\xi$, $\nu_n = \sqrt{n}\nu$ and that $0 < \beta \leq \beta_n \to \infty$ as $n\to \infty$. Then, for each $t\in [0,T]$,
\begin{equation}
\frac{1}{n}\Lambda_{n,t} \overset{p}\to \xi t,\quad \text{and}\quad
\frac{1}{n}[\sigma^2,\lambda_n]_t \overset{p}\to 
\rho\gamma\nu \xi^{1/2} \int_0^t \sigma_s\,\dd s,
\label{eq::claims_in_example}
\end{equation} 
as $n \to \infty$. We now prove~\eqref{eq::claims_in_example}: The expectation of the intensity is $\E\,\lambda_{n,t} = \xi_n$, and
\begin{equation} 
\begin{split}
\E\, \big|\frac{1}{n}(\lambda_{t,n} - \xi_n)\big|^2 & 
= \frac{1}{n^2}\,\E\, \big|\nu_n\int_0^t \lambda_{n,s}^{1/2}e^{-\beta_n(t-s)}\,\dd B_s\big|^2
= \frac{\nu_n^2}{n^2}\,\E\, \int_0^t \lambda_{n,s}e^{-2\beta_n(t-s)}\,\dd s\\
& = \frac{\nu_n^2}{n^2}\int_0^t \xi_{n}e^{-2\beta_n(t-s)}\,\dd s
= \frac{\nu_n^2}{n^2}\frac{\xi_n}{2\beta_n}\big(1 - e^{-2\beta_n t}\big). 
\end{split}
\label{eq::chebyshev.step1}
\end{equation}
Note that
\begin{equation}
\begin{split}
\E\,|\lambda_{t,n}/n|^2 & = \E\,|\lambda_{t,n}/n - \xi + \xi|^2
 = \E\,|\lambda_{t,n}/n - \xi|^2 + 2\,\E\,|\lambda_{t,n}/n - \xi|\xi+ \xi^2\\
& \leq \E\,|\lambda_{t,n}/n - \xi|^2 + 2\,\big(\E\,|\lambda_{t,n}/n - \xi|^2\big)^{1/2}\xi+ \xi^2,
\end{split}
\notag
\end{equation}
and from~\eqref{eq::chebyshev.step1}, $\E\,|\lambda_{t,n}/n - \xi|^2 = \nu^2(2\beta_n)^{-1}\big(1 - e^{-2\beta_n t}\big)$, from which it follows that for each $t$, $\sup_n \E\,|\lambda_{t,n}/n|^2 <\infty$. By Chebyshev{'}s inequality we get that for each $t$ the sequence of random variables $\{\lambda_{t,n}/n\}_{m\geq 1}$ are uniformly integrable (see Eq.~(25.13) in \citet[p.~338]{billingsley1995}). 
Moreover, from the above we see that $\E\,|\lambda_{t,n}/n| \leq (\nu^2/\beta + \xi^2)^{1/2}$ for all $t$ and $n$, and the right hand side is trivially integrable on $[0,T]$. Hence, the sequence of stochastic processes $\{\lambda_{n,s}/n\}_{n\geq 1}$ satisfies the conditions of \citet[Proposition II.5.2, p.~85]{andersen1993statistical}, and the first part of~\eqref{eq::claims_in_example} follows. For the second part we have that 
\begin{equation}
\begin{split}
\E\,|\sigma_t \lambda_{n,t}^{1/2}/\sqrt{n}|^2 & = \E\,\sigma_t^2 \lambda_{n,t}/n
= \E\,|(\sigma_t^2-\alpha + \alpha)| |(\lambda_{n,t}/n - \xi + \xi)|\\
& = \E\,|(\sigma_t^2-\alpha)(\lambda_{n,t}/n - \xi)| 
+ \E\,|(\sigma_t^2-\alpha)|\xi + \E\,|(\lambda_{n,t}/n - \xi)|\alpha + \alpha\xi,
\end{split}
\notag
\end{equation}  
which by three applications of H{\"o}lder{'}s inequality and the It{\^o} isometry is seen to be bounded by a constant, hence $\sup_n \E\,|\sigma_t \lambda_{n,t}^{1/2}/\sqrt{n}|^2 <\infty$, and uniform integrability of the random variables $\sigma_t \lambda_{n,t}^{1/2}/\sqrt{n}$ follows. Since $\E\,|\sigma_s\lambda_{n,s}^{1/2}/\sqrt{n}| \leq (\E\,|\sigma_s\lambda_{n,s}^{1/2}/\sqrt{n}|^2)^{1/2}$ for all $s$ and $n$, and a constant is integrable on $[0,T]$, so the second part of~\eqref{eq::claims_in_example} follows by the same argument as above.

\section{Notes on Theorem~\ref{theorem::Th3.multivariate}}\label{app::QCV.proof1}
The proof follows with trivial adjustments from \citet[Theorem~3, p.~208]{mykland2017assessment}. Note that the convergence rates change due to our Theorem~\ref{th::convergence.rates}. Recall the setup in \eqref{eq::GH.decomp1}, that is $\widehat{\Theta}_{(s,t]} - \Theta_{(s,t]} = M_{n,t}^{\theta} - M_{n,s}^{\theta} + e_{n,t}^{\theta} - e_{n,s}^{\theta}$ and $\widehat{\Lambda}_{(s,t]} - \Lambda_{(s,t]}  = M_{n,t}^{\lambda} - M_{n,s}^{\lambda} + e_{n,t}^{\lambda} - e_{n,s}^{\lambda}$. \citet[Theorem~3, p.~208]{mykland2017assessment} and the convergence rates from Theorem~\ref{th::convergence.rates} give 
\begin{equation}
\begin{split}
\qv_{B,K}(\widehat{\Theta},\widehat{\Lambda}) & = \overline{{\rm QV}}_{B,K}(\widehat{\Theta},\widehat{\Lambda}) + R_{n,k}(\Theta,\Lambda) + O_p\big((K_n\Delta_n + n^{-\alpha}) R_{n,k}(\Lambda)^{1/2} \big)\\ 
& \qquad \qquad \qquad \qquad + O_p\big((K_n\Delta_n + n^{-\beta}) R_{n,k}(\Theta)^{1/2} \big),
\end{split}
\label{eq::qv_before_rates}
\end{equation}
where $R_{n,k}(\Theta) = R_{n,k}(\Theta,\Theta)$ and 
\begin{equation}
R_{n,k}(\Theta,\Lambda)  = \frac{1}{K}\sum_{i=K}^{B-K}
(e_{n,t_{i+K}}^{\theta} - e_{n,t_{i}}^{\theta} - ( e_{n,t_{i}}^{\theta} - e_{n,t_{i-K}}^{\theta}))(e_{n,t_{i+K}}^{\lambda} - e_{n,t_{i}}^{\lambda} - ( e_{n,t_{i}}^{\lambda} - e_{n,t_{i-K}}^{\lambda})),
\notag
\end{equation}
while $\overline{\qv}_{B,K}(\widehat{\Theta},\widehat{\Lambda})$ is given by
\begin{equation}
\begin{split}
\overline{\qv}_{B,K}(\widehat{\Theta},\widehat{\Lambda})  & = 
2[M_n^{\theta},M_n^{\lambda}]_{T-} + \frac{2}{3}(K_n\Delta_n)^2\bigg( 1 - \frac{1}{K_n^2}\bigg)[\theta,\lambda]_{T-} + O_p\big(n^{-(\alpha+\beta)}(K_n\Delta_n)^{1/2}\big)\\
& \qquad  + \Delta_n^2\int_0^{T-}\bigg\{\bigg(\frac{t^{*}(s) - s}{\Delta_n} \bigg)^2 + \bigg(\frac{s - t_{*}(s)}{\Delta_n} \bigg)^2   \bigg\}\,\dd [\theta,\lambda]_{s} + O_p\big( (K_n\Delta_n)^{5/2}\big)\\
& \qquad + \Delta_n\int_0^{T-} \bigg( 1 - 2\frac{s - t_{*}(s)}{\Delta_n} \bigg)\,\dd [\theta,M_n^{\lambda}]_s + O_p\big(n^{-\beta}(K_n\Delta_n)^{3/2} \big)\\
& \qquad + \Delta_n\int_0^{T-} \bigg( 1 - 2\frac{s - t_{*}(s)}{\Delta_n} \bigg)\,\dd [\lambda,M_n^{\theta}]_s 
+ O_p\big(n^{-\alpha}(K_n\Delta_n)^{3/2} \big). 
\end{split} 
\notag
\end{equation} 
We now consider two different sets of restrictions on the edge effect. All other cases can be deduced from~\eqref{eq::qv_before_rates}. For all $t$ on a given grid, 
\begin{equation}
\begin{split}
\mbox{ Case (1): } &e_t^{\theta} = o_p((K_n\Delta_n)^{1/2}n^{-\alpha}), \quad\text{and}\quad
e_t^{\lambda} = o_p((K_n\Delta_n)^{1/2}n^{-\beta});\\
\mbox{ Case (2): } & e_t^{\theta} = o_p((K_n\Delta_n)^{3/4} n^{-\alpha}), \quad\text{and}\quad 
e_t^{\lambda} = o_p((K_n\Delta_n)^{3/4} n^{-\beta}).
\end{split}
\notag
\end{equation}
Under Case (1) we have that~\eqref{eq::qv_before_rates} is 
\begin{equation}
\begin{split}
\qv_{B,K}(\widehat{\Theta},\widehat{\Lambda})  = 2[M_n^{\theta},M_n^{\lambda}]_{T-} + \frac{2}{3}(K_n\Delta_n)^2[\theta,\lambda]_{T-} +  o_p\big((K_n\Delta_n)^2\big) + o_p(n^{-(\alpha + \beta)}).
\end{split}
\notag
\end{equation}  
While under Case (2) we find that~\eqref{eq::qv_before_rates} is 
\begin{equation}
\begin{split}
\qv_{B,K}(\widehat{\Theta},\widehat{\Lambda})  & = 2[M_n^{\theta},M_n^{\lambda}]_{T-} + \frac{2}{3}(K_n\Delta_n)^2[\theta,\lambda]_{T-}\\ 
& \qquad \qquad \qquad \qquad +  O_p\big((K_n\Delta_n)^{5/2}\big) + O_p((K_n\Delta_n)^{1/2}n^{-(\alpha + \beta)}).
\end{split}
\notag
\end{equation}  
It thus appears that the more stringent conditions on the edge effects in Case (2) are needed for the convergence rates of Theorem~\ref{th::convergence.rates} to {`}enter{'} Theorem~\ref{theorem::Th3.multivariate}. Do note, however, that this may be an artefact of the Cauchy--Schwarz inequality used in deriving~\eqref{eq::qv_before_rates}.

\section{Proof of Theorem~\ref{th::convergence.rates}}\label{app::conv.rates}
Recall that
\begin{equation}
t_{*,l} = t_{*,l}(s) = \max\{t_{i+K} \colon t_{i+K}\leq s , i \equiv l[2K] \}.
\label{eq::max_time} 
\end{equation}
It is enough to show the result when the sequences $\alpha^{(n)}$ and $\beta^{(n)}$ are local square-integrable martingales. Let
\begin{equation}
\begin{split}
Z_{n,l}(s) & =  \sum_{t_{i+K}\leq s ,\, i \equiv l[2K]}(\alpha^{(l,n)}_{t_{i+K}} - \alpha^{(l,n)}_{t_{i-K}})(\beta^{(l,n)}_{t_{i+K}} - \beta^{(l,n)}_{t_{i-K}})\\  
& \qquad \qquad \qquad \qquad  + (\alpha^{(l,n)}_{s} - \alpha^{(l,n)}_{t_{*,l}})(\beta^{(l,n)}_{s} -  \beta^{(l,n)}_{t_{*,l}}) - [\alpha^{(l,n)},\beta^{(l,n)}]_{s},
\end{split}
\notag
\end{equation}
and set $Z_n(s) = K^{-1}\sum_{l=1}^{2K}Z_{n,l}(s)$. Let the stopping time $\tau$ and the positive constants $a_{+}$ and $b_{+}$ be such that, for $t \le \tau$,  $\dd\langle \alpha^{(n)}, \alpha^{(n)} \rangle_t / \dd t \leq a_{+}^2$ and
$\dd\langle\beta^{(n)} , \beta^{(n)} \rangle_t/\dd t \leq b_{+}^2$. In particular,
$|\dd\langle \alpha^{(n)} \beta^{(n)} \rangle_t/\dd t| \le a_{+}b_{+}$,
by the Kunita--Watanabe inequality (see e.g.,~\cite[Theorem~II.25, p.~69]{protter2004}). By It{\^o}{'}s lemma we have that 
\begin{equation}
\begin{split}
  \langle Z_{n,l_1},Z_{n,l_2} \rangle_{\tau}  
 & = \int_0^{\tau}  (\alpha_s^{(l_1,n)} - \alpha_{t_{*,l_1}}^{(l_1,n)})(\alpha_s^{(l_2,n)} - \alpha_{t_{*,l_2}}^{(l_2,n)})\,\dd\langle\beta^{(l_1,n)},\beta^{(l_2,n)} \rangle_{s}[2]\\
& \qquad + \int_0^{\tau}  (\alpha_s^{(l_1,n)} - \alpha_{t_{*,l_1}}^{(l_1,n)})(\beta_s^{(l_2,n)} - \beta_{t_{*,l_2}}^{(l_2,n)})\,\dd\langle\beta^{(l_1,n)},\alpha^{(l_2,n)} \rangle_{s}[2].
\end{split}
\notag
\end{equation}
From which 
\begin{equation}
\begin{split}
& \E\,\langle Z_n,Z_n\rangle_{\tau}  = \frac{1}{4K^2}\sum_{l_1=1}^{2K}\sum_{l_2=1}^{2K}\E\,\langle Z_{n,l_1},Z_{n,l_2}\rangle_{\tau}\\
&  \qquad = \frac{1}{4K^2}\sum_{l_1=1}^{2K}\sum_{l_2=1}^{2K}\E\,\bigg\{\int_0^{\tau}  (\alpha_s^{(l_1,n)} - \alpha_{t_{*,l_1}}^{(l_1,n)})(\alpha_s^{(l_2,n)} - \alpha_{t_{*,l_1}}^{(l_2,n)})\,\dd\langle\beta^{(l_1,n)},\beta^{(l_2,n)} \rangle_{s}[2]\\
&  \qquad\qquad\qquad + \int_0^{\tau}  (\alpha_s^{(l_1,n)} - \alpha_{t_{*,l_1}}^{(l_1,n)})(\beta_s^{(l_2,n)} - \beta_{t_{*,l_1}}^{(l_2,n)})\,\dd\langle\beta^{(l_1,n)},\alpha^{(l_2,n)} \rangle_{s}[2]\bigg\}.
\end{split}
\label{eq::ref_terms}
\end{equation}
Changing the order of summation and integration we have that, 
\begin{equation}
\begin{split}
\sum_{l_1=1}^{2K}&\sum_{l_2=1}^{2K}  \E\,\int_0^{\tau}  (\alpha_s^{(l_1,n)} - \alpha_{t_{*,l_1}}^{(l_1,n)})g_s^{(l_1,n)}(\alpha_s^{(l_2,n)} - \alpha_{t_{*,l_1}}^{(l_2,n)})g_s^{(l_2,n)}\,\dd\langle\beta^{(n)},\beta^{(n)} \rangle_{s}\\
& = \int_0^{T}\E\, \sum_{l_1=1}^{2K}\sum_{l_2=1}^{2K}  (\alpha_{s\wedge \tau}^{(l_1,n)} - \alpha_{t_{*,l_1}\wedge\tau}^{(l_1,n)})g_{s}^{(l_1,n)}(\alpha_{s\wedge\tau}^{(l_2,n)} - \alpha_{t_{*,l_1}\wedge\tau}^{(l_2,n)})g_s^{(l_2,n)}\,\dd\langle\beta^{(n)},\beta^{(n)} \rangle_{s}\\
& = \int_0^{T}\E\, \bigg[\sum_{l=1}^{2K}  (\alpha_{s\wedge \tau}^{(l,n)} - \alpha_{t_{*,l}\wedge\tau}^{(l,n)})g_{s}^{(l,n)}\bigg]^2\,\dd\langle\beta^{(n)},\beta^{(n)} \rangle_{s}\\ 
& \leq b_{+}^2 \int_0^{T}\E\, \bigg(\sum_{l=1}^{2K}  (\alpha_{s\wedge \tau}^{(l,n)} - \alpha_{t_{*,l}\wedge\tau}^{(l,n)})g_{s}^{(l,n)}\bigg)^2\,\dd s,
\end{split}
\notag 
\end{equation} 
and similarly for the second term on the right in~\eqref{eq::ref_terms}. For the third and fourth terms on the right in~\eqref{eq::ref_terms} we use that $\dd\langle\beta^{(n)},\alpha^{(n)} \rangle_{s}/\dd s \leq a_{+}b_{+}$ for $s\leq \tau$, and H{\"o}lder{'}s inequality,   
\begin{equation}
\begin{split}
\sum_{l_1=1}^{2K}&\sum_{l_2=1}^{2K} \E\, \int_0^{\tau}  (\alpha_s^{(l_1,n)} - \alpha_{t_{*,l_1}}^{(l_1,n)})g_{s}^{(l_1,n)}(\beta_s^{(l_2,n)} - \beta_{t_{*,l_1}}^{(l_2,n)})f_{s}^{(l_2,n)}\,\dd\langle\beta^{(n)},\alpha^{(n)} \rangle_{s}\\
& = \int_0^{T}\E\,\sum_{l_1=1}^{2K}\sum_{l_2=1}^{2K}   (\alpha_{s\wedge\tau}^{(l_1,n)} - \alpha_{t_{*,l_1},\wedge\tau}^{(l_1,n)})g_{s}^{(l_1,n)}(\beta_{s\wedge\tau}^{(l_2,n)} - \beta_{t_{*,l_1}\wedge\tau}^{(l_2,n)})f_{s}^{(l_2,n)}\,\dd\langle\beta^{(n)},\alpha^{(n)} \rangle_{s}\\
& = \int_0^{T}\E\,\bigg(\sum_{l_1=1}^{2K} (\alpha_{s\wedge\tau}^{(l_1,n)} - \alpha_{t_{*,l_1},\wedge\tau}^{(l_1,n)})g_{s}^{(l_1,n)}\bigg)\bigg( \sum_{l_2=1}^{2K}(\beta_{s\wedge\tau}^{(l_2,n)} - \beta_{t_{*,l_1}\wedge\tau}^{(l_2,n)})f_{s}^{(l_2,n)}\bigg)\,\dd\langle\beta^{(n)},\alpha^{(n)} \rangle_{s}\\
& \leq a_{+}b_{+} \int_0^{T}\E\,\bigg(\sum_{l_1=1}^{2K} (\alpha_{s\wedge\tau}^{(l_1,n)} - \alpha_{t_{*,l_1},\wedge\tau}^{(l_1,n)})g_{s}^{(l_1,n)}\bigg)\bigg( \sum_{l_2=1}^{2K}(\beta_{s\wedge\tau}^{(l_2,n)} - \beta_{t_{*,l_1}\wedge\tau}^{(l_2,n)})f_{s}^{(l_2,n)}\bigg)\,\dd s\\
& \leq a_{+}b_{+} \int_0^{T}\norm{\sum_{l_1=1}^{2K} (\alpha_{s\wedge\tau}^{(l_1,n)} - \alpha_{t_{*,l_1},\wedge\tau}^{(l_1,n)})g_{s}^{(l_1,n)}}_2\norm{\sum_{l_2=1}^{2K}(\beta_{s\wedge\tau}^{(l_2,n)} - \beta_{t_{*,l_1}\wedge\tau}^{(l_2,n)})f_{s}^{(l_2,n)}}_2\,\dd s,
\end{split}
\notag
\end{equation}
and similarly for the fourth term. Now, all the action takes place in expressions of the form
\begin{equation}
\begin{split}
%
& \norm{\sum_{l=1}^{2K}  (\alpha_{s\wedge \tau}^{(l,n)}  - \alpha_{t_{*,l}\wedge\tau}^{(l,n)})g_{s}^{(l,n)}}_2^2  
= \E\,\sum_{l=1}^{2K}  (\alpha_{s\wedge \tau}^{(l,n)} - \alpha_{t_{*,l}\wedge\tau}^{(l,n)})g_{s}^{(l,n)}
\sum_{l=1}^{2K}  (\alpha_{s\wedge \tau}^{(l,n)} - \alpha_{t_{*,l}\wedge\tau}^{(l,n)})g_{s}^{(l,n)}\\
& \qquad\qquad = \sum_{l_1=1}^{2K}\sum_{l_2=1}^{2K}\E\,  (\alpha_{s\wedge \tau}^{(l_1,n)} - \alpha_{t_{*,l_1}\wedge\tau}^{(l_1,n)})
(\alpha_{s\wedge \tau}^{(l_2,n)} - \alpha_{t_{*,l_2}\wedge\tau}^{(l_2,n)})g_{s}^{(l_1,n)}g_{s}^{(l_2,n)}.
\end{split}
\label{eq::key2}
\end{equation}
Since $|f_s^{(l,n)}|\leq 1$ and $|g_s^{(l,n)}|\leq 1$ for all $s$, $n$, and $l$, we have that 
\begin{equation}
\begin{split}
\text{\eqref{eq::key2}} & = 
\sum_{l_1=1}^{2K}\sum_{l_2=1}^{2K}\E\,\{  \big(\langle\alpha^{(l_1,n)},\alpha^{(l_2,n)}\rangle_{s\wedge \tau} - \langle\alpha^{(l_1,n)},\alpha^{(l_2,n)}\rangle_{(t_{*,l_1}\vee t_{*,l_2})\wedge \tau}\big) g_{s}^{(l_1,n)}g_{s}^{(l_2,n)}\}\\
& = \sum_{l_1=1}^{2K}\sum_{l_2=1}^{2K}\E\,\{\int_{(t_{*,l_1}\vee t_{*,l_2})\wedge \tau}^{s\wedge \tau} f_u^{(l_1,n)}f_u^{(l_2,n)} \,\dd\langle\alpha^{(n)},\alpha^{(n)} \rangle_{u}g_{u}^{(l_1,n)}g_{u}^{(l_2,n)}\\
& \leq a_{+}^2 \sum_{l_1=1}^{2K}\sum_{l_2=1}^{2K} (s - (t_{*,l_1}\vee t_{*,l_2})).
\end{split}
\label{eq::angleZZ.ineq}
\end{equation}
For $l=1,\ldots,2K$, define 
\begin{equation}
h_s^{(l,n)} = \sum_{K \leq i \leq B - K, i\equiv l[2K]}(s - t_{i-K})I\{t_{i-K}\leq s < t_{i+K}\},
\notag
\end{equation}
and notice that
\begin{equation}
(s - t_{*,l}) = (s - t_{*,l}(s)) = h_s^{(l,n)}.
\notag
\end{equation}
Substituting the bound in~\eqref{eq::angleZZ.ineq} and the three similar ones into $\E\,\langle Z_n,Z_n\rangle_{\tau}$, then  
\begin{equation}
\begin{split}
\E\,\langle Z_n,Z_n\rangle_{\tau} & \lesssim  \frac{1}{K^2} \sum_{l_1=1}^{2K}\sum_{l_2=1}^{2K}\int_{0}^{T} (s - (t_{*,l_1}\vee t_{*,l_2}))\,\dd s\\
& = \frac{1}{K^2} \sum_{l_1=1}^{2K}\bigg\{ \int_0^{T}(s - (t_{*,l_1}\vee t_{*,1}))\,\dd s + \cdots +\int_0^{T}(s - (t_{*,l_1}\vee t_{*,2K}))\,\dd s\bigg\}\\
& \leq \frac{1}{K^2}\sum_{l_1=1}^{2K}\bigg\{ \int_0^{T}(s - t_{*,l_1})\,\dd s  + \cdots+\int_0^{T}(s - t_{*,l_1})\,\dd s\bigg\}
= \frac{2}{K}\sum_{l=1}^{2K}\int_0^{T}(s - t_{*,l})\,\dd s\\
& = \frac{2}{K}\sum_{l=1}^{2K}\int_0^{T} h_s^{(l,n)}\,\dd s
= \frac{2}{K}\sum_{l=1}^{2K}\sum_{K \leq i \leq B - K,\equiv l[2K]}\int_{t_{i-K}}^{t_{i+K}} (s - t_{i-K})\,\dd s\\
& = \frac{2}{K}\sum_{i=K}^{B-K}\int_{t_{i-K}}^{t_{i+K}} (s - t_{i-K})\,\dd s
= \frac{4}{K}\sum_{i=K}^{B-K} (K\Delta_n)^2
= 4 {T} K\Delta_n,
\end{split}
\notag
\end{equation}
where the proportionality constant left out is $\max(a_{+}^4,b_{+}^4)$. Let 
\begin{equation}
\tau_n = \inf\{t\in[0,T] \colon \langle \alpha^{(n)},\alpha^{(n)}\rangle_t > a_{+}^2t\;\text{or}\;\langle \beta^{(n)},\beta^{(n)}\rangle_t > b_{+}^2t\}. 
\notag
\end{equation}
By Condition~\ref{cond:rate}, $\tau_n \to T$ as $n \to \infty$. Let $\eps > 0$ and choose $a_{+}$ and $b_{+}$ sufficiently large, so that $P(\tau_n \neq T) \leq \varepsilon/2$, and let $c = \max(a_{+}^4,b_{+}^4)$. Then 
\begin{equation}
\begin{split} 
P(\langle Z_n,Z_n\rangle_T/(4c {T} K\Delta_n) > M)&  
\leq P(\langle Z_n,Z_n\rangle_{\tau_n}/(4c {T} K\Delta_n) > M) + P(\tau_n \neq T)\\
& \leq M^{-1}\E\,[\langle Z_n,Z_n\rangle_{\tau_n}/(4c {T} K\Delta_n)] + P(\tau_n \neq T)\\
& = M^{-1} + \varepsilon/2 \leq \varepsilon, 
\end{split}
\notag
\end{equation}
provided $ M \geq 2/\varepsilon$. This shows that $\langle Z_n,Z_n\rangle_T/(4c {T} K\Delta_n)$ is tight, so 
\begin{equation}
\langle Z_n,Z_n\rangle_T = O_p(4c {T} K_n\Delta_n) = O_p(K_n\Delta_n). 
\notag
\end{equation}
By Lenglart{'}s inequality \citep[p.~86]{andersen1993statistical}, for any $\delta >0$ and $M>0$
\begin{equation}
P(\sup_{0 \leq t\leq T}|Z_n(t)| > \delta) \leq \frac{M}{\delta^2} + P(\langle Z_n,Z_n\rangle_T > M). 
\notag
\end{equation} 
With the same $\delta = M$ and the same $M$ as above, $P(\sup_{0 \leq t\leq T}|Z_n(t)| > \delta) \leq (3/2)\varepsilon$, from which we conclude that
\begin{equation}
\sup_{0\leq t \leq T}|Z_n(t)| = O_p\big((K\Delta_n)^{1/2}\big).
\notag
\end{equation}

\section{Proof of Theorem~\ref{theorem::clt}}\label{app::CLT_proof}
For this theorem we are assuming that the sequences $\alpha^{(n)}$ and $\beta^{(n)}$ are square-integrable local martingales, both equal to zero at time $t = 0$. For $l=1,\ldots,2K$, we define 
\begin{equation}
\begin{split}
&Z_{n,l}(t)  = \sum_{t_{i+K}\leq t ,\, i \equiv l[2K]}(\alpha_{t_{i+K}}^{(l,n)} - \alpha_{t_{i-K}}^{(l,n)})(\beta_{t_{i+K}}^{(l,n)} - \beta_{t_{i-K}}^{(l,n)})\\
& \qquad \qquad \qquad \qquad 
+ (\alpha_{t}^{(l,n)} - \alpha_{t_{*,l}}^{(l,n)})(\beta_{t}^{(l,n)} - \beta_{t_{*,l}}^{(l,n)}) - [\alpha^{(l,n)},\beta^{(l,n)}]_{t}\\
& \;= \sum_{t_{i+K}\leq t ,\, i \equiv l[2K]}\big\{\int_{t_{i-K}}^{t_{i+K}}(\alpha_s^{(l,n)} - \alpha_{t_{i-K}}^{(l,n)})\,\dd \beta_s^{(l,n)} 
+ \int_{t_{i-K}}^{t_{i+K}}(\beta_s^{(l,n)} - \beta_{t_{i-K}}^{(l,n)})\,\dd \alpha_s^{(l,n)} \big\}\\
& \qquad \qquad \qquad \qquad + \int_{t_{*,l}}^{t}(\alpha_s^{(l,n)} - \alpha_{t_{*,l}}^{(l,n)})\,\dd \beta_s^{(l,n)} 
+ \int_{t_{*,l}}^{t}(\beta_s^{(l,n)} - \beta_{t_{*,l}}^{(l,n)})\,\dd \alpha_s^{(l,n)},
\end{split}
\label{eq::errorMG1}
\end{equation}  
and set
\begin{equation}
Z_{n}(t) = \frac{1}{2K}\sum_{l=1}^{2K}Z_{n,l}(t),
\notag
\end{equation}  
so that $Z_n(T)$ is the martingale in~\eqref{eq::CLT_MG}. We will verify that the sequence $(K\Delta_n)^{-1/2}Z_{n}(t)$ satisfies Conditions (i)--(iii) of Theorem~\ref{theorem::th2.28_general}, and the claim will follow. 

Recall that by Condition~\ref{cond:rate} we assume there are processes, say $a_{n,t}$, $b_{n,t}$, and $c_{n,t}$, such that 
\begin{equation}
\dd\langle \alpha^{(n)},\alpha^{(n)} \rangle_t = a_{n,t}^2\,\dd t, \quad
\dd\langle \beta^{(n)},\beta^{(n)} \rangle_t = b_{n,t}^2\,\dd t,\quad\text{and}\quad 
\dd\langle \alpha^{(n)},\beta^{(n)} \rangle_t = c_{n,t}\,\dd t,
\notag
\end{equation}
and that these are locally bounded uniformly in $n$, that is, there is a sequence of stopping times $\tau_n$ such that for $0 \leq t \leq \tau_n$ we have finite constants $a_{+}, b_{+}$, and $c_{+}$ such that $a_{n,t} \leq a_{+}$, $b_{n,t} \leq b_{+}$, and $|c_{n,t}| \leq c_{+}$, and $P(\tau_n = T) \to 1$ as $n \to \infty$. Moreover, we assume that $a_{n,t}^2, b_{n,t}^2$, and $c_{n,t}$ are locally continuous in mean square, see Definition~\ref{assumption3}.     

The quadratic variation of $Z_{n}(t)$ is 
\begin{equation}
\begin{split}
\langle Z_n,Z_n\rangle_{T} & = \frac{1}{4K^2}\sum_{l_1=1}^{2K}\sum_{l_2=1}^{2K} \langle Z_{n,l_1},Z_{n,l_2}\rangle_{T}.\\
\end{split}
\notag
\end{equation}
Here
\begin{equation}
\begin{split}
\langle Z_{n,l_1},Z_{n,l_2}\rangle_{T} & = 
\int_0^{T}  (\alpha_s^{(l_1,n)} - \alpha_{t_{*,l_1}}^{(l_1,n)})(\alpha_s^{(l_2,n)} - \alpha_{t_{*,l_2}}^{(l_2,n)})\,\dd\langle\beta^{(l_1,n)},\beta^{(l_2,n)} \rangle_{s}[2]\\
& \qquad\quad + \int_0^{T}  (\alpha_s^{(l_1,n)} - \alpha_{t_{*,l_1}}^{(l_1,n)})(\beta_s^{(l_2,n)} - \beta_{t_{*,l_2}}^{(l_2,n)})\,\dd\langle\beta^{(l_1,n)},\alpha^{(l_2,n)} \rangle_{s}[2]\\
& = \sum_{j=1}^{4}\langle Z_{n,l_1},Z_{n,l_2}\rangle_{t}^{(j)},
\end{split}
\label{eq::cross_angles}
\end{equation}
by which we define $\langle Z_{n,l_1},Z_{n,l_2}\rangle_{t}^{(j)}$ for $j =1,2,3,4$. Start by concentrating on $\langle Z_{n,l_1},Z_{n,l_2}\rangle_{T}^{(1)}$, which is given by 
\begin{equation}
\begin{split}
\langle Z_{n,l_1},Z_{n,l_2}\rangle_{T}^{(1)} & = \int_0^{T}  (\alpha_s^{(l_1,n)} - \alpha_{t_{*,l_1}}^{(l_1,n)})(\alpha_s^{(l_2,n)} - \alpha_{t_{*,l_2}}^{(l_2,n)})\,\dd\langle\beta^{(l_1,n)},\beta^{(l_2,n)} \rangle_{s}\\
& = \int_0^{T}  (\alpha_s^{(l_1,n)} - \alpha_{t_{*,l_1}}^{(l_1,n)})g_s^{(l_1,n)}(\alpha_s^{(l_2,n)} - \alpha_{t_{*,l_2}}^{(l_2,n)})g_s^{(l_2,n)}\,\dd\langle\beta^{(n)},\beta^{(n)} \rangle_{s}.
\end{split}
\label{eq::first_term}
\end{equation}
Write
\begin{equation}
\{t_{i-K},t_{i+K} : i \equiv l[2K],\,K\leq i \leq B-K\} = \{t_{0,l},t_{1,l},t_{2,l} ,\ldots\},
\notag
\end{equation}
where the indices on the right hand side are such that $t_{i,l} < t_{i+1,l}$, and let $\G^{(l)}$ be the set of these time points, i.e.~$\G^{(l)} = \{t_{0,l} < t_{1,l} < t_{2,l} < \cdots\}$. With this notation we have, e.g. that
\begin{equation}
\begin{split}
\sum_{K\leq i \leq B-K,\,i\equiv l[2K]}(\alpha_{t_{i+K}}^{(l,n)} - \alpha_{t_{i-K}}^{(l,n)})(\beta_{t_{i+K}}^{(l,n)} - \beta_{t_{i-K}}^{(l,n)}) 
= \sum_{t_{i+1,l}\leq T}(\alpha_{t_{i+1,l}}^{(l,n)} - \alpha_{t_{i,l}}^{(l,n)})
(\beta_{t_{i+1,l}}^{(l,n)} - \beta_{t_{i,l}}^{(l,n)}).
\end{split}
\notag
\end{equation}
The time $t_{*,l}$ defined in~\eqref{eq::max_time} is now simply $t_{*,l}  = \min\{t_i \in \G^{(l)} : t_i \leq s\} = \min\{t_{i,l}  : t_{i,l} \leq s\}$. Attach the number $t_{-1,l} = 0 $ to $\G^{(l)}$ if it is not already there, and suppose, without loss of generality, that $t_{i,l_1} < t_{i,l_2}$ for all $i$, and that $t_{0,l_1} = 0$. We can then write 
\begin{equation}
\begin{split}
& \langle Z_{n,l_1},Z_{n,l_2}\rangle_{T}^{(1)} 
 = \int_0^{T}  (\alpha_s^{(l_1,n)} - \alpha_{t_{*,l_1}}^{(l_1,n)})(\alpha_s^{(l_2,n)} - \alpha_{t_{*,l_2}}^{(l_2,n)})\,\dd\langle\beta^{(l_1,n)},\beta^{(l_2,n)} \rangle_{s}\\
& \qquad = \sum_{i\,:\,t_{i+1,l_1}\leq T}\big\{\int_{t_{i,l_1}}^{t_{i,l_2}}
(\alpha_s^{(l_1,n)} - \alpha_{t_{i,l_1}}^{(l_1,n)})(\alpha_s^{(l_2,n)} - \alpha_{t_{i-1,l_2}}^{(l_2,n)})\,\dd\langle\beta^{(l_1,n)},\beta^{(l_2,n)} \rangle_{s}\\
& \qquad \qquad \qquad   + \int_{t_{i,l_2}}^{t_{i+1,l_1}}
(\alpha_s^{(l_1,n)} - \alpha_{t_{i,l_1}}^{(l_1,n)})(\alpha_s^{(l_2,n)} - \alpha_{t_{i,l_2}}^{(l_2,n)})\,\dd\langle\beta^{(l_1,n)},\beta^{(l_2,n)} \rangle_{s}\big\}\\
& \qquad \qquad \qquad \quad + \int_{t_{*,l_1}(T)}^T (\alpha_s^{(l_1,n)} - \alpha_{t_{*,l_1}(s)}^{(l_1,n)})(\alpha_s^{(l_2,n)} - \alpha_{t_{*,l_2}(s)}^{(l_2,n)})\,\dd\langle\beta^{(l_1,n)},\beta^{(l_2,n)} \rangle_{s}.
\end{split}
\label{eq::key_decomp}
\end{equation}
We now want to show that~\eqref{eq::key_decomp} is  
\begin{equation}
\begin{split}
& \langle Z_{n,l_1},Z_{n,l_2}\rangle_{T}^{(1)} 
= \sum_{i\,:\,t_{i+1,l_1}\leq T}\big\{\int_{t_{i,l_1}}^{t_{i,l_2}}
\int_{t_{i,l_1}}^s \,\dd \langle \alpha^{(l_1,n)},\alpha^{(l_2,n)} \rangle_s
\,\dd\langle\beta^{(l_1,n)},\beta^{(l_2,n)} \rangle_{s}\\
& \qquad\qquad   + \int_{t_{i,l_2}}^{t_{i+1,l_1}}
\int_{t_{i,l_2}}^s \,\dd \langle \alpha^{(l_1,n)},\alpha^{(l_2,n)} \rangle_s
\,\dd\langle\beta^{(l_1,n)},\beta^{(l_2,n)} \rangle_{s}\big\}\\
& \qquad\qquad 
+ \int_{t_{*,l_1}(T)}^T \int_{(t_{*,l_1}\vee t_{*,l_2})(T) }^s \,\dd \langle \alpha^{(l_1,n)},\alpha^{(l_2,n)} \rangle_s \,\dd\langle\beta^{(l_1,n)},\beta^{(l_2,n)} \rangle_{s} + o_p(K\Delta_n).
\end{split}
\label{eq::part1}
\end{equation}
The key is to show equalities of the type
\begin{equation}
\begin{split}
&\int_{t_{i,l_1}}^{t_{i,l_2}}
(\alpha_s^{(l_1,n)} - \alpha_{t_{i,l_1}}^{(l_1,n)})(\alpha_s^{(l_2,n)} - \alpha_{t_{i-1,l_2}}^{(l_2,n)})\,\dd\langle\beta^{(l_1,n)},\beta^{(l_2,n)} \rangle_{s}\\
& \qquad 
= \int_{t_{i,l_1}}^{t_{i,l_2}}(\langle \alpha^{(l_1,n)},\alpha^{(l_2,n)}\rangle_{s}
- \langle \alpha^{(l_1,n)},\alpha^{(l_2,n)}\rangle_{t_{i,l_1}})
\,\dd\langle\beta^{(l_1,n)},\beta^{(l_2,n)} \rangle_{s}
+ \text{negligible terms},
\end{split}
\notag
\end{equation}
and that the {`}negligible terms{'} are of the appropriate order. Recall that $t_{i-1,l_2} < t_{i,l_1}$, thus by It{\^o}{'}s formula 
\begin{equation}
\begin{split}
&(\alpha_s^{(l_1,n)} - \alpha_{t_{i,l_1}}^{(l_1,n)})(\alpha_s^{(l_2,n)} - \alpha_{t_{i-1,l_2}}^{(l_2,n)}) - (\langle \alpha^{(l_1,n)},\alpha^{(l_2,n)}\rangle_{s}
- \langle \alpha^{(l_1,n)},\alpha^{(l_2,n)}\rangle_{t_{i,l_1}})\\
& \qquad \qquad \qquad \qquad= \int_{t_{i,l_1}}^{s}(\alpha_u^{(l_1,n)} - \alpha_{t_{i,l_1}}^{(l_1,n)})\,\dd \alpha_u^{(l_2,n)}
+ \int_{t_{i,l_1}}^{s}(\alpha_u^{(l_2,n)} - \alpha_{t_{i-1,l_2}}^{(l_2,n)})\,\dd \alpha_u^{(l_1,n)}.
\end{split}
\notag
\end{equation}
This means that 
\begin{equation}
\begin{split}
&\int_{t_{i,l_1}}^{t_{i,l_2}}
(\alpha_s^{(l_1,n)} - \alpha_{t_{i,l_1}}^{(l_1,n)})(\alpha_s^{(l_2,n)} - \alpha_{t_{i-1,l_2}}^{(l_2,n)})\,\dd\langle\beta^{(l_1,n)},\beta^{(l_2,n)} \rangle_{s}\\
&  = \int_{t_{i,l_1}}^{t_{i,l_2}}
(\langle \alpha^{(l_1,n)},\alpha^{(l_2,n)}\rangle_{s}
- \langle \alpha^{(l_1,n)},\alpha^{(l_2,n)}\rangle_{t_{i,l_1}})\,\dd\langle\beta^{(l_1,n)},\beta^{(l_2,n)} \rangle_{s}\\
&    
+ \int_{t_{i,l_1}}^{t_{i,l_2}}\big\{ \int_{t_{i,l_1}}^{s}(\alpha_u^{(l_1,n)} - \alpha_{t_{i,l_1}}^{(l_1,n)})\,\dd \alpha_u^{(l_2,n)} 
+ \int_{t_{i,l_1}}^{s}(\alpha_u^{(l_2,n)} - \alpha_{t_{i-1,l_2}}^{(l_2,n)})\,\dd \alpha_u^{(l_1,n)} \big\}\,\dd\langle \beta^{(l_1,n)},\beta^{(l_2,n)}\rangle_s.\\
\end{split}
\notag
\end{equation}
We now turn to the two last terms on the right hand side of this expression, and look at
\begin{equation} 
\begin{split}
&\int_{t_{i,l_1}}^{t_{i,l_2}}\int_{t_{i,l_1}}^{s}(\alpha_u^{(l_1,n)} - \alpha_{t_{i,l_1}}^{(l_1,n)})\,\dd \alpha_u^{(l_2,n)} \,\dd\langle \beta^{(l_1,n)},\beta^{(l_2,n)}\rangle_s\\
& \qquad \qquad  = \int_{t_{i,l_1}}^{t_{i,l_2}}\int_{t_{i,l_1}}^{s}(\alpha_u^{(l_1,n)} - \alpha_{t_{i,l_1}}^{(l_1,n)})\,\dd \alpha_u^{(l_2,n)} g_s^{(l_1,n)}g_s^{(l_2,n)}b_{n,s}^2\,\dd s\\
& \qquad \qquad  = \int_{t_{i,l_1}}^{t_{i,l_2}}\int_{t_{i,l_1}}^{s}(\alpha_u^{(l_1,n)} - \alpha_{t_{i,l_1}}^{(l_1,n)})\,\dd \alpha_u^{(l_2,n)} g_s^{(l_1,n)}g_s^{(l_2,n)}b_{n,t_{i,l_1}}^2\,\dd s\\
& \qquad \qquad \quad + \int_{t_{i,l_1}}^{t_{i,l_2}}\int_{t_{i,l_1}}^{s}(\alpha_u^{(l_1,n)} - \alpha_{t_{i,l_1}}^{(l_1,n)})\,\dd \alpha_u^{(l_2,n)} g_s^{(l_1,n)}g_s^{(l_2,n)}(b_{n,s}^2 - b_{n,t_{i,l_1}}^2)\,\dd s.
\end{split}
\label{eq::two_terms}
\end{equation}
Consider the two terms on the right in~\eqref{eq::two_terms} separately, starting with the first term. Define the functions  
\begin{equation}
G_t^{(l_1,l_2)} = \int_0^t g_{s}^{(l_1,n)}g_{s}^{(l_2,n)}\,\dd s,\qquad l_1,l_2 = 1,\ldots 2K,
\notag
\end{equation}
and note that since $|g_{s}^{(l_1,n)}| \leq 1$ for all $s$, the functions $G_t^{(l_1,l_2)}$ are Lipschitz with constant $1$, that is,
\begin{equation}
|G_t^{(l_1,l_2)} - G_s^{(l_1,l_2)}| \leq |t - s|, \quad\text{for all $t,s$}.
\notag
\end{equation} 
An application of It{\^o}{'}s formula yields
\begin{equation} 
\begin{split}
& \dd\{\int_{t_{i,l_1}}^{s}(\alpha_u^{(l_1,n)} - \alpha_{t_{i,l_1}}^{(l_1,n)})\,\dd \alpha_u^{(l_2,n)} (G_{t_{i,l_2}}^{(l_1,l_2)} - G_{s}^{(l_1,l_2)})\}\\
& \qquad \qquad = - \int_{t_{i,l_1}}^{s}(\alpha_u^{(l_1,n)} - \alpha_{t_{i,l_1}}^{(l_1,n)})\,\dd \alpha_u^{(l_2,n)}g_{s}^{(l_1,n)}g_{s}^{(l_2,n)}\,\dd s \\
& \qquad \qquad\qquad \qquad+ (G_{t_{i,l_2}}^{(l_1,l_2)} - G_{s}^{(l_1,l_2)})(\alpha_s^{(l_1,n)} - \alpha_{t_{i,l_1}}^{(l_1,n)})\,\dd \alpha_s^{(l_2,n)}.
\end{split}
\notag
\end{equation}
Integrating from $t_{i,l_1}$ to $t_{i,l_2}$, 
\begin{equation}
\begin{split}
&\int_{t_{i,l_1}}^{t_{i,l_2}}\int_{t_{i,l_1}}^{s}(\alpha_u^{(l_1,n)} - \alpha_{t_{i,l_1}}^{(l_1,n)})\,\dd \alpha_u^{(l_2,n)}g_{s}^{(l_1,n)}g_{s}^{(l_2,n)}\,\dd s \\
& \qquad \qquad =  \int_{t_{i,l_1}}^{t_{i,l_2}}(G_{t_{i,l_2}}^{(l_1,l_2)} - G_{s}^{(l_1,l_2)})(\alpha_s^{(l_1,n)} - \alpha_{t_{i,l_1}}^{(l_1,n)})\,\dd \alpha_s^{(l_2,n)}.
\end{split}
\notag
\end{equation}
Since $G_{s}^{(l_1,l_2)}$ is deterministic it is predictable, so the right hand side of this expression is a martingale. Then, for $ t_{i,l_2} < \tau_n$, 
\begin{equation}
\begin{split}
&\norm{\int_{t_{i,l_1}}^{t_{i,l_2}}\int_{t_{i,l_1}}^{s}(\alpha_u^{(l_1,n)} - \alpha_{t_{i,l_1}}^{(l_1,n)})\,\dd \alpha_u^{(l_2,n)} g_{s}^{(l_1,n)}g_{s}^{(l_2,n)}b_{n,t_{i,l_1}}^2\,\dd s}_2^2  \\
& \qquad \qquad \leq b_{+}^2\norm{\int_{t_{i,l_1}}^{t_{i,l_2}}(G_{t_{i,l_2}}^{(l_1,l_2)} - G_{s}^{(l_1,l_2)})(\alpha_s^{(l_1,n)} - \alpha_{t_{i,l_1}}^{(l_1,n)})\,\dd \alpha_s^{(l_2,n)}}_2^2\\
& \qquad \qquad = b_{+}^2\,\E\,\int_{t_{i,l_1}}^{t_{i,l_2}}(G_{t_{i,l_2}}^{(l_1,l_2)} - G_{s}^{(l_1,l_2)})^2(\alpha_s^{(l_1,n)} - \alpha_{t_{i,l_1}}^{(l_1,n)})^2\,\dd \langle \alpha^{(l_2,n)},\alpha^{(l_2,n)}\rangle_s\\
& \qquad \qquad = b_{+}^2\,\E\,\int_{t_{i,l_1}}^{t_{i,l_2}}(G_{t_{i,l_2}}^{(l_1,l_2)} - G_{s}^{(l_1,l_2)})^2(\alpha_s^{(l_1,n)} - \alpha_{t_{i,l_1}}^{(l_1,n)})^2(f_{s}^{(l,n)})^2a_{n,s}^2\,\dd s\\
& \qquad \qquad \leq a_{+}^2b_{+}^2\, \E\,\int_{t_{i,l_1}}^{t_{i,l_2}}(G_{t_{i,l_2}}^{(l_1,l_2)} - G_{s}^{(l_1,l_2)})^2(\alpha_s^{(l_1,n)} - \alpha_{t_{i,l_1}}^{(l_1,n)})^2\,\dd s\\
& \qquad \qquad \leq a_{+}^2b_{+}^2\, \E\,\int_{t_{i,l_1}}^{t_{i,l_2}}(t_{i,l_2} - s)^2\E\,(\alpha_s^{(l_1,n)} - \alpha_{t_{i,l_1}}^{(l_1,n)})^2\,\dd s\\
& \qquad \qquad \leq a_{+}^2b_{+}^2\, \int_{t_{i,l_1}}^{t_{i,l_2}}(t_{i,l_2} - s)^2\E\,(\langle\alpha^{(l_1,n)},\alpha^{(l_1,n)}\rangle_s -
\langle\alpha^{(l_1,n)},\alpha^{(l_1,n)}\rangle_{t_{i,l_1}} )\,\dd s\\
& \qquad \qquad \leq a_{+}^4b_{+}^2\, \int_{t_{i,l_1}}^{t_{i,l_2}}(t_{i,l_2} - s)^2(s - t_{i,l_1})\,\dd s
= \frac{a_{+}^4 b_{+}^2}{12}(t_{i,l_2} - t_{i,l_1})^4.
\end{split}
\notag
\end{equation}
Since the martingale increments are uncorrelated, this gives
\begin{equation}
\begin{split}
& \E\, \big(\sum_{i\,:\,t_{i+1,l_1}\leq \tau_n} \int_{t_{i,l_1}}^{t_{i,l_2}}\int_{t_{i,l_1}}^{s}(\alpha_u^{(l_1,n)} - \alpha_{t_{i,l_1}}^{(l_1,n)})\,\dd \alpha_u^{(l_2,n)} g_{s}^{(l_1,n)}g_{s}^{(l_2,n)}b_{n,s}^2\,\dd s		\big)^2\\
& \qquad \qquad \leq 
b_{+}^2\,\E\, \big( \sum_{i\,:\,t_{i+1,l_1}\leq \tau_n} \int_{t_{i,l_1}}^{t_{i,l_2}}(G_{t_{i,l_2}}^{(l_1,l_2)} - G_{s}^{(l_1,l_2)})(\alpha_s^{(l_1,n)} - \alpha_{t_{i,l_1}}^{(l_1,n)})\,\dd \alpha_s^{(l_2,n)}		\big)^2\\
& \qquad \qquad =
b_{+}^2\sum_{i\,:\,t_{i+1,l_1}\leq \tau_n}\E\, \big(  \int_{t_{i,l_1}}^{t_{i,l_2}}(G_{t_{i,l_2}}^{(l_1,l_2)} - G_{s}^{(l_1,l_2)})(\alpha_s^{(l_1,n)} - \alpha_{t_{i,l_1}}^{(l_1,n)})\,\dd \alpha_s^{(l_2,n)}		\big)^2\\
& \qquad \qquad \leq 
\frac{a_{+}^4b_{+}^2}{12} \sum_{i\,:\,t_{i+1}\leq \tau_n}(t_{i,l_2} - t_{i,l_1})^4.
\end{split}
\notag
\end{equation}
By Chebyshev{'}s inequality we have that for any $\varepsilon > 0$, 
\begin{equation}
\begin{split}
& P\big\{ \big|\sum_{i\,:\,t_{i+1,l_1}\leq T} \int_{t_{i,l_1}}^{t_{i,l_2}}\int_{t_{i,l_1}}^{s}(\alpha_u^{(l_1,n)} - \alpha_{t_{i,l_1}}^{(l_1,n)})\,\dd \alpha_u^{(l_2,n)} g_{s}^{(l_1,n)}g_{s}^{(l_2,n)}b_{n,s}^2\,\dd s		\big| \geq \varepsilon \big\}\\
& \qquad \qquad \qquad \qquad \leq  \frac{1}{\varepsilon^2}\frac{a_{+}^4b_{+}^2}{12} \sum_{i\,:\,t_{i+1}\leq \tau_n}(t_{i,l_2} - t_{i,l_1})^4 + P(\tau_n \neq T), 
\end{split}
\notag
\end{equation}
which shows that $\sum_{i\,:\,t_{i+1,l_1}\leq T} \int_{t_{i,l_1}}^{t_{i,l_2}}\int_{t_{i,l_1}}^{s}(\alpha_u^{(l_1,n)} - \alpha_{t_{i,l_1}}^{(l_1,n)})\,\dd \alpha_u^{(l_2,n)} g_{s}^{(l_1,n)}g_{s}^{(l_2,n)}b_{n,s}^2\,\dd s$ is $o_p(K\Delta_n)$ as $K\Delta_n \to 0$. We now turn to the second term in~\eqref{eq::two_terms}. For $t_{i,l_1} \leq s < \tau_n$,
\begin{equation}
\begin{split}
& \norm{\int_{t_{i,l_1}}^{s}(\alpha_u^{(l_1,n)} - \alpha_{t_{i,l_1}}^{(l_1,n)})\,\dd \alpha_u^{(l_2,n)} g_s^{(l_1,n)}g_s^{(l_2,n)}(b_{n,s}^2 - b_{n,t_{i,l_1}}^2)}_1\\
& \qquad \qquad \leq \norm{\int_{t_{i,l_1}}^{s}(\alpha_u^{(l_1,n)} - \alpha_{t_{i,l_1}}^{(l_1,n)})\,\dd \alpha_u^{(l_2,n)}}_2 \norm{g_s^{(l_1,n)}g_s^{(l_2,n)}(b_{n,s}^2 - b_{n,t_{i,l_1}}^2)}_2\\
& \qquad \qquad = \norm{\int_{t_{i,l_1}}^{s}(\alpha_u^{(l_1,n)} - \alpha_{t_{i,l_1}}^{(l_1,n)})\,\dd \alpha_u^{(l_2,n)}}_2 |g_s^{(l_1,n)}g_s^{(l_2,n)}|\norm{b_{n,s}^2 - b_{n,t_{i,l_1}}^2}_2\\
& \qquad \qquad \leq \norm{\int_{t_{i,l_1}}^{s}(\alpha_u^{(l_1,n)} - \alpha_{t_{i,l_1}}^{(l_1,n)})\,\dd \alpha_u^{(l_2,n)}}_2 \norm{b_{n,s}^2 - b_{n,t_{i,l_1}}^2}_2\\
& \qquad \qquad = \big(\E\,\int_{t_{i,l_1}}^{s}(\alpha_u^{(l_1,n)} - \alpha_{t_{i,l_1}}^{(l_1,n)})^2 (f_u^{(l_2,n)})^2a_{n,u}^2\,\dd u\big)^{1/2} \norm{b_{n,s}^2 - b_{n,t_{i,l_1}}^2}_2\\
& \qquad \qquad \leq a_{+} \big(\E\,\int_{t_{i,l_1}}^{s}( \langle \alpha^{(l_1,n)},\alpha^{(l_1,n)}\rangle_u - \langle \alpha^{(l_1,n)},\alpha^{(l_1,n)}\rangle_{t_{i,l_1}}) \,\dd u\big)^{1/2} \norm{b_{n,s}^2 - b_{n,t_{i,l_1}}^2}_2\\
& \qquad \qquad \leq a_{+}^2 \big(\int_{t_{i,l_1}}^{s}(u - t_{i,l_1}) \,\dd u\big)^{1/2} \norm{b_{n,s}^2 - b_{n,t_{i,l_1}}^2}_2\\
& \qquad \qquad = \frac{a_{+}^2}{\sqrt{2}} (s - t_{i,l_1}) \norm{b_{n,s}^2 - b_{n,t_{i,l_1}}^2}_2 
\leq \frac{a_{+}^2}{\sqrt{2}} (s - t_{i,l_1}) \sup_{t_{i,l_1} \leq s \leq t_{i,l_2}}\norm{b_{n,s}^2 - b_{n,t_{i,l_1}}^2}_2.
\end{split}
\notag
\end{equation}
From this we get that for $t_{i,l_2} < \tau_n$,
\begin{equation}
\begin{split}
& \norm{\int_{t_{i,l_1}}^{t_{i,l_2}}\int_{t_{i,l_1}}^{s}(\alpha_u^{(l_1,n)} - \alpha_{t_{i,l_1}}^{(l_1,n)})\,\dd \alpha_u^{(l_2,n)} g_s^{(l_1,n)}g_s^{(l_2,n)}(b_{n,s}^2 - b_{n,t_{i,l_1}}^2)\,\dd s}_1\\
& \qquad \qquad \leq \int_{t_{i,l_1}}^{t_{i,l_2}}
\norm{\int_{t_{i,l_1}}^{s}(\alpha_u^{(l_1,n)} - \alpha_{t_{i,l_1}}^{(l_1,n)})\,\dd \alpha_u^{(l_2,n)} g_s^{(l_1,n)}g_s^{(l_2,n)}(b_{n,s}^2 - b_{n,t_{i,l_1}}^2)\,\dd s}_1\\
& \qquad \qquad \leq \frac{a_{+}^2}{\sqrt{2}} \int_{t_{i,l_1}}^{t_{i,l_2}} (s - t_{i,l_1}) \,\dd s\,\sup_{t_{i,l_1} \leq s \leq t_{i,l_2}}\norm{b_{n,s}^2 - b_{n,t_{i,l_1}}^2}_2\\
& \qquad \qquad = \frac{a_{+}^2}{2^{3/2}} (t_{i,l_2} - t_{i,l_1})^2 \,\sup_{t_{i,l_1} \leq s \leq t_{i,l_2}}\norm{b_{n,s}^2 - b_{n,t_{i,l_1}}^2}_2,
\end{split}
\notag
\end{equation}
from which
\begin{equation}
\begin{split}
& \norm{\sum_{i\,:\, t_{i+1,l_1}\leq \tau_n}\int_{t_{i,l_1}}^{t_{i,l_2}}\int_{t_{i,l_1}}^{s}(\alpha_u^{(l_1,n)} - \alpha_{t_{i,l_1}}^{(l_1,n)})\,\dd \alpha_u^{(l_2,n)} g_s^{(l_1,n)}g_s^{(l_2,n)}(b_{n,s}^2 - b_{n,t_{i,l_1}}^2)\,\dd s}_1\\
& \qquad \qquad \leq \sum_{i\,:\, t_{i+1,l_1}\leq \tau_n}\int_{t_{i,l_1}}^{t_{i,l_2}}
\norm{\int_{t_{i,l_1}}^{s}(\alpha_u^{(l_1,n)} - \alpha_{t_{i,l_1}}^{(l_1,n)})\,\dd \alpha_u^{(l_2,n)} g_s^{(l_1,n)}g_s^{(l_2,n)}(b_{n,s}^2 - b_{n,t_{i,l_1}}^2)\,\dd s}_1\\
& \qquad \qquad \leq \frac{a_{+}^2}{2^{3/2}} \sum_{i\,:\, t_{i+1,l_1}\leq \tau_n}(t_{i,l_2} - t_{i,l_1})^2 \,\sup_{t_{i,l_1} \leq s \leq t_{i,l_2}}\norm{b_{n,s}^2 - b_{n,t_{i,l_1}}^2}_2\\
& \qquad \qquad \leq \frac{a_{+}^2}{2^{3/2}} \big(\sum_{i\,:\, t_{i+1,l_1}\leq \tau_n}(t_{i,l_2} - t_{i,l_1})^2 \big)\,\sup_{0 \leq |t - s| \leq K_n\Delta_n}\norm{b_{n,t}^2 - b_{n,s}^2}_2.
\end{split}
\notag
\end{equation}
By Markov{'}s inequality 
\begin{equation}
\begin{split}
& P\big\{\big| \sum_{i\,:\, t_{i+1,l_1}\leq T}\int_{t_{i,l_1}}^{t_{i,l_2}}\int_{t_{i,l_1}}^{s}(\alpha_u^{(l_1,n)} - \alpha_{t_{i,l_1}}^{(l_1,n)})\,\dd \alpha_u^{(l_2,n)} g_s^{(l_1,n)}g_s^{(l_2,n)}(b_{n,s}^2 - b_{n,t_{i,l_1}}^2)\,\dd s \big| \geq \varepsilon\big\}  \\
& \qquad \qquad \quad \leq \frac{1}{\varepsilon}\frac{a_{+}^2}{2^{3/2}} \big(\sum_{i\,:\, t_{i+1,l_1}\leq \tau_n}(t_{i,l_2} - t_{i,l_1})^2 \big)\,\sup_{0 \leq |t - s| \leq K_n\Delta_n}\norm{b_{n,t}^2 - b_{n,s}^2}_2 
+ P(\tau_n \neq T).
\end{split}
\notag
\end{equation}
By Assumption~\ref{assumption3}, i.e.~mean square continuity of $b_{n,s}^2$, we get
\begin{equation}
\sum_{i\,:\, t_{i+1,l_1}\leq T}\int_{t_{i,l_1}}^{t_{i,l_2}}\int_{t_{i,l_1}}^{s}(\alpha_u^{(l_1,n)} - \alpha_{t_{i,l_1}}^{(l_1,n)})\,\dd \alpha_u^{(l_2,n)} g_s^{(l_1,n)}g_s^{(l_2,n)}(b_{n,s}^2 - b_{n,t_{i,l_1}}^2)\,\dd s = o_p(K\Delta_n).
\notag
\end{equation}
This completes the proof of~\eqref{eq::part1}, and obviously, the same holds for $\langle Z_{n,l_1},Z_{n,l_2}\rangle_{T}^{(2)}$. 

We must now show that similar results apply to $\langle Z_{n,l_1},Z_{n,l_2}\rangle_{T}^{(3)}$ and $\langle Z_{n,l_1},Z_{n,l_2}\rangle_{T}^{(4)}$. It suffices to look at $\langle Z_{n,l_1},Z_{n,l_2}\rangle_{T}^{(3)}$. In analogy with~\eqref{eq::key_decomp}, we can write 
\begin{equation}
\begin{split}
\langle Z_{n,l_1},Z_{n,l_2}\rangle_{T}^{(3)} 
& = \int_0^{T}  (\alpha_s^{(l_1,n)} - \alpha_{t_{*,l_1}}^{(l_1,n)})(\beta_s^{(l_2,n)} - \beta_{t_{*,l_2}}^{(l_2,n)})\,\dd\langle\beta^{(l_1,n)},\alpha^{(l_2,n)} \rangle_{s}\\
& = \sum_{i \,:\, t_{i+1,l_1}\leq T}\big\{ \int_{t_{i,l_1}}^{t_{i,l_2}}
(\alpha_s^{(l_1,n)} - \alpha_{t_{i,l_1}}^{(l_1,n)})(\beta_s^{(l_2,n)} - \beta_{t_{i-1,l_2}}^{(l_2,n)})\,\dd\langle\beta^{(l_1,n)},\alpha^{(l_2,n)} \rangle_{s}\\
& \qquad \qquad \qquad  + \int_{\lambda_{i}}^{t_{i+1,l_1}}
(\alpha_s^{(l_1,n)} - \alpha_{t_{i,l_1}}^{(l_1,n)})(\beta_s^{(l_2,n)} - \beta_{t_{i,l_2}}^{(l_2,n)})\,\dd\langle\beta^{(l_1,n)},\alpha^{(l_2,n)} \rangle_{s}\big\}\\
& \qquad \qquad 
+ \int_{t_{*,l_1}(T)}^T (\alpha_s^{(l_1,n)} - \alpha_{t_{*,l_1}(s)}^{(l_1,n)})(\beta^{(l_2,n)} - \beta_{t_{*,l_2}(s)}^{(l_2,n)})\,\dd\langle\beta^{(l_1,n)},\alpha^{(l_2,n)} \rangle_{s},
\end{split}
\notag
\end{equation}
and we will have to deal with error terms of the form
\begin{equation}
\begin{split}
& \int_{t_{i,l_1}}^{t_{i,l_2}}\int_{t_{i,l_1}}^{s}(\alpha_u^{(l_1,n)} - \alpha_{t_{i,l_1}}^{(l_1,n)})\,\dd \beta_{u}^{(l_2,n)}\,\dd\langle \beta^{(l_1,n)},\alpha^{(l_2,n)}\rangle_s\\
& \qquad \qquad = \int_{t_{i,l_1}}^{t_{i,l_2}}\int_{t_{i,l_1}}^{s}(\alpha_u^{(l_1,n)} - \alpha_{t_{i,l_1}}^{(l_1,n)})\,\dd \beta_{u}^{(l_2,n)} g_s^{(l_1,n)}f_s^{(l_2,n)} c_{n,s}\,\dd s\\
& \qquad \qquad = \int_{t_{i,l_1}}^{t_{i,l_2}}\int_{t_{i,l_1}}^{s}(\alpha_u^{(l_1,n)} - \alpha_{t_{i,l_1}}^{(l_1,n)}) \,\dd \beta_{u}^{(l_2,n)} g_s^{(l_1,n)}f_s^{(l_2,n)} c_{n,t_{i,l_1}}\,\dd s\\
& \qquad \qquad \qquad \qquad 
\int_{t_{i,l_1}}^{t_{i,l_2}}\int_{t_{i,l_1}}^{s}(\alpha_u^{(l_1,n)} - \alpha_{t_{i,l_1}}^{(l_1,n)})\,\dd \beta_{u}^{(l_2,n)} g_s^{(l_1,n)}f_s^{(l_2,n)} (c_{n,s} - c_{n,t_{i,l_1}})\,\dd s.
\end{split}
\notag
\end{equation}
Define the functions $H_t^{(l_1,l_2)} = \int_0^tg_s^{(l_1,n)}f_s^{(l_2,n)}\,\dd s$, and note that $|H_t^{(l_1,l_2)} - H_s^{(l_1,l_2)}| \leq |t - s|$ for all $t,s$. Then for $t_{i,l_2} < \tau_n$, 
\begin{equation}
\begin{split}
& \int_{t_{i,l_1}}^{t_{i,l_2}}\int_{t_{i,l_1}}^{s}(\alpha_u^{(l_1,n)} - \alpha_{t_{i,l_1}}^{(l_1,n)})\,\dd \beta_{u}^{(l_2,n)} g_s^{(l_1,n)}f_s^{(l_2,n)} c_{n,t_{i,l_1}}\,\dd s\\
& \qquad \qquad \qquad \qquad = \int_{t_{i,l_1}}^{t_{i,l_2}}(H_{t_{i,l_2}}^{(l_1,l_2)} - H_{s}^{(l_1,l_2)})
(\alpha_s^{(l_1,n)} - \alpha_{t_{i,l_1}}^{(l_1,n)})\,\dd\beta_s^{(l_2,n)}.
\end{split}
\notag
\end{equation}
Analogous to above, this gives that 
\begin{equation}
\begin{split}
& \E\,\big(\sum_{i\,:\, t_{i+1,l_1}\leq \tau_n}\int_{t_{i,l_1}}^{t_{i,l_2}}\int_{t_{i,l_1}}^{s}(\alpha_u^{(l_1,n)} - \alpha_{t_{i,l_1}}^{(l_1,n)}) \dd \beta_{u}^{(l_2,n)} g_s^{(l_1,n)}f_s^{(l_2,n)} c_{n,t_{i,l_1}}\,\dd s\big)^2\\
& \qquad \qquad \qquad \qquad \qquad \qquad \qquad \qquad \leq \frac{a_{+}^2b_{+}^2|c_{+}|}{12} \sum_{i\,:\,t_{i+1,l_1}\leq \tau_n}(t_{i,l_2} - t_{i,l_1})^4.
\end{split}
\notag
\end{equation}
Looking back at the derivations above, we also see that 
\begin{equation}
\begin{split}
& \norm{\sum_{i\,:\, t_{i+1,l_1}\leq \tau_n}\int_{t_{i,l_1}}^{t_{i,l_2}}\int_{t_{i,l_1}}^{s}(\alpha_u^{(l_1,n)} - \alpha_{t_{i,l_1}}^{(l_1,n)})\,\dd \beta_{u}^{(l_2,n)} g_s^{(l_1,n)}f_s^{(l_2,n)} (c_{n,s} - c_{n,t_{i,l_1}})\,\dd s}_1\\
& \qquad \qquad \qquad \qquad \leq \frac{a_{+}b_{+}}{2^{3/2}} \big(\sum_{i\,:\, t_{i+1,l_1}\leq \tau_n}(t_{i,l_2} - t_{i,l_1})^2 \big)\,\sup_{0 \leq |t - s| \leq K_n\Delta_n}\norm{c_{n,t} - c_{n,s}}_2.
\end{split}
\notag
\end{equation}
By the same localisation techniques used previously, this establishes that 
(cf.~\eqref{eq::cross_angles}),
\begin{equation}
\begin{split}
& \langle Z_{n,l_1},Z_{n,l_2}\rangle_{T}^{(3)} + \langle Z_{n,l_1},Z_{n,l_2}\rangle_{T}^{(4)}\\ 
& \; = 
\int_0^{T}  (\langle\alpha^{(l_1,n)},\beta_s^{(l_2,n)} \rangle_s - 
\langle\alpha^{(l_1,n)},\beta_s^{(l_2,n)} \rangle_{t_{*,l_1}\vee t_{*,l_1}})
g_s^{(l_1,n)}f_s^{(l_2,n)}\,\dd\langle\beta^{(n)},\alpha^{(n)} \rangle_{s}\\
& \; + \int_0^{T}  (\langle\beta^{(l_1,n)},\alpha_s^{(l_2,n)} \rangle_s - 
\langle\beta^{(l_1,n)},\alpha_s^{(l_2,n)} \rangle_{t_{*,l_1}\vee t_{*,l_1}})
f_s^{(l_1,n)}g_s^{(l_2,n)}\,\dd\langle\alpha^{(n)},\beta^{(n)} \rangle_{s} + o_p(K_n\Delta_n).
\end{split}
\notag
\end{equation} 
Hence,
\begin{equation}
\begin{split}
&\langle Z_n,Z_n\rangle_{T}  = \frac{1}{4K^2}\sum_{l_1=1}^{2K}\sum_{l_2=1}^{2K} \langle Z_{n,l_1},Z_{n,l_2}\rangle_{T}\\
&  = \frac{1}{4K^2}\sum_{l_1=1}^{2K}\sum_{l_2=1}^{2K}\bigg\{\int_0^{T}
\int_{ t_{*,l_1}\vee t_{*,l_2} }^{s}\,\dd \langle \alpha^{(n)},\alpha^{(n)} \rangle_{u}\,f_u^{(l_1,n)}f_u^{(l_2,n)} g_s^{(l_1,n)}g_s^{(l_2,n)}\,\dd\langle\beta^{(n)},\beta^{(n)} \rangle_{s}\\
& \qquad + \int_0^{T} \int_{ t_{*,l_1}\vee t_{*,l_2} }^{s}\,\dd \langle \beta^{(n)},\beta^{(n)} \rangle_{u}\,g_u^{(l_1,n)}g_u^{(l_2,n)}f_s^{(l_1,n)}f_s^{(l_2,n)}\,\dd\langle\alpha^{(n)},\alpha^{(n)} \rangle_{s}\\
& \qquad + \int_0^{T}
\int_{ t_{*,l_1}\vee t_{*,l_2} }^{s}\,\dd \langle \alpha^{(n)},\beta^{(n)} \rangle_{u}\, g_u^{(l_2,n)}f_u^{(l_1,n)} g_{s}^{(l_1,n)}f_{s}^{(l_2,n)}\,\dd\langle\beta^{(n)},\alpha^{(n)} \rangle_{s}\\
& \qquad + \int_0^{T}
\int_{ t_{*,l_1}\vee t_{*,l_2} }^{s}\,\dd \langle \beta^{(n)},\alpha^{(n)} \rangle_{u}\, f_u^{(l_2,n)}g_u^{(l_1,n)}f_{s}^{(l_1,n)}g_{s}^{(l_2,n)}\,\dd\langle\alpha^{(n)},\beta^{(n)} \rangle_{s}\bigg\} + o_p(K_n\Delta t)\\
&  = \frac{1}{4K^2}\sum_{l_1=1}^{2K}\sum_{l_2=1}^{2K}\int_0^{T}
\int_{ t_{*,l_1}\vee t_{*,l_2} }^{s}\bigg\{\,f_u^{(l_1,n)}f_u^{(l_2,n)}g_s^{(l_1,n)}g_s^{(l_2,n)}\,\dd\langle \alpha^{(n)},\alpha^{(n)}\rangle_u \frac{\dd\langle \beta^{(n)},\beta^{(n)}\rangle_s}{\dd s}\\
& \qquad\qquad \quad  + g_u^{(l_1,n)}g_u^{(l_2,n)}f_s^{(l_1,n)}f_s^{(l_2,n)}\,\dd\langle \beta^{(n)},\beta^{(n)}\rangle_u \frac{\dd\langle \alpha^{(n)},\alpha^{(n)}\rangle_s}{\dd s}\\
& \qquad\qquad \quad  + (f_u^{(l_1,n)}g_u^{(l_2,n)}) (g_{s}^{(l_1,n)}f_{s}^{(l_2,n)})\,\dd \langle \alpha^{(n)},\beta^{(n)} \rangle_{u} \,\frac{\dd\langle\beta^{(n)},\alpha^{(n)} \rangle_{s}}{\dd s}\\
& \qquad\qquad \quad + (f_u^{(l_2,n)}g_u^{(l_1,n)})(f_{s}^{(l_1,n)}g_{s}^{(l_2,n)})\,\dd \langle \beta^{(n)},\alpha^{(n)} \rangle_{u}\,\frac{\dd\langle\alpha^{(n)},\beta^{(n)} \rangle_{s}}{\dd s}\bigg\}\,\dd s + o_p(K_n\Delta t)\\
& = (K_n\Delta t) \int_0^{T}\kappa_s^{(n)}\,\dd s + o_p(K_n\Delta t). 
\end{split}
\notag
\end{equation}
By assumption, there is a $\Falg$-measurable process $\kappa_s$ such that $\int_0^{t}\kappa_s^{(n)}\,\dd s \to_p \int_0^t \kappa_s\,\dd s$ for all $t$ as $n$ tends to infinity, so this shows that $(K\Delta_n)^{-1/2} Z_n(t)$ satisfies Condition~(i) of Theorem~\ref{theorem::th2.28_general}.

We now turn to Condition~(ii), that is the Lindeberg condition, of said theorem. To verify that this condition holds, we appeal to Condition~(ii)$^{\prime\prime}$ of Corollary~\ref{lemma::condition_ii}. We must verify that the sequence $(K_n\Delta_n)^{-1/2}Z_n$ is P-UT, that $\sup_{0 \leq t \leq T}\abs{\Delta Z_{n}(t)} = o_p(K\Delta_n)$ as $n \to \infty$, and that $\sup_{0 \leq t \leq T}\E\, (K\Delta_n)^{-1}(\Delta Z_{n}(t))^2 < \infty$ for all $n$. We have seen that $(K\Delta_n)^{-1}\langle Z_n,Z_n\rangle_t$ converges in probability, hence also in distribution, to the continuous and increasing process $\int_0^t \kappa_s\,\dd s$. By \citet[Theorem~VI.3.37, p.~354]{jacod2003limit} this yields process convergence of $(K\Delta_n)^{-1}\langle Z_n,Z_n\rangle$ to $\int_0^{\cdot} \kappa_s\,\dd s$, which means that $(K\Delta_n)^{-1}\langle Z_n,Z_n\rangle = O_p(1)$ in the sense of Definition~\ref{def:order-in-prob-b}. To see that $(K\Delta_n)^{-1/2}Z_n$ is P-UT, let $H^n$ be any predictable process with $\abs{H_t^n} \leq 1$, and let $H^n\cdot Z_n(t)$ be the elementary stochastic integral (see \citet[p.~377]{jacod2003limit} for both definitions). Now, $\E\, (H^n\cdot Z_n(t))^2 = \E\, (H^n\cdot Z_n(t))^2 = \E\,(H^n)^2\cdot[Z_n,Z_n]_t \leq \E\, \langle Z_n,Z_n\rangle_t$, so by Lenglart{'}s inequality, for any $ \eps,\eta > 0$, 
\begin{equation}
\begin{split}
P((K\Delta_n)^{-1/2}\abs{H^n\cdot Z_n(t)} \geq \eps) & \leq 
P((K\Delta_n)^{-1}\sup_{t \leq T}\abs{H^n\cdot Z_n(t)}^2 \geq \eps^2)\\ 
& \leq \eta/\eps^2 
+ P((K\Delta_n)^{-1}\langle Z_n,Z_n\rangle_T \geq \eta),
\end{split}
\notag
\end{equation}    
and that $(K\Delta_n)^{-1/2}Z_n$ is P-UT follows from the definition \citep[Definition~VI.6.1, p.~377]{jacod2003limit}, because $(K\Delta_n)^{-1}\langle Z_n,Z_n\rangle_T$ is tight. 

The jumps of $Z_n(t)$ are 
\begin{equation}
\Delta Z_n(t) = \frac{1}{2K}\sum_{l=1}^{2K}\Delta Z_{n,l}(t).
\notag
\end{equation}
Using~\eqref{eq::errorMG1} we see that  
\begin{equation}
\begin{split}
\Delta Z_{n,l}(t)  & = (\alpha_t^{(l,n)} - \alpha_{t_{*,l}}^{(l,n)})\Delta \beta_t^{(l,n)}
+ (\beta_t^{(l,n)}- \beta_{t_{*,l}}^{(l,n)})\Delta \alpha_t^{(l,n)}\\
& = (\alpha_t^{(l,n)} - \alpha_{t_{*,l}}^{(l,n)})g_{t}^{(l,n)} \Delta\beta_t^{(n)}
+ (\beta_t^{(l,n)}- \beta_{t_{*,l}}^{(l,n)})f_{t}^{(l,n)}\Delta \alpha_t^{(n)}\\
& = \big(\int_{t_{*,l}}^{t}f_s^{(l,n)}\,\dd \alpha_s^{(n)}\big)\,g_{t}^{(l,n)} \Delta\beta_t^{(n)}
+ \big(\int_{t_{*,l}}^{t}g_s^{(l,n)}\,\dd \beta_s^{(n)}\big)\,f_{t}^{(l,n)}\Delta \alpha_t^{(n)}.
\end{split}
\label{eq::Z_nl_jump}
\end{equation}
For any $t \leq \tau_n$, where $\tau_n = \inf (t \colon \langle \alpha^{(n)},\alpha^{(n)}\rangle_t > a_{+}^2 t \;\text{or}\; \langle \beta^{(n)},\beta^{(n)}\rangle_t > b_{+}^2)$, we have that by the It{\^o} isometry, 
\begin{equation}
\begin{split}
\E\,\big(\int_{t_{*,l}}^{t}f_s^{(l,n)}\,\dd \alpha_s^{(n)}\big)^2 \leq \E\, \int_{t_{*,l}}^{t}(f_s^{(l,n)})^{2}\,\dd \langle\alpha^{(n)},\alpha^{(n)}\rangle_s \leq a_{+}^2 (t - t_{*,l}) \leq a_{+}^2 K\Delta_n, 
\end{split}
\notag
\end{equation}
which show that $\big(\int_{t_{*,l}}^{t}f_s^{(l,n)}\,\dd \alpha_s^{(n)}\big)^2$ is $L$-dominated \citep[Definition~I.3.29, p.~35]{jacod2003limit} by the predictable process $\int_{t_{*,l}}^{t}(f_s^{(l,n)})^{2}\,\dd \langle\alpha^{(n)},\alpha^{(n)}\rangle_s$, and that the latter is $O_p((K\Delta_n)^{1/2})$. Therefore, by Condition~\ref{cond:rate},
\begin{equation}
P(\sup_{t \leq T}\abs{\int_{t_{*,l}}^{t}f_s^{(l,n)}\,\dd \alpha_s^{(n)}} \geq \eps)
 \leq P(\sup_{t \leq \tau_n}\abs{\int_{t_{*,l}}^{t}f_s^{(l,n)}\,\dd \alpha_s^{(n)}} \geq \eps)
+ P(\tau_n < T),
\notag
\end{equation}
which shows that both $\sup_{t\leq T}\abs{\int_{t_{*,l}}^{t}f_s^{(l,n)}\,\dd \alpha_s^{(n)}}$ and $\sup_{t\leq T}\abs{\int_{t_{*,l}}^{t}g_s^{(l,n)}\,\dd \beta_s^{(n)}}$ are $O_p((K\Delta_n)^{1/2})$ as $n \to \infty$. But for any $\eps > 0$,
\begin{equation}
\sup_{t \leq T}\abs{\Delta \beta_t^{(n)}} \leq \sup_{t \leq T}\abs{\Delta \beta_t^{(n)}}I\{\abs{\Delta \beta_t^{(n)}} \geq \eps\} + \eps
\leq \int_{\abs{x} \geq \eps}\abs{x}\,\mu_{\beta}^n([0,t]\times \dd x) + \eps \to \eps,
\notag
\end{equation}
as $n \to \infty$ by the Lindeberg condition in \eqref{eq::spot_lindeberg}, combined with Lenglart{'}s inequality. But since $\eps > 0$ was arbitrary, $\sup_{t\leq T}\abs{\Delta \beta_t^{(n)}} = o(1)$, and we conclude that $\sup_{t\leq T}\abs{\Delta Z_{n,l}(t)} = o_p((K\Delta_n)^{1/2})$. For the last condition, by \citet[Theorem~I.4.47(c), p.~52]{jacod2003limit}, the triangle inequality, and using that $[Z_n,Z_n]_t^2$ is an increasing process, we have that for any $t$, 
\begin{equation}
\E\, (\Delta Z_n(t))^2 = \E\, \Delta [Z_n,Z_n]_t \leq 2\E\, [Z_n,Z_n]_T = 2\E\,\langle Z_n,Z_n\rangle_T,
\notag
\end{equation}
but in Appendix~\ref{app::conv.rates} we saw that $\E\,\langle Z_n,Z_n\rangle_T \lesssim 4TK\Delta_n$. Thus $\sup_{t\leq T}(K\Delta_n)^{-1}\E\, (\Delta Z_n(t))^2<\infty$ for all $n$, and we conclude that Condition~(ii)$^{\prime\prime}$ of Corollary~\ref{lemma::condition_ii} is satisfied, and therefore also the Lindeberg condition of Theorem~\ref{theorem::th2.28_general}.

It remains to check Condition~(iii) of Theorem~\ref{theorem::th2.28_general}, namely that $(K\Delta_n)^{-1/2}\langle Z_n,X^n\rangle_t \overset{p}\to 0$ for each $t \in [0,T]$,
where $X^n$ is a sequence of bounded martingales. It is enough to check this condition for a sequence of processes $X^n$ that is either a sequence of Wiener processes, or a sequence of Poisson processes (this is a consequence of the representation theorem in \citet[Theorem~14.5.7, p.~360]{cohen2015stochastic}. This means that the the sequence $X^n$ has predictable quadratic variation $\langle X^n,X^n \rangle_t = t$ or $\langle X^n,X^n \rangle_t = \int_0^t \lambda_s \,\dd s$ for some deterministic function $\lambda$. For simplicity of notation we assume that $X^n$ is a sequence of Wiener processes. By the Kunita--Watanabe inequality, for $h > 0$ and $t+h \leq \tau_n$,
\begin{equation}
\begin{split}
|\langle\beta^{(l,n)}(h),X^n\rangle_{t+h} - \langle\beta^{(l,n)}(h),X^n\rangle_{t}|
& \leq \big( \int_{t}^{t+h} (g_s^{(l,n)})^2b_{n,s}^2\,\dd s \,\langle X^n,X^n\rangle_{(t,t+h]} 	\big)^{1/2}\\
& \leq \big( \int_{t}^{t+h} b_n(s)^2 \,\dd s \,h\big)^{1/2} \leq b_{+}  h. 
\end{split}
\notag
\end{equation}
Thus, $\dd |\langle \beta^{(l,n)}(h),X^n \rangle_t|/\dd t \leq b_{+}$, where 
\begin{equation}
|\langle \beta^{(l,n)}(h),X^n \rangle_t| = \langle \beta^{(l,n)}(h),X^n \rangle_t^{+} 
+ \langle \beta^{(l,n)}(h),X^n \rangle_t^{-},
\notag
\end{equation} 
denotes the positive plus the negative part of the function. For a fixed $l$ and $t_{i+K} \leq \tau_n$ (here $\alpha^{(l,n)} = \alpha^{(l,n)}(h)$, $\beta^{(l,n)} = \beta^{(l,n)}(h)$, etc.) 
\begin{equation}
\begin{split}
\norm{\int_{t_{i-K}}^{{t_{i+K}}} (\alpha^{(l,n)}_s 
- \alpha^{(l,n)}_{t_{i-K}})\,\dd \langle\beta^{(l,n)},X^n\rangle_s}_1
& \leq \E\,\int_{t_{i-K}}^{{t_{i+K}}} |(\alpha^{(l,n)}_s 
- \alpha^{(l,n)}_{t_{i-K}})|\,\dd |\langle \beta^{(l,n)},X^n\rangle_s|\\
& \leq b_{+}\,\int_{t_{i-K}}^{{t_{i+K}}} \norm{(\alpha^{(l,n)}_s 
- \alpha^{(l,n)}_{t_{i-K}})}_1\,\dd s\\
& \leq b_{+}\,\int_{t_{i-K}}^{{t_{i+K}}} \norm{(\alpha^{(l,n)}_s 
- \alpha^{(l,n)}_{t_{i-K}})}_2\,\dd s\\
& \leq a_{+}b_{+}\,\int_{t_{i-K}}^{{t_{i+K}}} (s - t_{i-K})\,\dd s
= \frac{a_{+}b_{+}}{2} (t_{i+K} - t_{i-K})^2, 
\end{split}
\notag
\end{equation}
where for the third inequality we have used H{\"o}lder{'}s inequality. Then, for $t < \tau_n$,
\begin{equation}
\begin{split}
\norm{\langle Z_{n,l},X^n\rangle_t}_1 & 
\leq \sum_{t_{i+K}\leq t,\,i\equiv l[2K]}\norm{\int_{t_{i-K}}^{{t_{i+K}}} (\alpha^{(l,n)}_s(h) 
- \alpha^{(l,n)}_{t_{i-K}}(h))\,\dd \langle \beta^{(l,n)}(h),N^n\rangle_s}_1[2]\\
& \qquad \qquad \qquad \qquad + \norm{\int_{t_{*,l}}^{t} (\alpha^{(l,n)}_s(h) 
- \alpha^{(l,n)}_{t_{*,l}}(h))\,\dd \langle \beta^{(l,n)}(h),N^n\rangle_s}_1[2]\\
& \leq a_{+}b_{+} \sum_{t_{i+K}\leq t,\,i\equiv l[2K]} (t_{i+K} -t_{i-K})^2 
+ a_{+}b_{+}(t - t_{*,l})^2 = O(K\Delta_n),
\end{split}
\notag
\end{equation}
hence $\langle Z_{n},X^n\rangle_t = (2K)^{-1}\sum_{l=1}^{2K} \langle Z_{n,l}(h),X^n\rangle_t =  o_p((K\Delta_n)^{1/2})$ for each $t \in [0,\tau_n]$, and the third requirement of Theorem~\ref{theorem::th2.28_general} follows from Condition~\ref{cond:rate}. 

We have now shown that the martingale sequence $(K\Delta_n)^{-1/2}Z_{n}$ converges stably in law to a $\Falg$-conditional Gaussian martingale with variance process $\int_0^t \kappa_s\,\dd s$. This proves the theorem.

\bibliography{highfreq_etcRefs}
\bibliographystyle{ecta}

\end{document}